%% file: FINAL_TICKET_5_4_15.tex
\documentclass[10pt]{article}
\usepackage{epsfig, amsfonts}
\usepackage{amsmath, amsthm}
\usepackage{caption}

\usepackage{amssymb}
\usepackage{latexsym}
\usepackage{graphicx}
\usepackage{bm}
\usepackage{enumerate}
\usepackage{pstricks}
\usepackage{pst-plot}
\usepackage{color}

\usepackage[titletoc]{appendix}
\usepackage[pdftitle={Critical Ticket},
  pdfauthor={Jennings, Pender},
  pdffitwindow=true,
  breaklinks=true,
  colorlinks=true,
  urlcolor=blue,
  linkcolor=red,
  citecolor=red,
  anchorcolor=red]{hyperref}
\usepackage[numbers]{natbib}

\newtheorem{theorem}{Theorem}[section]
\newtheorem{proposition}[theorem]{Proposition}
\newtheorem{lemma}[theorem]{Lemma}

\newtheorem{remark}[theorem]{Remark}

\newenvironment{proofof}[1]{\par {\sc Proof of #1.} \hskip 5pt}{\hfill\qed\par}

\catcode`\@=11
\title{Comparisons of Ticket  and Standard Queues}

\author{
        Otis B. Jennings
    \\
        otisbjennings@gmail.com
     \and
        Jamol Pender \\
        School of Operations Research and Information Engineering\\
        Cornell University \\
        jjp274@cornell.edu
     }
\begin{document}

\newcommand{\R}{\mathcal R}
\newcommand{\Rp}{\mathcal{R}_+}
\newcommand{\bD}{\bf D}
\newcommand{\posInt}{\mathcal N}
\newcommand{\Prob}{\mathbb P}
\newcommand{\Exp}{\mathbb E}

\maketitle
\begin{abstract}

	Upon arrival to a ticket queue, a customer is offered a slip of paper with a number on it -- indicating the order of arrival to the system -- and is told the number of the customer currently in service.
	The arriving customer then chooses whether to take the slip or balk,  a decision based on the perceived queue length and associated waiting time.
	Even after taking a ticket, a customer may abandon the queue, an event that will be unobservable until the abandoning customer would have begun service.
	In contrast, a standard queue has a physical waiting area so that abandonment is apparent immediately when it takes place and balking is based on the actual queue length at the time of arrival.
	

	We prove heavy traffic limit theorems for the generalized ticket and standard queueing processes, discovering that the processes converge together to the same limit,
		a regulated Ornstein-Uhlenbeck (ROU) process.
	One conclusion is that for a highly utilized service system with a relatively patient customer population, the ticket and standard queue performances are asymptotically indistinguishable
		on the scale typically uncovered under heavy traffic approaches.
	Next, we heuristically estimate several performance metrics of the ticket queue, some of which are of a sensitivity typically undetectable under diffusion scaling.
	The estimates are tested using simulation and are shown to be quite accurate under a general collection of parameter settings.

\end{abstract}

\section{Introduction}

In many service settings, newly arriving customers are given information about the number of individuals preceding them in line, even when this line is virtual.
There are several ways in which this information may be passed on.
For example, a customer visiting either a delicatessen or the department of motorized vehicles (DMV) is offered a ticket with a number on it and, via some physical display, is informed of the current customer being serviced.
In restaurants, dinner parties may either be told about the estimated wait or told how many similarly configured dinner parties are ahead of them.
Often these two forms of information are roughly interchangeable: given the service rate, knowledge of the queue length yields an estimate of the delay, and vice versa.
Being informed about the delay in service provision, customers then choose whether to join the queue or to balk.

The fact that a customer initially accepts the estimated delay and joins the queue does not guarantee that the customer will wait around until service can begin.
Customers may renege on their initial decision and abandon the queue.
In environments where customers are physically waiting in line for their service -- such as at a bank or grocery store -- abandonment is immediately apparent to service providers and other customers alike.
However, neither the delicatessen service personnel, fellow ticket holders, nor potential ticket holders are aware when someone has chosen to abandon their ticket.
This event is not discovered until the ticket's number is called and no one responds.
In general, the number of outstanding tickets may be larger than the number of customers actually waiting for service.

A {\it ticket queue} refers to the setup typically employed in delicatessens, but can be thought of more generally as a mechanism for tracking the number of {\it potential} customers yet to be served
	and for maintenance of a first-come-first-served protocol.
By potential it is meant that these customers have been triaged, joined the queue,  and have committed in principle to be served, yet may ultimately abandon before service can actually begin.
Some of the customers may renege on this implicit commitment and leave;  reneging customers typically will not inform the system manager of their decision to forgo their place in line.
As a result, what is generally perceived as the queue length -- the number of potential customers to be served --  will in fact be an upper bound.

A visual comparison of abandonment in the two queueing types is depicted in Figure \ref{fig:ticket}.
	\begin{figure}
		\begin{center}
			\includegraphics[width=5in]{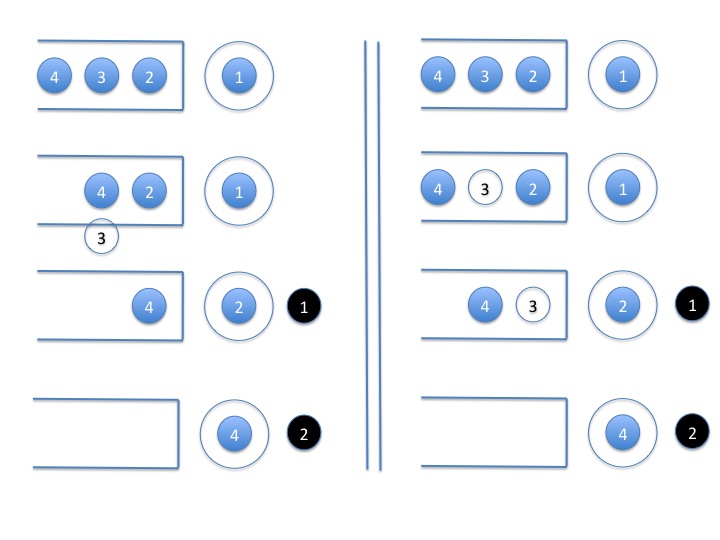}
		\end{center}
		\caption{Queueing dynamics of the ticket and standard queues when abandonment occurs.} \label{fig:ticket}
	\end{figure}
The left side of the figure shows a progression of three events as they are experienced in the standard queue.
The first image has a queue with four customers, numbered 1 through 4 and with customer \#1 currently in service.
In the next frame, customer \#3 abandons the queue, the system immediately detects the defection and customer \#4 replaces
	customer \#3 in line.
Then customer \#1 completes service.
Finally customer \#4 reaches the front of the queue and begins service in the 4th frame after the service completion for customer \#2.

The right side of the image depicts the same sequence of events but as they are experienced in a ticket queue.
Notice that when customer \#3 abandons his ticket, the system is unaware.
Thus, the perceived queue length (including the customer in service) is four.
After customer \#1 completes service, all customers advance and customer \#2 begins service.
After customer \#2's service is complete, an attempt is made to handle customer \#3, and it is determined immediately
	that customer \#3 has left.
So customer \#4 begins service at this time.
What we see here is that in two of the four frames, the ticket queue's perceived length is larger than that of the standard queue.

Several questions emerge:
What are the dynamics of the collection of outstanding tickets?
What is the actual number of customers remaining, i.e., those who have yet to renege?
What fraction of customers renege?
What fraction of customers balk?
Does the difference in implementation of ticket and standard queues lead to a marked difference in performance between the two customer organizing methods?

Some heuristics seeking to address these questions were provided by \cite{ref:XGO}.
Our paper revisits the ticket queue to provide new heuristics and intuition that is 
	complementary to that of \cite{ref:XGO}.
In particular, informed by a heavy traffic limit theorem, we conclude that for highly utilized systems with relatively patient customers,
	ticket queues and conventional queues are indistinguishable.
One of the conclusions of \cite{ref:XGO} is that for heavily loaded systems with relatively \textit{impatient} customers,
	the ticket queue experiences a higher percentage of balking.
In addition to these divergent insights, what sets this paper apart from  \cite{ref:XGO} is that it provides estimates of the distribution of the queue length process,
	which can then be used to estimate other stochastic elements.
More details about these estimates are provided below.
Moreover, general assumptions about the abandonment distribution and balking distributions are employed; random variables in \cite{ref:XGO}
	are all exponentially distributed.

The paper appeals to heavy traffic limit theory in the pursuit of its approximations and comparisons of the ticket and standard queues.
In this regard, it is similar in approach to \cite{ref:WG}, which studies a generalized diffusion-scaled single-server queue with abandonment {\it or} balking.
Our paper concerns a realistic situation where there is both balking  {\it and} abandonment.
Moreover, technical complications arise as we base balking on the queue that customers perceive, 
	whereas abandonment is based on the delays that customers experience.
Ultimately, what one obtains under diffusion scaling is a regulated (at zero) Ornstein-Uhlenbeck (ROU) process whose constant drift is related to the difference between the arrival and service rates
	and whose restorative drift involves the derivatives of the abandonment and balking distributions, 
	both evaluated at zero.
The diffusion limit of the critically loaded ticket queue is identical, despite the fact that for the ticket queue, abandonment of a ticket
	is not matched with a shortening of the queue until the ticket of the abandoning customer reaches the front of the queue.
Further, given the same primitive random variable elements, the diffusion-scaled ticket and standard queueing processes converge together to the same diffusion limit,
	a stronger notion of process similarity.
In other words, not only are the dynamics of the ticket and standard queues asymptotically similar, the processes are asymptotically coupled.

As mentioned above, we estimate the distribution of the queueing processes employing heuristic interpretations of the limit theorem.
Additionally, we estimate  performance metrics such as abandonment and balking probabilities and the expected number of abandoned tickets in circulation.
Each of these additional quantities is typically lost under conventional heavy traffic limit theory, but are manifest on a smaller -- i.e., more sensitive -- scaling.
The approach for estimating the expected number of abandoned tickets in circulation provides the additional insight needed to compare the subtle differences between the standard and ticket queues.

We should also mention that the differences between standard and ticket queues also highlights the different ways that delay information can be communicated to the customer.  This communication between the service and its customers is important because customers will make their decision to wait or to leave the queue based on the information that they receive from the manager of the system.  For example, see the following papers on research on queues with delay announcements and estimation \cite{ref:IW,ref:AB,ref:W,ref:W2,ref:ASW,ref:GZ,ref:MP2}.  

As for the assumptions used throughout, the paper also fits within the growing literature on queues with generally distributed abandonment distributions, specifically those that are not exponentially distributed.
Some recent examples of such papers include \cite{ref:ZM,ref:LW,ref:BH,ref:WG,ref:RW,ref:JR,ref:GRZ,ref:JPu}. 
The first two of these are in a multiserver setting, whereas the others are in the single server regime.
The last two papers use measure-valued processes to model system dynamics.

The remainder of the paper proceeds as follows.
Both the ticket queueing model and standard queueing model (with abandonment and balking) are presented 
	in the next section.
The main result, Theorem \ref{thm:main}, is presented in Section \ref{sec:construction}.
Section \ref{sec:approximations} contains the heavy traffic-inspired approximations for performance measures and some interpretations.
The proofs of the main results are provided in Section \ref{sec:proof}.
Extensive numerical results are presented in Section \ref{sec:numerics}.
Concluding remarks and extensions follow.

\subsection{Notation}
We conclude this introduction with notational conventions.
Let $\R$ denote the set of all reals, $\Rp$ the set of nonnegative reals, and $\posInt$ the set of 
	strictly positive integers.
For a Polish space $\mathcal{S}$, let $\bD(\Rp,\mathcal{S})$ denote the space of right continuous
	with left limits functions from $\Rp$ into $\mathcal{S}$.
The Polish spaces we consider here are $\Rp$ and $\Rp^2$.

For ease of exposition, quantities that are related to the ticket queue will be appended with the subscript $T$ and those associated with the standard queue will have the subscript $S$ .
We use the subscript $\alpha$ as a place holder for either ticket ($\alpha = T$) or standard ($\alpha = S$) queues.
The standard hazard rate function $h:\R \mapsto \Rp$ is the ratio of the density of the standard normal distribution 
	to the tail of the standard normal distribution: $h(x) = \phi(x) / (1-\Phi(x))$ for every $x \in \R$.
For enumeration purposes, we use the letters $i, j,$ and $k$ to represent nonnegative integers. 
The letters $q, r,s,$ and $t$ are used to represent time.
Typically, the symbols $\delta, \epsilon, \eta$ are used to represent small positive real numbers.
In contrast, the letters $K$ and $L$ are used to 
	represent large quantities, predominately as upper bounds.

\section{The model basics} \label{sec:model}

In this section, we provide the primitive random variables for the queueing processes, describe the construction of the ticket and standard queues, and discuss the intricacies of simulating the ticket queues.

\subsection{Random variables}

We construct the ticket and standard queueing processes using the same collection of random variables.
Each customer \textit{that arrives after time zero}
	has an interarrival time, (potential) service time, initial time tolerance and (potential) abandonment time.
For the $i^{th}$ customer, these times are captured in the quadruple $(u_i, v_i, b_i, d_i)$. 
The letters 'b' and 'd' denote balking and deadline, respectively.
The mutually independent  sequences $\{ u_i, i\geq 1\} $, $\{ v_i, i\geq 1\} $, $\{ b_i, i\geq 1\} $, $\{d_i, i\geq 1\} $,
	which are all defined on the same probability space $( \Omega, \mathcal{F}, \mathcal{P}) $, are each i.i.d.
The exogenous arrival rate of jobs is $\lambda$ and the service rate is $\mu$.
The sequences $\{ u_i, i\geq 1\} $ and $\{ v_i, i\geq 1\} $ have unitary means and the interarrival and service times of the $i_{th}$ customer are $u_i/\lambda$ and $v_i / \mu$, respectively.
The arrival time of the $i_{th}$ job occurs at time
$$
	t_i = (1/\lambda) \sum_{j=1}^i u_j.
$$
The unitary interarrival and service times have variances $\sigma^2_a$ and  $\sigma^2_s$, respectively.
The quantities $b_i$ and $d_i$ represent the random variables associated with balking and reneging, respectively.
Let $F_b$ and $F_d$ denote their respective cumulative distribution functions.
We assume the these functions both vanish at zero.
Moreover, their derivatives exist at zero and the sum of these derivatives is strictly positive.
We define the sum as
\begin{equation} \label{eq:theta}
	\theta
	\equiv
		F^\prime_b(0) + F^\prime_d(0).
\end{equation}

\subsection{Initial conditions} \label{sec:initialConditions}

Consider jobs that are present at time zero.
These jobs do not require arrival times.
Neither do they require balking random variables.
Hence, for these initial jobs, we provide 
	only two sequences of random variables: unitized potential service times, $\{ \hat{v}_i, i\geq 1\} $,
	and \textit{residual} deadline quantities,  $\{ \hat{d}_i, i\geq 1\} $.
Let $Q(0)$ denote the number of initial jobs.
If initial job $i \le Q(0)$ has not begun service before time $\hat{d}_i$, this job will abandon. 
If this initial job has not abandoned, then its service time will be $\hat{v}_i / \mu$.
The initial potential service times are i.i.d. and have the same distribution as the potential service times 
	of the jobs arriving after time zero.
The residual deadlines do not necessarily have the same distribution. 
Let $\hat{F}_{d,i}$ be the cumulative distribution of $\hat{d}_i$.
We impose the following uniform restriction on their distributions near zero:
There exists an $\hat{f} >0$ and an $h_0 > 0$ such that
$$
	\sup_i \frac{ \hat{F}_{d,i}(h)}{h} \le \hat{f}, \qquad \forall h \le h_0.
$$

The workload at time 0, denoted $W(0)$, is the amount of effort required to process jobs present at time zero.
Because of deadlines running out before service has begun, some of the $Q(0)$ jobs present at time zero 
	will not be served.
Therefore, it is not as simple as adding up the service times of the first $Q(0)$ jobs.
Instead, let $\hat{w}_i$ denote the cumulative amount of server effort required among the first $i$ jobs in-queue 
	at time zero.
The $i^{th}$ job will be served if and only if it is sufficiently patient, or if $\hat{d}_i > \hat{w}_{i-1}$, where $\hat{w}_0 =  0$.
We can define the $\hat{w}_i$'s recursively:
$$
	\hat{w}_i 
		= \sum_{j=1}^i \frac{\hat{v}_j}{\mu} \cdot 1(\hat{d}_j > \hat{w}_{j-1})
		= \hat{w}_{i-1} + \frac{\hat{v}_i}{\mu} \cdot 1(\hat{d}_i > \hat{w}_{i-1}),
		 \qquad i \ge 1.
$$
It follows that $W(0) = \hat{w}_{Q(0)}$.

\subsection{The queueing processes}

Let the process $Q_T= \{Q_T(t), t \ge 0\}$ track the dynamics of the ticket queue and $Q_S= \{Q_S(t), t \ge 0\}$ track that of the standard queue.
The difference between the ticket and standard queue is the timing of when the abandonment is accounted for.
Otherwise the system dynamics are identical.
For example, customers arrive to the ticket and standard queue in the same sequence, with the same balking and reneging tendencies, and with the same service requirements.
The common arrival process is $A = \{A(t), t \ge 0\}$.

All other processes are indexed by the abandonment protocol $\alpha \in \{S,T\}$.
In particular, $B_\alpha = \{B_\alpha(t), t \ge0\}$  is the balking process
	that tracks as a function of time the number of arriving customers who leave immediately upon arrival.
The reneging process $R_\alpha = \{R_\alpha(t), t \ge0\}$ tracks the number of jobs 
	\textit{that arrive after time zero}, who have abandoned,
	and whose abandonment has been detected in the system.
Recall that in the standard queue an abandoned customer is immediately detected once they leave, whereas for the ticket queue
	abandonment is only apparent at the moment at which service for that customer would have begun.
The process $S_\alpha = \{S_\alpha(t), t \ge0\}$ tracks the number of service completions, 
	\textit{of jobs that arrived after time zero},
	as a function of how
	much effort the server has expended.
Relatedly, the busy time process $T_\alpha = \{T_\alpha(t), t \ge0\}$ reports how much time has been spent processing
	jobs as a function of time $t$, including initial jobs.
The idle time process $I_\alpha = \{I_\alpha(t), t \ge0\}$  is complementary to $T_\alpha$: $I_\alpha(t) = t - T_\alpha(t)$ for each $t \ge 0$.
Lastly, the workload process $W_\alpha = \{W_\alpha(t), t \ge0\}$ reports as function of time the amount of effort required by the
	server to process those customers currently in-queue that will not abandon.
	
The system must clear out all initial jobs in-queue before it can start processing jobs that arrive after time zero.
Let $\hat{Q}_\alpha = \{\hat{Q}_\alpha(t), t \ge 0\}$ track the number of remaining initial jobs at time $t$.
Some of these jobs may abandon.
Let $\hat{R}_\alpha = \{\hat{R}_\alpha(t), t \ge0\}$ track, as a function of time, the number of jobs 
	\textit{who arrive before time zero}, who have abandoned,
	and whose abandonment has been detected in the system.
The standard and ticket queues start with the same collection of jobs at time zero. 
Hence, among these initial jobs, the same subset of jobs abandon both the standard and ticket queues.
What is different about the abandonment processes is the timing of when the abandonment is detected.
What is identical between the ticket and standard queues is the initial job service completion process:
$\hat{S} = \{\hat{S}(t), t \ge0\}$, which is defined in the next section.

Equations governing the ticket and standard queue are below.
For each $t \ge 0$ and  $\alpha \in \{S,T\},$
\begin{equation} \label{eq:mod:queue}
	Q_\alpha(t)
	=
		\hat{Q}_\alpha(t)
		+
		A(t)
		-
		B_\alpha(t)
		-
		R_\alpha(t)
		-
		S_\alpha( (T_\alpha(t) - W(0))^+ ), 
\end{equation}
\begin{equation} \label{eq:mod:hatQueue}
	\hat{Q}_\alpha(t)
	=
		Q(0)
		-
		\hat{R}_\alpha(t)
		-
		\hat{S}(t), 
\end{equation}
\begin{equation} \label{eq:mod:arrivals}
	A(t) = \sup \left\{j \ge 0: \sum_{i=1}^j u_i/\lambda \le t \right \}
\end{equation}
\begin{equation} \label{eq:mod:balking}
	B_\alpha(t)
	=
		\sum_{i=1}^{A(t)} 1(b_i \le Q_\alpha(t_i-) / \mu),
\end{equation}
\begin{equation} \label{eq:mod:renegingT}
	R_T(t)
	=
		\sum_{i=1}^{A(t)}
			1(b_i > Q_T(t_i-) / \mu)
			\cdot
			1(d_i \le W_T(t_i-))
			\cdot
			1(W_T(t_i-) \le t - t_i), 
\end{equation}
\begin{equation} \label{eq:mod:renegingHatT}
	\hat{R}_T(t)
	=
		\sum_{i=1}^{Q(0)}
			1(\hat{d}_i \le \hat{w}_{i-1})
			\cdot
			1(\hat{w}_{i-1} \le t ), 
\end{equation}
\begin{equation} \label{eq:mod:renegingS}
	R_S(t)
	=
		\sum_{i=1}^{A(t)}
			1(b_i > Q_S(t_i-) / \mu)
			\cdot
			1(d_i \le \min(W_S(t_i-), t - t_i)), 
\end{equation}
\begin{equation} \label{eq:mod:renegingHatS}
	\hat{R}_S(t)
	=
		\sum_{i=1}^{Q(0)}
			1(\hat{d}_i \le \min(\hat{w}_{i-1}, t)), 
\end{equation}
\begin{equation} \label{eq:mod:busytime}
	T_\alpha(t) = \int_0^t 1(Q_\alpha(s) > 0) ds, 
\end{equation}
\begin{equation} \label{eq:mod:idletime}
	I_\alpha(t) = t - T_\alpha(t),
\end{equation}
and
\begin{equation} \label{eq:mod:workload}
	W_\alpha(t)
	=
		W(0)
		-
		T_\alpha(t)
		+
		\sum_{i=1}^{A(t)}
			(v_i/ \mu)
			\cdot
			1(b_i > Q_\alpha(t_i-)/ \mu)
			\cdot
			1(d_i > W_\alpha(t_i-)).
\end{equation}

Interpreting \eqref{eq:mod:queue} and \eqref{eq:mod:hatQueue}, 
	the queue length process 
	consists of the remaining initial jobs and
	may increase with each arrival, provided the corresponding customer does not balk.
In addition to the departure of the initial jobs, 
	the queue length process decreases whenever there is an abandonment or a service completion
	among the jobs that arrive after time zero.
Initially, service is allocated entirely to the initial jobs and remains so until, $W(0)$, when those jobs have departed entirely.
At this point service allocation is given entirely to jobs arriving after time zero.
The initial jobs experience only abandonment and service completion. The remaining job process is a decreasing function,
	hits zero, and remains there.

As for the balking process and \eqref{eq:mod:balking},
	an arriving customer joins if their initial delay threshold $b_i$ is sufficiently large.
When a customer arrives to the system, it is first triaged and given the opportunity to join the queue.
Under the ticket queue implementation, joining the queue involves the acceptance of the offered numbered ticket.
In addition to knowing which ticket it will be given, the customer is told the number of the ticket holder currently in service.
For the standard queue, joining the queue involves standing in a physical line.
In both cases, the decision of whether or not to join the queue is based on the customer's expectation of the delay until service and 	her tolerance for such a delay.
Customers convert the queue information into an expectation of delay until service.
We assume the conversion is naive and the same for both ticket and standard queueing environments:
Given the queue length, the customer estimates the delay by dividing the queue length by the service rate $\mu$, 
	a quantity assumed to be known by all customers.
That is, if the $i_{th}$ customer arrives at time $t$, the customer joins the queue if $b_i > Q_\alpha(t-)/\mu$;
	otherwise the customer balks.
The sequence of balking tolerances has common distribution function $F_b$.
Naturally, the probability that a customer arriving at time $t$ balks is $F_b(Q_\alpha(t-) / \mu)$.

Consider \eqref{eq:mod:renegingT} and \eqref{eq:mod:renegingS}.
A customer who joins the queue is not guaranteed to stick around for service.
If the delay that customer $i$ experiences in the queue reaches $d_i$ then that customer will abandon from the queue.
The time that a customer arriving at time $t$ would have to wait is captured by $W_\alpha(t-)$.
An analogous workload process is the main object of study in \cite{ref:JR}.
The distribution of the abandonment time random variables $d_i$ is denoted $F_d$; the `d' stands for customer \textit{deadline}.
It follows that a non-balking customer arriving at time $t$ will abandon the queue with probability $F_d(W_\alpha(t-))$.
What sets the ticket queue apart from the standard one is the time at which the process $Q_\alpha$ reflects the abandonment of a job.
Hence the need to express $R_T(t)$ and $R_S(t)$ separately in \eqref{eq:mod:renegingT} and \eqref{eq:mod:renegingS},
	respectively.
Consider \eqref{eq:mod:renegingT} and how it captures customer abandonment.
Not only must the tolerance $d_i$ be smaller than the delay that the customer must endure before service,
	the system does not know that the customer has abandoned until that delay has expired.
Equations \eqref{eq:mod:renegingHatT} and \eqref{eq:mod:renegingHatS} are the analogous formulations for 
	reneging customers who are present at time zero. 

The workload process is also referred to as the virtual waiting time process because it tracks,
	as a function of time, the amount of time a sufficiently patient, non-balking customer would have to wait before receiving service.
The virtual waiting time process increases by the service time whenever a job arrives to the system that will eventually receive service.
The process decreases at rate one whenever it is greater than zero, or equivalently, whenever the queue length is nonzero.
The cumulative amount that the workload has decreased by time $t$ is precisely equal to the total busy time $T_\alpha(t)$.

\subsection{State space descriptors and simulation of the ticket queue}

Simulating the ticket queue is more complicated than simulating the standard queue.
We describe below the intricacies of simulating the ticket queue under both Markovian and non-Markovian assumptions.
We use simulation later to assess the accuracy of our approximations and heuristics.

Under Markovian assumptions, the state space of the ticket queue can be captured by a vector of zeros and ones.
The length of the vector corresponds to the number of outstanding tickets, including the ticket of the customer currently in service. (For convenience, assume that the first element of the vector is the leftmost element.)
If the vector has a nonzero length, the first element corresponds with the customer in service and by convention is a one.
The other elements correspond with the other unresolved tickets.
Further, the order of these elements reflects the relative order of the corresponding customers' arrivals and, under our first-come-first-served assumptions, the order in which resolution will take place.
Ones in the vector represent customers who have not abandoned the queue.
Some of these customers may abandon before resolution takes place.
When a customer abandons, the corresponding element turns into a zero.
When service of a sufficiently patient customer takes place, the state vector shifts to the left because at least the first element of the vector must be removed.
Either the second element is a zero or it is a one.
In the latter case, the entire state vector shifts by one element.
If the second element is a zero, this represents an abandoned ticket.
Starting with this zero in the second element, all contiguous zeros will be will be removed from the state descriptor.
The assumption here is that resolution of abandoned tickets is instantaneous.
When a job arrives to the system and chooses to join the queue a one is appended to the end of the state vector.
If the customer balks, no change in the state takes place.

The system transition is governed by exponential clocks for each unresolved ticket that has yet to be abandoned and is not being processed,
	one clock for the job in service, and one for the next arriving job.
If the clock associated with an unresolved ticket not in service expires,
	then the associated customer abandons and the element in the state descriptor changes from a 1 to a 0.
If there is an arrival, then another random variable is generated and compared to the weighted queue length 
	to determine whether the customer balks;
	if balking occurs the state does not change.
If the clock associated with the job in service expires, service completion ensues and the state changes as described above.

Alternatively, one could have one exponential clock for all unresolved tickets that are not in service and have not been abandoned.
The rate of this clock is equal to the number of such tickets multiplied by the abandonment rate of a single individual.
If this clock is the one that expires next  then the actual abandoning customer is found by 
	randomly choosing between the
	non-abandoned waiting customers with equal probability.
When there is a change in system state, this clock must be recalculated because the number of unresolved tickets 
	 will have changed as well. 

Under general assumptions on the random variables, the state space must contain the residual interarrival time of the next customer to arrive, the residual service time of the customer at the front of the queue,
	and for each customer yet to reach the front of the queue, the residual abandonment time.
There is an alternative state space formulation, if one is content with only knowing which of the customers in-queue will \textit{eventually} be served.
For this alternative, one must track the virtual waiting time process, $W_T$, which yields as a function of time the amount of time that a customer must wait until service begins, and the eventual service time of jobs that will be served.
These service times can be kept in a vector, similar to the vector of zeros and ones above. 
From the time of their arrival, jobs that will have abandoned before reaching the front of the queue have a zero in their
	corresponding element of the vector.
The virtual waiting time process jumps at the time of an arrival by the service time of the corresponding customer only if this 
	customer actually joins the queue and is sufficiently patient;
	see for instance \cite{ref:WG} or \cite{ref:JR}.
What is lost in this formulation is the timing of the individual jobs' abandonment times.
Gained is the freedom from having to track residual abandonment times for each job in-queue.

\section{Comparing the queue processes} \label{sec:construction}
In this section we provide the main result, a heavy traffic limit theorem that serves as the theoretical underpinning of 
	the heuristics forwarded in the subsequent section.

To facilitate comparing the ticket and standard queue, we appeal to heavy traffic limit theory.
To this end, we consider a sequence of systems, indexed by $n$.
The arrival and service rates of the $n_{th}$ system are $\lambda^n$ and $\mu^n$.
Equations \eqref{eq:mod:queue}--\eqref{eq:mod:busytime} have straightforward analogs with the $\lambda$ and $\mu$ replaced by $\lambda^n$ and $\mu^n$, respectively.
For each $t \ge 0,$ and $\alpha \in \{S,T\},$
\begin{equation} \label{eq:mod:queue^n}
	Q^n_\alpha(t)
	=
		\hat{Q}_\alpha^n(t)
		+
		A^n(t)
		-
		B^n_\alpha(t)
		-
		R^n_\alpha(t)
		-
		S^n_\alpha(T^n_\alpha(t) - W^n(0)), 
\end{equation}
\begin{equation} \label{eq:mod:hatQueue^n}
	\hat{Q}^n_\alpha(t)
	=
		Q^n(0)
		-
		\hat{R}^n_\alpha(t)
		-
		\hat{S}^n (t),
\end{equation}
\begin{equation} \label{eq:mod:arrivals^n}
	A^n(t) = \sup \left \{j \ge 0: \sum_{i=1}^j u_i/\lambda^n \le t \right \}
\end{equation}
\begin{equation} \label{eq:mod:balking^n}
	B^n_\alpha(t)
	=
		\sum_{i=1}^{A^n(t)} 1(b_i \le Q^n_\alpha(t^n_i-) / \mu^n), \quad \alpha \in \{S,T\},
\end{equation}
\begin{equation} \label{eq:mod:renegingT^n}
	R^n_T(t)
	=
		\sum_{i=1}^{A^n(t)}
			1(b_i > Q^n_T(t_i-) / \mu^n)
			\cdot
			1(d_i \le W^n_T(t_i-))
			\cdot
			1(W^n_T(t_i-) \le t - t_i), 
\end{equation}
\begin{equation} \label{eq:mod:renegingHatT^n}
	\hat{R}^n_T(t)
	=
		\sum_{i=1}^{Q^n(0)}
			1(\hat{d}_i \le \hat{w}^n_{i-1})
			\cdot
			1(\hat{w}^n_{i-1} \le t ), 
\end{equation}
\begin{equation} \label{eq:mod:renegingS^n}
	R^n_S(t)
	=
		\sum_{i=1}^{A^n(t)}
			1(b_i > Q^n_S(t_i-) / \mu^n)
			\cdot
			1(d_i \le \min(W^n_S(t_i-), t - t_i)), 
\end{equation}
\begin{equation} \label{eq:mod:renegingHatS^n}
	\hat{R}^n_S(t)
	=
		\sum_{i=1}^{Q^n(0)}
			1(\hat{d}_i \le \min(\hat{w}^n_{i-1}, t)), 
\end{equation}
\begin{equation} \label{eq:mod:busytime^n}
	T^n_\alpha(t) = \int_0^t 1(Q^n_\alpha(s) > 0) ds, \quad \alpha \in \{S,T\},
\end{equation}
\begin{equation} \label{eq:mod:workload^n}
	W^n_\alpha(t)
	=
		W^n(0)
		-
		T^n_\alpha(t)
		+
		\sum_{i=1}^{A^n(t)}
			(v_i/ \mu^n)
			\cdot
			1(b_i > Q^n_\alpha(t_i-)/ \mu^n)
			\cdot
			1(d_i > W^n_\alpha(t_i-)), \quad \alpha \in \{S,T\},
\end{equation}
and
\begin{equation} \label{eq:mod:idletime^n}
	I^n_\alpha(t) = t - T^n_\alpha(t), \quad \alpha \in \{S,T\},
\end{equation}
where
$$
	t^n_i = (1/\lambda^n) \sum_{j=1}^i u_j,
$$
$$ 
	W^n(0) = \hat{w}^n_{Q^n(0)},
	\qquad 
	\hat{w}^n_0 = 0,
	\qquad
	\mbox{and} \qquad
	\hat{w}^n_i
	=
	\sum_{j=1}^i
	\frac{\hat{v}_j}{\mu^n} \cdot 1(\hat{d}_j > \hat{w}^n_{j-1}), \qquad i \ge 1.
$$
The associated scaled processes are $Q^n_\alpha = \{Q^n_\alpha(t), t \ge 0\},$ 
	$\hat{Q}^n_\alpha = \{\hat{Q}^n_\alpha(t), t \ge 0\},$ 
	$A^n = \{A^n(t), t \ge 0\},$ $B^n_\alpha = \{B^n_\alpha(t), t \ge 0\},$ 
	$R^n_\alpha = \{R^n_\alpha(t), t \ge 0\},$
	$\hat{R}^n_\alpha = \{\hat{R}^n_\alpha(t), t \ge 0\},$
	$S^n_\alpha = \{S^n_\alpha(t), t \ge 0\},$ 
	$\hat{S}^n_\alpha = \{\hat{S}^n_\alpha(t), t \ge 0\},$ 
	$T^n_\alpha = \{T^n_\alpha(t), t \ge 0\},$ 
	$I^n_\alpha = \{I^n_\alpha(t), t \ge 0\}$, and $W^n_\alpha = \{W^n_\alpha(t), t \ge 0\}.$

We envision the arrival and service rates each being order $n$ and differing by a quantity that is  order $\sqrt{n}$.
So, in the absence of abandonment and balking, one would expect the queue length to be order $\sqrt{n}$
	and for the workload process to be order $1/\sqrt{n}$.
In fact, this intuition is  true in the presence of both balking and abandonment.
Hence, we define for each $\alpha \in \{S,T\}$,
	the diffusion scaled queue length process $\tilde{Q}^n_\alpha = \{\tilde{Q}^n_\alpha(t), t \ge 0\}$
	and the inflated workload process $\tilde{W}^n_\alpha = \{\tilde{W}^n_\alpha(t), t \ge 0\}$,
	 where for each $t \ge 0$,
$$
	\tilde{Q}^n_\alpha(t) = \frac{Q^n_\alpha(t)}{\sqrt{n}} 
		\qquad \mbox{ and } \qquad
	\tilde{W}^n_\alpha(t) = \sqrt{n} W^n_\alpha(t).
$$
Notice that we do not scale time as the arrival rates and service rates are already proportional to $n$.
Also notice that the balking and abandonment times do not change with $n$.
The reasoning is that demand may change and service speed must adjust accordingly,
	however, individuals will still have the same desires for and assessment of service quality.

We introduce the processes $\epsilon^n_\alpha = \{\epsilon^n_\alpha(t), t \ge 0\}$ for each $\alpha \in \{S,T\}$, where for each $t \ge0$,
\begin{equation} \label{eq:epslion^n}
	\epsilon^n_\alpha(t)
	=
		\hat{R}^n_\alpha(t) 
		+
		\left(
		R^n_\alpha(t)
		-
		\sum_{i=1}^{A^n(t)}
			1(d_i \le Q^n_\alpha(T^n_i-)/\mu^n)
		\right).
\end{equation}
The idea is to replace the reneging process $R^n_\alpha$ with a process that ignores whether the job has balked when considering whether it will renege.
In reality, a balking customer leaves and, as a result, the question of whether balking customers would renege is a moot one.
The introduction of this process also eliminates the concern of when the reneging customer causes a decrease in the queue length.  Here we assume that the customer never actually enters the queue.
One further subtly is that the workload is replaced by the weighted queue length -- the quantity used to determine balking -- so that one need only track the process $Q^n_\alpha$ rather than the joint process $(Q^n_\alpha,W^n_\alpha)$.
Lastly, a benefit of this formulation is that it allows for the ticket and standard queues to be handled simultaneously.
Ultimately, we will show that the process $\epsilon^n_\alpha$ is negligible under diffusion scaling, which partially argues why the diffusion-scaled ticket and standard queues converge together to the same limit.
The process $\epsilon^n_\alpha(\cdot)$  also eliminates the reneging of the initial jobs altogether;
	the reneging of initial jobs is shown to be neglibile in Proposition \ref{prop:noInitialReneging}.

The process $S^n_\alpha$ tracks the number of customers (who arrive after time 0) 
	served as a function of total effort dedicated to customers.
As not all customers receive service, 
	the service times that determine $S^n_\alpha$ are a subset of $\{v_i, i \ge 1\}$,
	the collection of potential service times of arriving jobs.
This subset differs for the ticket queue and the standard queue.
The index of the $i_{th}$ job whose service time contributes to $S^n_\alpha$ -- that is, who is actually served -- is
$$
	j^n_\alpha(i)
	=
	\inf
	\left \{k \ge 1: \sum_{\ell=1}^k
		1(b_i > Q^n_\alpha(t^n_\ell-) / \mu^n)
		\cdot
		1(d_i > W^n_\alpha(t^n_\ell-) \ge  i \right \} , \quad i \ge 1, \quad \alpha \in\{S,T\}.
$$
The sequence of service times that are actually used is denoted $\{v^n_\alpha(i), i \ge 1\}$, 
	where $v^n_\alpha(i) = v^n_{j^n_\alpha(i)}$.
Because the service time of job $j^n_\alpha(i)$ is independent of all random variables that dictate whether this service time is used,
	the filtered sequence $\{v^n_\alpha(i), i \ge 1\}$ is i.i.d. 
	and has the same distribution as the original (unfiltered) collection of service times.
Therefore, any property of the unfiltered service times -- such as weak laws of large numbers or invariance principles --
	holds for the filtered sequence.
Finally, we can write $S^n_\alpha$(t) for each $t \ge 0$ as
\begin{equation} \label{eq:S}
	S^n_\alpha(t)
	=
		\sup \left\{k \ge 0: (1/ \mu^n) \sum_{i=1}^k v^n_\alpha(i) \le t\right\}, \quad \alpha \in \{S,T\}.
\end{equation}

The processes $\hat{S} = \{\hat{S}(t), t \ge0\}$ and $\hat{S}^n = \{\hat{S}^n(t), t \ge0\}$ are also defined analogously.  

Now we can write the diffusion-scaled queue length processes for each $t \ge0$ as
\begin{eqnarray}\nonumber
	\tilde{Q}^n_\alpha(t)
	&=&
		\tilde{Q}^n(0)
		+
		\tilde{A}^n(t)
		-
		\tilde{M}^n_{b,\alpha}(\bar{A}^n(t))
		-
		\tilde{M}^n_{d,\alpha}(\bar{A}^n(t))
		-
		\tilde{\epsilon}^n_\alpha(t)
		-
		\tilde{\delta}^n_\alpha(t)
		-
		\tilde{S}^n_\alpha( (T^n_\alpha(t) - W^n(0))^+) \\
	&&   - \label{eq:mod:queue^tilde}
		\theta \int_0^t \tilde{Q}_\alpha^n(s) ds
		+
		\frac{(\lambda^n - \mu^n)}{\sqrt{n}}  t
		+
		\tilde{Y}^n_\alpha(t), \quad \alpha \in \{S,T\},
\end{eqnarray}
where, for each $\alpha \in \{S,T\},$ $\tilde{Q}_\alpha^n(0) = (1/\sqrt{n}) Q^n(0)$ is the scaled initial queue length 
	and, for each $t \ge 0$,
\begin{equation} \label{eq:mod:arrival^tilde}
	\tilde{A}^n(t) = (1/\sqrt{n}) \left( A^n(t) - \lambda^n t \right),
\end{equation}
\begin{equation} \label{eq:mod:arrival^bar}
	\bar{A}^n(t) = (1/n) A^n(t),
\end{equation}
\begin{equation} \label{eq:mod:balking^tilde}
	\tilde{M}^n_{b,\alpha}(t)
		=
		(1/\sqrt{n}) \sum_{i=1}^{\lfloor nt \rfloor}
			\left(
			1(b_i \le Q^n_\alpha(t^n_i-)/\mu^n)
			-
			F_b(\sqrt{n} \tilde{Q}^n_\alpha(t^n_i-)/\mu^n)
			\right),
\end{equation}
\begin{equation} \label{eq:mod:reneging^tilde}
	\tilde{M}^n_{d,\alpha}(t)
		=
		(1/\sqrt{n}) \sum_{i=1}^{\lfloor nt \rfloor}
			\left(
			1(d_i \le Q^n_\alpha(t^n_i-)/\mu^n)
			-
			F_d(\sqrt{n} \tilde{Q}^n_\alpha(t^n_i-)/\mu^n)
			\right),
\end{equation}
\begin{equation} \label{eq:mod:departure^tilde}
	\tilde{S}^n_\alpha(t) = (1/\sqrt{n}) \left( S^n_\alpha(t) - \mu^n t \right),
\end{equation}
\begin{equation} \label{eq:mod:epsilon^epsilon}
	\tilde{\epsilon}^n_\alpha = (1/\sqrt{n}) \epsilon^n_\alpha(t),
\end{equation}
\begin{eqnarray} \label{eq:mod:delta^tilde}
	\tilde{\delta}^n_\alpha(t)
	&=&
	\frac{1}{\sqrt{n}} \sum_{i=1}^{A^n(t)}
		\left(
		F_b(\sqrt{n} \tilde{Q} ^n_\alpha(t^n_i-)/\mu^n)
		+
		F_d(\sqrt{n} \tilde{Q}^n_\alpha(t^n_i-)/\mu^n)
	\right)
	- \theta \int_0^t \tilde{Q}^n_\alpha(s) ds
	\\ &&+ \nonumber
	\frac{1}{\sqrt{n}} \left(
		\hat{S}^n(t) - \mu^n \min(t,W^n(0))
	\right),
\end{eqnarray}
and
\begin{equation} \label{eq:mod:idleTime^tilde}
	\tilde{Y}^n_\alpha(t)
	=
	\left( \frac{\mu^n}{n} \right)
		\tilde{I}^n_\alpha(t)
	=
	\left( \frac{\mu^n}{\sqrt{n}} \right)
		 I^n_\alpha(t).
\end{equation}
We refer to $\tilde{A}^n = \{\tilde{A}^n(t), t \ge 0\}$ as the diffusion-scaled arrival process and to  $\bar{A}^n = \{\bar{A}^n(t), t \ge 0\}$ as its fluid-scaled analog.
The reader may notice that the  process 
	$\tilde{\delta}^n_\alpha = \{\tilde{\delta}^n_\alpha(t), t \ge 0\}$
	has what looks like an instantaneous drift that is proportionate to the value of the scaled queue length process.
The remaining processes are centered and diffusion scaled versions of their original analogs:
	$\tilde{M}^n_{b,\alpha} = \{\tilde{M}^n_{b,\alpha}(t), t \ge 0\}$,
	$\tilde{M}^n_{d,\alpha} = \{\tilde{M}^n_{d,\alpha}(t), t \ge 0\}$,
	$\tilde{S}^n_\alpha = \{\tilde{S}^n_\alpha(t), t \ge 0\}$,
	$\tilde{\epsilon}^n_\alpha = \{\tilde{\epsilon}^n_\alpha(t), t \ge 0\}$,
	and $\tilde{Y}^n_\alpha = \{Y^n_\alpha(t), t \ge 0\}$.

\subsection{A heavy traffic limit theorem}

In order to prove a heavy traffic limit theorem, we assume for our sequence of systems indexed by $n$, that arrival and service rates are order $n$  quantities and are asymptotically identical; that is, as $n \to \infty$,
\begin{equation} \label{eq:limiting rates}
	\lambda^n/n \to \mu
	\quad \mbox{ and }
	\quad
	\mu^n/n \to \mu.
\end{equation}
Further, the difference between the two should be an order $\sqrt{n}$ quantity such that as we take the limit $n \to \infty$,
\begin{equation} \label{eq:limiting beta}
	(\lambda^n - \mu^n)/ \sqrt{n} = \beta^n \to \beta \in (-\infty,\infty).
\end{equation}
One can refer to \eqref{eq:limiting rates} as the heavy traffic condition; the expression implies that
\begin{equation} \label{eq:heavyTraffic}
	\rho^n  = \lambda^n/\mu^n
	\to
	1,
\end{equation}
as $n \to \infty$.
We assume that the random variables associated with balking and abandonment are unaffected by the change in the index $n$.
Define
\begin{equation} \label{eq:sigma}
	\sigma
	\equiv
	\mu
	\sqrt{ \sigma^2_a
		+
		\sigma^2_s}
\end{equation}
as the standard deviation associated with the arrival and service times.
Lastly, define $B = \{B(t), t \ge 0\}$ as a Brownian motion with no drift and an infinitesimal variance of 1.

The framework developed in \cite{ref:RW} justifies the alternative representation of \eqref{eq:mod:queue^tilde}:
\begin{equation} \label{eq:queueIdle}
	(\tilde{Q}^n_\alpha, \tilde{Y}^n_\alpha)
	=
	(\Phi_\theta, \Psi_\theta)(\tilde{Q}^n(0) + \tilde{X}^n_\alpha),
\end{equation}
where $(\Phi_\theta, \Psi_\theta): \bD(\Rp,\R) \mapsto \bD(\Rp,\Rp^+)$ is a 
	Lipshitz continuous map,
	 $\tilde{X}^n_\alpha = \{\tilde{X}^n_\alpha(t), t \ge 0\}$, and for each $t \ge 0$ and $\alpha \in \{S,T\},$
\begin{equation} \label{eq:mod:X^tilde}
	\tilde{X}^n_\alpha(t)
	=
		\tilde{A}^n(t)
		-
		\tilde{M}^n_{b,\alpha}(\bar{A}^n(t))
		-
		\tilde{M}^n_{d,\alpha}(\bar{A}^n(t))
		-
		\tilde{\epsilon}^n_\alpha(t)
		-
		\tilde{\delta}^n_\alpha(t)
		-
		\tilde{S}^n_\alpha(T^n_\alpha(t))
		+
		\frac{(\lambda^n - \mu^n)}{\sqrt{n}}  t.
\end{equation}
The elements of $\tilde{X}^n$ are those that either will converge to Brownian motions or that are 
	asymptotically negligible.
The limiting stochastic process is the following
\begin{equation} \label{eq:tildeX}
	\tilde{X} = \beta e + \sigma B.
\end{equation}

We now present our main result for the diffusion scaled queue length and workload processes.
\begin{theorem} \label{thm:main}
	If 
\begin{equation} \label{eq:thm:main}
	(\tilde{Q}^n(0), \tilde{W}^n(0)) \Rightarrow (\tilde{Q}_0, \tilde{Q}_0/\mu), \qquad \mbox{as } n \to \infty,
\end{equation}
then 
\begin{equation} \label{eq:thm:main2}
		((\tilde{Q}^n_S, \tilde{W}^n_S, \tilde{Y}^n_S),(\tilde{Q}^n_T, \tilde{W}^n_T, \tilde{Y}^n_T))
			 \Rightarrow
		((\tilde{Q}, \tilde{Q}/\mu, \tilde{Y}),(\tilde{Q}, \tilde{Q}/\mu, \tilde{Y})),
	\qquad \mbox{as } n \to \infty,
\end{equation}
	where $\tilde{Q}(0)$ is equal in distribution to $\tilde{Q}_0$,
	$\tilde{Q} = \Phi_\theta(\tilde{Q}(0) + \tilde{X})$, 
	$\tilde{Y} = \Psi_\theta(\tilde{Q}(0) + \tilde{X})$,
	and together $\tilde{Q}$ and $\tilde{Y}$  obey the following stochastic differential equation
	\begin{equation} \label{eq:SDE}
		d \tilde{Q} (t)
		=
		-\theta (\beta/\theta - \tilde{Q}(t)) dt
		+
		\sigma d B(t)
		+
		d \tilde{Y}(t).
	\end{equation}
\end{theorem}
\begin{remark}
	The process $\tilde{Q}$ is referred to as an ROU process.
	The steady state distribution of $\tilde{Q}$ is a truncated (at zero) normal variable,
	\begin{equation} \label{eq:ROU_steadyState}
		\tilde{Q}(\infty)
		=
		Normal \left(\frac{\beta}{\theta},\frac{\sigma^2}{2 \theta},0,\infty \right),
	\end{equation}
	whose mean is
	\begin{equation} \label{eq:ROU_mean}
		E[\tilde{Q}(\infty)]
		=
		\frac{\beta}{\theta}
			+
			\frac{\sigma}{\sqrt{2 \theta}}
			h\left(-\frac{\beta}{\sigma \sqrt{\theta/2}} \right).
	\end{equation}
	where the hazard function $h(\cdot)$ is defined as the ratio of the density and the tail of the standard
	normal distribution:
	\begin{equation} \label{eq:Hazard_rate}
		h\left(x\right)
		=
			\frac{\varphi(x)}{1-\Phi(x)}, 
		\qquad \mbox{for all } x \in \R.
	\end{equation}
\end{remark}

\subsection{Preliminaries}
We conclude this section with several results that are well known in the heavy traffic literature.
As such, we do not provide proofs.
The lemmas are all similar in substance to those in Lemma 3.1 of \cite{ref:J}.
\begin{lemma} {Bounded total arrivals.}  \label{lem:boundedArrivals}
	For any $t \ge 0$,
	$$
		\lim_{n \to \infty}
		 \Prob \left(A^n(t) > 2 \mu n t \right) =0.
	$$
\end{lemma}
\begin{lemma} {Bounded maximum service time.} \label{lem:boundedMaxService}
	For any $\epsilon, K, t \ge 0$,
	$$
		\lim_{n \to \infty}
		\Prob \left(\sup_{i \le  K n t}
		v^n_i > \epsilon / \sqrt{n} \right) 
		=
		\lim_{n \to \infty}
		\Prob \left(\sup_{i \le  K n t}
		\frac{\hat{v}_i}{\mu^n} > \epsilon / \sqrt{n} \right) = 0.
	$$
\end{lemma}
\begin{lemma}{Functional law of large numbers for initial service times} \label{lem:serviceTimeFLLN}
	For any $\epsilon, b> 0$,
	$$
		\limsup_{n \to \infty}
		\Prob\left(
			\sup_{j,k \le b\sqrt{n}}
			\sqrt{n} \left|
				\sum_{i=j+1}^k
					\frac{\hat{v}_i}{\mu^n}
				-
				\frac{(k-j)}{\mu n}
			\right|
			> \epsilon
		\right)
		=
		0.
	$$
\end{lemma}
Lemma \ref{lem:boundedArrivals} places an upper bound on the arrival process.
This bound allows us to replace the number of arrivals in an interval with a deterministic upper bound.
Likewise, Lemma \ref{lem:boundedMaxService} gives a uniform upper bound on service times that are of order $n$ in quantity. 
Used in conjunction with Lemma \ref{lem:boundedArrivals}, Lemma \ref{lem:boundedMaxService} places an upper bound
	on all service times during any finite interval of time.
Lemma \ref{lem:serviceTimeFLLN} places a bound on the amount by which the service times of initial jobs can differ from 
	their expected value.

The following four results pertain to the arrival of jobs and the arrival of potential work. 
The first result states that jobs arrive in a linear fashion.
\begin{lemma} {Uniformly bounded fluid arrivals.}  \label{lem:uniformArrivals}
	For any $\epsilon$ and $t > 0$,
	$$	\lim_{n \to \infty}
		\Prob \left(  \sup_{s \le t} \left|
			\bar A^n(s)
			- \mu s
			\right|
			 > \epsilon 
		\right) = 0.   	$$
\end{lemma}
The second result, based on the heavy traffic condition, is a functional law of large numbers and 
	states that asymptotically, potential work arrives at rate 1, uniformly over compact intervals.
\begin{lemma}  {Uniformly bounded fluid potential workload.} \label{lem:uniformPotentialWork}
	For any $\epsilon$ and $t > 0$,
	$$
		\lim_{n \to \infty}
		\Prob \left(  \sup_{s \le t} \left|
			\sum_{i = 1}^{A^n(s)}
				v^n_i
			- s
			\right|
			 > \epsilon 
		\right) = 0.
	$$
\end{lemma}
The above lemma is a key component for demonstrating that server idleness is asymptotically negligible; 
	see Proposition \ref{prop:allocationLimit}.
The lemma is also used for the proof of tightness of our sequence of scaled queueing processes; 
	see Proposition \ref{prop:tightness}.
However, we also need another version of the above lemma, but for short time intervals:
\begin{lemma} {Net potential workload tightness.}  \label{lem:shortPotentialWork}
	For any $\epsilon$ and $t > 0$, there exists a $\delta > 0$ such that
	$$
		\lim_{n \to \infty}
		\Prob \left(  \sup_{u<v \le t, v-u < \delta} \left|
			\sum_{i = A^n(u)+1}^{A^n(v)}
				v^n_i
			- (v-u)
			\right|
			 > \frac{\epsilon}{\sqrt{n}}
		\right) = 0.
	$$
\end{lemma}
For intuition as to why this lemma holds, consider the required server effort that all arrivals would contribute 
	to the workload process if none of the customers abandoned or balked. 
Then center this process at each time $t$ by $t$ itself, the potential amount of time that the server could have
	been working had the server never idled.
If this centered process is then scaled up by $\sqrt{n}$ and the limit is taken, the result is a Brownian motion.
It follows that the sequence of processes is tight and this fact is used in the proof of Proposition \ref{prop:tightness} .

\section{Approximations and interpretations} \label{sec:approximations}

Now we use the limits of the previous section to approximate several performance metrics.
Throughout this section, assume that we have a queueing system with arrival rate $\lambda$, service rate $\mu$, 
	balking distribution $F_b$, abandonment distribution $F_d$, standard deviation of interarrival times $\sigma_a$,
	and standard deviation of service times $\sigma_s$. 
To draw connections between the formal limiting procedure with the original queueing system, we make the following notational conventions:
$$
	\mu^n = n,
\qquad
	\beta = \frac{\lambda - \mu}{\sqrt{\mu}},
\qquad
	\lambda^n = n + \beta \sqrt{n}
$$
and
%

$$
	\hat{\sigma} = \sqrt{
		\left(\sigma_a \cdot \lambda \right)^2
		+
		\left(\sigma_s \cdot \mu\right)^2
	}.
$$ 
Notice that defining any two of $\lambda$, $\mu$ and $\beta$ uniquely defines the third.

\subsection{Distribution of the ticket queue in steady state}

Noting the scaling $Q^n = \sqrt{n} \tilde{Q}^n  \approx \sqrt{\mu} \tilde{Q}$, we approximate the steady state distribution of our queueing process using the steady state distribution of the ROU process:
\begin{equation} \label{eq:queueLengthDistribution}
	Q \approx Normal \left( \frac{\lambda - \mu}{\theta}, \frac{\mu \hat\sigma^2}{2 \theta}, 0, \infty \right),
\end{equation}
that is, the ticket queue distribution is approximated by a normal distribution with mean $(\lambda - \mu) / \theta$, variance $\mu \hat\sigma^2 / (2 \theta)$ and truncated to lie within $[0,\infty)$.
Note the substitution: $\beta \approx (\lambda - \mu) / \sqrt{\mu}$.

\subsection{The expected ticket queue length}

We can also approximate the expected queue length as

\begin{equation}  \label{eq:expectedQueueLength}
	E[Q]
		\approx
	\frac{\lambda - \mu}{\theta}
		+
	\hat\sigma
	\sqrt{\frac{ \mu}{2 \theta}}
		h \left(\frac{(1-\rho)}{\hat\sigma}
			\sqrt{\frac{2 \mu}{\theta}} \right)
\end{equation}
where $\rho  = \lambda / \mu$.  If one is also interested in approximations for higher order cumulant moments, see for example, \cite{ref:Pen, ref:Pen2, ref:MP}.

\subsection{The fraction of abandonment}

There are three approximations forwarded for the abandonment probability.  To simplify notation we let $g(0) = F_d'(0)$ and $f(0) = F_b'(0)$.   
The first approximation takes the rate of abandonment from the queue and divides by the total arrival rate:
\begin{equation} \label{eq:abandonmentFraction1} 
	\alpha_1
	\approx
		\frac{g(0) E[Q]}{\lambda}
	\approx
		\frac{\rho-1}{\rho}
			\frac{g(0)}{\theta}
		+
		\frac{\hat\sigma g(0)}{\rho \sqrt{2 \theta \mu}}
		h \left(\frac{(1-\rho)}{\hat\sigma}
			\sqrt{\frac{2 \mu}{\theta}} \right).
\end{equation}
The second approach starts with computing the cumulative distribution function evaluated at the expected delay:
\begin{eqnarray} \label{eq:abandonmentFraction2}
	\alpha_2
	&&\approx
		E\left[ G\left(\frac{Q}{\mu}\right)\right]
	=
		E\left[ G\left(\frac{\tilde{Q}^n}{\sqrt{\mu}} \right) \right]
	\approx
		g(0) \frac{E[\tilde{Q}^n]}{\sqrt{\mu}}
	=
		g(0) \frac{E[Q]}{\mu} \\
	&&\approx
		(\rho-1)
			\frac{g(0)}{\theta}
		+
		\frac{\hat\sigma g(0)} {\sqrt{2 \theta \mu}}
			h \left(\frac{(1-\rho)}{\hat\sigma}
			\sqrt{\frac{2 \mu}{\theta}} \right). \nonumber
\end{eqnarray}
The last approach is the simplification of the first two under the assumption that $\rho = 1$:
\begin{equation} \label{eq:abandonmentFraction3}
	\alpha_3
	\approx
		\frac{\sigma g(0)} {2 \sqrt{\pi \theta \mu}}.
\end{equation}

\subsection{The fraction of balking customers}

The balking probabilities are similar to the abandonment ones and, as such, have three versions.

\begin{equation} \label{eq:balkingFraction1}
	\gamma_1
	\approx
		\frac{\rho-1}{\rho}
			\frac{f(0)}{\theta}
		+
		\frac{\hat\sigma f(0)}{\rho \sqrt{2 \theta \mu}}
			h \left(\frac{(1-\rho)}{\hat\sigma}
			\sqrt{\frac{2 \mu}{\theta}} \right),
\end{equation}

\begin{equation} \label{eq:balkingFraction2}
	\gamma_2
	\approx
		(\rho-1)
			\frac{f(0)}{\theta}
		+
		\frac{\hat\sigma f(0)} {\sqrt{2 \theta \mu}}
			h \left(\frac{(1-\rho)}{\hat\sigma}
			\sqrt{\frac{2 \mu}{\theta}} \right),
\end{equation}
and when $\rho = 1$, we have that
\begin{equation} \label{eq:balkingFraction3}
	\gamma_3
	\approx
		\frac{\sigma f(0)} {2 \sqrt{\pi \theta \mu}}.
\end{equation}
	
\subsection{The expected number of unresolved abandoned tickets}

Given the number of unresolved tickets a fixed fraction of these are expected to be abandoned:
\begin{equation} \label{eq:unresolvedAbandonedTicketDist1}
	\mathbb{E} [X(t)]
	\approx
		\frac{1}{2} G\left(\frac{Q(t) }{\mu} \right) Q(t)
	\approx
		 \frac{g(0) Q(t)^2}{2 \mu}
	 \approx
	 	\frac{g(0)}{2} \tilde{Q}^2
\approx
	 	\frac{g(0)}{2} E[\tilde{Q}(\infty)]^2.
\end{equation}

\subsection{Interpretation}
So why should we believe that there is very little difference between the ticket queue and the standard queue in steady state?
In the absence of balking, one would assume that the number of customers in the standard queue would be smaller than that in the ticket queue, as the former rids itself of customers who will not add to the server workload.
One mechanism that reduces this difference is that the ticket queue, albeit longer, ultimately sees  less work than its queue length would suggest.
Hence it must be resolving its ticket queue faster than the standard queue is processing its customers.
Moreover, the concentration of abandoned tickets is typically greater among the tickets close to the front of the queue, as these tickets have been in circulation the longest.
But the closer the ticket is to the front, the sooner it gets resolved.
The more abandoned tickets, the faster the server resolves such tickets.
Hence, the ticket queue tends to drive itself back toward the standard queue status the farther away it deviates from it.

Adding in the balking customers further lessens the difference between the two queueing scenarios.
If the ticket queue is longer than the standard queue then the former has more customers balking at the front end.
To conclude, as the distance grows between the length of the ticket and standard queues, so do the forces that force the coupling of the two queueing models.
This notion is formally expressed in Theorem \ref{thm:asymptoticCoupling}.

\section{Proof of the main results} \label{sec:proof}
The results that follow lead up to the proof of the main result at the conclusion of this section.
Some proofs are delayed until the Appendix.

\subsection{Asymptotic Boundedness}
We argue first that the scaled queue length processes  and the workload processes are asymptotically bounded.
\begin{lemma} \label{lem:bounded}
	Under \eqref{eq:thm:main}, we have that for any $t,\eta >0$ there exists a $K=K(\eta) >0$ such that for each 
	$\alpha \in \{S,T\}$,
	\begin{equation} \label{eq:lemBoundedQueue}
		\limsup_{n \to \infty} \Prob \left( \sup_{s \in [0,t]} \tilde{Q}^n_\alpha(s) > K \right) < \eta
	\end{equation}
	and
	\begin{equation} \label{eq:lemBoundedWorkload}
		\limsup_{n \to \infty} \Prob \left( \sup_{s \in [0,t]} \tilde{W}^n_\alpha(s) > K \right) < \eta.
	\end{equation}
\end{lemma}
\begin{proof}
	The queue length processes for the ticket and standard queues can both be bounded path-wise by a 
		third queueing process that contains neither balking nor abandonment.
	It is standard that this third scaled queue length process converges to a reflected Brownian motion.
	It also follows that this third scaled queue length process exhibits the boundedness expressed in \eqref{eq:lemBoundedQueue};
		e.g., see Lemma 3.4 of \cite{ref:J}.
	And because this third process bounds the ticket and standard queueing processes for every time $t$, 
		the result in \eqref{eq:lemBoundedQueue} follows.  The same arguments hold for the workload processes in \eqref{eq:lemBoundedWorkload} and this concludes the proof.
\end{proof}
Lemma \ref{lem:bounded} emphasizes the orders of magnitude of the queueing and workload processes.
This lemma will be used frequently in conjunction with the balking and abandonment distributions 
	to place bounds on abandonment and balking frequencies.

\subsection{Abandonment and balking frequencies}
Next, we cover several properties of the accumulation of balking and abandonment events among the arriving jobs.  The following lemmas, which besides Lemma \ref{lem:theta2}, are provided without proof,
	use the fact that the derivatives of the balking and abandonment distributions exist at zero; see \eqref{eq:theta}.
The first lemma is used throughout this section and follows from a straightforward application of Taylor's Expansion.
\begin{lemma} \label{lem:theta}
	For any $K>0$,
	$$
		 \frac{F_b(K/\sqrt{n})}{K/\sqrt{n}} + \frac{F_d(K/\sqrt{n})}{K/\sqrt{n}}  < 2 \theta 
	$$
	for sufficiently large $n$.
\end{lemma}
The second lemma is similar.
\begin{lemma}\label{lem:theta2}
	For any $\delta, K>0$,
	$$
		\sup_{s \in [0,K]} 
		\left(
			\frac{F_b((s+\delta)/\sqrt{n}) - F_b(s/\sqrt{n})}{\delta/\sqrt{n}} 
			+
			\frac{F_d((s+\delta)/\sqrt{n}) - F_d(s/\sqrt{n})}{\delta/\sqrt{n}} 
		\right)
		< 2 \theta 
	$$
	for sufficiently large $n$.	
\end{lemma}

The next result shows that one can choose a sufficiently small $\delta$ such that,
	uniformly over all subintervals of $[0,t]$ of size $\delta$,
	the total number of jobs that arrive in any subinterval that either abandon or balk
	is arbitrarily small.
The proof of this and subsequent results can be found in the Appendix.
\begin{proposition} \label{prop:boundedBalkingAbandonment_shortInterval}
	For any $\epsilon, \eta, t >0$ and $K >0$, there exists a $\delta >0$ such that
	\begin{equation} \label{eq:proof:prop:boundedBalkingAbandonment.0}
		\limsup_{n \to \infty}
		\Prob \left( 
			\sup_{s \le t}
			\sum_{i = A^n(s)+1}^{A^n(s+\delta)}
			\left(
				1(b_i \le K/\sqrt{n}) + 1(d_i \le K/\sqrt{n}) 
			\right)
			> \epsilon \sqrt{n}
		\right)
		< \eta.
	\end{equation}
\end{proposition}

Likewise, for a sufficiently small time interval, the amount of potential workload contribution 
	associated with balking or abandoning jobs arriving during the interval is smaller than order $1/\sqrt{n}$.
This result is a key element in the proof of tightness of our scaled queue length and workload processes;
	 see Proposition \ref{prop:tightness}.
\begin{proposition} \label{prop:boundedLostWorkload_shortInterval}
	For any $\eta, t >0$ and $K >0$, there exists a $\delta$ such that
	\begin{equation} \label{eq:proof:prop:boundedBalkingAbandonment_shortInterval}
		\limsup_{n \to \infty}
		\Prob \left( 
			\sup_{s \le t}
			\sum_{i = A^n(s)+1}^{A^n(s+\delta)}
			v^n_i \cdot
			\left(
				 1(b_i \le K / \sqrt{n}) + 1(d_i \le K/\sqrt{n}) 
			\right)
			> \frac{\epsilon}{ \sqrt{n}}
		\right)
		< \eta.
	\end{equation}
\end{proposition}

So far our propositions have been able to replace the queueing and workload processes with upper bounds
	early in the proofs. 
For the following result, 
	where we show that  the centered and scaled balking and approximate abandonment processes converge to zero,
	such substitutions cannot be made immediately.
\begin{proposition}{Centered balking and reneging processes are negligible.} \label{prop:centeredBalkingReneging}
	Under the assumptions of Theorem \ref{thm:main}, for each $\alpha \in \{S,T\}$, and 
	any $\varepsilon, \eta, t > 0$,
	\begin{equation} \label{eq:centeredBalking}
		\limsup_{n \to \infty}
		 \Prob \left(
		 	\sup_{s \in [0,t]}
			\left|
				\tilde{M}^n_{b,\alpha}(s)
			\right|
			> \varepsilon
		\right) 
		< \eta
	\end{equation}
and
		\begin{equation} \label{eq:centeredAbandonment}
		\limsup_{n \to \infty}
		 \Prob \left(
		 	\sup_{s \in [0,t]}
			\left|
				\tilde{M}^n_{d,\alpha}(s)
			\right|
			> \varepsilon
		\right) 
		< \eta.
	\end{equation}
\end{proposition}
\textit{The implications here are that the balking and reneging random variables can be replaced with their respective distribution functions.}

\subsection{Coupled Processes}

An interpretation of Lemma \ref{lem:bounded} is that the queue length is order $\sqrt{n}$.
In a model with no balking or abandonment, this fact would be sufficient to draw a linear relationship
	between the queue length and the workload of the form $Q/\mu \approx W$.
In the presence of balking or abandonment, this relationship is justified in \cite{ref:WG}.
The key is that the number of jobs in queue who do not contribute to the workload is 
	negligible with respect to $\sqrt{n}$. 
We have the same result here.
\begin{proposition} {State space collapse.}\label{prop:queueWorkload} 
	Under the conditions of \eqref{eq:thm:main}, we have that for any $t, \varepsilon, \eta >0$ and each 
	$\alpha \in \{S,T\}$,
	$$
		\limsup_{n \to \infty} 
		\Prob \left( 
			\sup_{s \in [0,t]} 
			\left |\tilde{Q}^n_\alpha(s) - \mu \tilde{W}^n_\alpha(s) \right|
			> \varepsilon 
		\right) < \eta.
	$$
\end{proposition}

Next, we  establish that the sequences $\{\tilde{Q}^n_T, n \ge 1\}$ and $\{\tilde{Q}^n_S, n \ge 1\}$  
	converge to the same limit, if anything at all, as do $\{\tilde{W}^n_T, n \ge 1\}$ and $\{\tilde{W}^n_S, n \ge 1\}$.
\begin{theorem} {Asymptotic coupling.} \label{thm:asymptoticCoupling}
	Under the assumptions of Theorem \ref{thm:main}, 
	for any  $\varepsilon, \eta,t > 0$, 
	\begin{equation} \label{eq:thm:asymptoticCoupling1}
		\limsup_{n \to \infty} 
		\Prob \left(
			\sup_{s \in [0, t]} 
				\left| \tilde{Q}^n_S(s) - \tilde{Q}^n_T(s) \right| 
			> \varepsilon
		\right)
		< \eta
	\end{equation}
	and
	\begin{equation} \label{eq:thm:asymptoticCoupling2}
		\limsup_{n \to \infty} 
		\Prob \left(
			\sup_{s \in [0, t]} 
				\left| \tilde{W}^n_S(s) - \tilde{W}^n_T(s) \right| 
			> \varepsilon
		\right)
		< \eta.
	\end{equation}
\end{theorem}
\begin{proof}
	We first show that \eqref{eq:thm:asymptoticCoupling2} holds. 
	Then \eqref{eq:thm:asymptoticCoupling1} follows from applying the Triangle Inequality twice 
		and Propositions \ref{prop:queueWorkload} for both $\alpha = S$ and $\alpha = T$.
	
	Fix $\varepsilon, \eta, t > 0$. Removing the absolute value signs yields
	\begin{eqnarray} \nonumber
	\lefteqn{
	\Prob \left(
			\sup_{s \in [0, t]} 
				\left| \tilde{W}^n_S(s) - \tilde{W}^n_T(s) \right| 
			> \varepsilon
		\right) 
	}\\
	&\le& \label{eq:thm:asymptoticCoupling2.1} 
	\Prob \left(
			\sup_{s \in [0, t]} 
				\tilde{W}^n_S(s) - \tilde{W}^n_T(s)
			> \varepsilon
		\right) 
	+
	\Prob \left(
			\sup_{s \in [0, t]} 
				\tilde{W}^n_T(s) - \tilde{W}^n_S(s)
			> \varepsilon
		\right).
	\end{eqnarray}		
	We will show that 
	\begin{equation} \label{eq:thm:asymptoticCoupling2.2}
	\limsup_{n \to \infty} 
	\Prob \left(
			\sup_{s \in [0, t]} 
				\tilde{W}^n_T(s) - \tilde{W}^n_S(s)
			> \varepsilon
		\right) 
	<
	\eta/2
	\end{equation}
	and then 
	\begin{equation} \label{eq:thm:asymptoticCoupling2.3}
	\limsup_{n \to \infty} 
	\Prob \left(
			\sup_{s \in [0, t]} 
				\tilde{W}^n_S(s) - \tilde{W}^n_T(s)
			> \varepsilon
		\right) 
	<
	\eta/2
	\end{equation}
	will follow by symmetry.
	
	There are two steps to  demonstrating that \eqref{eq:thm:asymptoticCoupling2.2} holds.
		First we restrict the amount by which the gap between $\tilde W_T$ and $ \tilde{W}_S$ can grow at any instant.
	Then we show that once the gap is of a certain size -- one that is smaller than $\epsilon$ -- the gap 
		necessarily must shrink.
	To this end, introduce the notation for any $s_0 \le  t$:
	$$
		\tau^n(s_0) = \inf \{s \in [s_0, t]: \tilde{W}^n_T(s) - \tilde{W}^n_S(s) \ge \varepsilon\}
	$$
	and
	$$
		\gamma^n(s_0) = \inf \{s \in [s_0, t]: \tilde{W}^n_T(s) - \tilde{W}^n_S(s) \le \varepsilon/2\}.
	$$
	where either stopping time is equal to $t$ if the corresponding infimum is taken over an empty set. 
	Suppose $\tilde{W}^n_T$ is greater than $\tilde{W}^n_S$.
	The depletion of the former is always as fast as the that of the latter since the servers work at the same rate.	
	Therefore, the gap between $\tilde{W}^n_T$ and $\tilde{W}^n_S$ can only increase due to
		jumps in $\tilde{W}^n_T$.
	If jumps in $\tilde{W}^n_T$ are all strictly less than $\varepsilon/4$, then before the gap 
		can exceed $\varepsilon$, it must first assume some value in $(3\varepsilon/4,\varepsilon)$.
	Then, once in this interval, the process must hit $\varepsilon$ before falling below $\varepsilon/2$.
	Otherwise, the gap must again assume a value in  $(3\varepsilon/4,\varepsilon)$ before it reaches $\varepsilon$.
	Hence, we have that 
	\begin{eqnarray}  \label{eq:thm:asymptoticCoupling2.2.0}
		\Prob \left(
			\sup_{s \in [0, t]} 
				\tilde{W}^n_T(s) - \tilde{W}^n_S(s)
			> \varepsilon
		\right) 
	&\le&
		\Prob \left(
			\sup_{s \in [0, t]} 
				\tilde{W}^n_T(s) - \tilde{W}^n_T(s-)
			\ge \varepsilon / 4
		\right)  \\ \nonumber
	&&+ \, 
		\Prob \left( \exists s_0 \le t \; \mbox{ s.t. } 
				\tilde{W}^n_T(s_0) - \tilde{W}^n_S(s_0) \in \left(\frac{3 \varepsilon}{4},\varepsilon\right) \mbox{ and }
				\tau^n(s_0) < \gamma^n(s_0)
		\right)  
	\end{eqnarray}	
	For the first term on the right hand side, a jump in the ticket queue workload 
		must be due to a large service time associated with an arriving job.
	By Lemmas \ref{lem:boundedArrivals} and \ref{lem:boundedMaxService},
	\begin{equation}  \label{eq:thm:asymptoticCoupling2.2.1}
		\Prob \left(
			\sup_{s \in [0, t]} 
				\tilde{W}^n_T(s) - \tilde{W}^n_T(s-)
			\ge \varepsilon / 4
		\right) 
		\le
		\Prob \left( A^n(t) > 2 \mu n t \right) 
		+ \,
		\Prob \left( \sup_{i \le 2 \mu n t} v^n(i) \ge \varepsilon/4 \right)
		< \frac{\eta}{2}.
	\end{equation}
	As for the second term on the right hand side,  $\tilde{W}^n_T$ is greater than $\tilde{W}^n_S$ 
		throughout the interval $[s_0, \min(\tau^n(s_0), \gamma^n(s_0)]$.
	It follows then that any new job that arrives to both queues during that interval and 
		eventually abandons the standard queue 
		must also abandon the ticket queue. 
	So for the gap between the processes to increase during this interval, it must be due to jobs that 
		balk at the standard queue but not at the ticket queue;
		this is possible only if the ticket queue is smaller than the standard queue during this interval.
	It follows that
	\begin{eqnarray*}  
	\lefteqn{
		\Prob \left( \exists s_0 \le t \; \mbox{ s.t. } 
				\tilde{W}^n_T(s_0) - \tilde{W}^n_S(s_0) \in \left(\frac{3 \varepsilon}{4},\epsilon\right) \mbox{ and }
				\tau^n(s_0) < \gamma^n(s_0)
		\right)  
	} \\ \nonumber
	&\le&
		\Prob \left( \exists s \le t \; \mbox{ s.t. } 
				\tilde{W}^n_T(s) - \tilde{W}^n_S(s) \ge  \frac{ \varepsilon}{2} \mbox{ and }
				\tilde{Q}^n_T(s) - \tilde{Q}^n_S(s) < 0
		\right)  	\\  
	&\le& 
		\Prob \left( 
			\sup_{s \in [0,t]} 
			\left |  \frac{\tilde{Q}^n_T(s)}{\mu} - \tilde{W}^n_T(s) \right|
			>  \frac{\varepsilon}{4} 
		\right)
		+
		\Prob \left( 
			\sup_{s \in [0,t]} 
			\left |  \frac{\tilde{Q}^n_S(s)}{\mu} - \tilde{W}^n_S(s) \right|
			>  \frac{\varepsilon}{4} 
		\right),
	\end{eqnarray*}		
	where the second inequality is a consequence of the Triangle Inequality.
	Applying Proposition \ref{prop:queueWorkload} twice yields 
	\begin{equation} \label{eq:thm:asymptoticCoupling2.2.2}
		\Prob \left( \exists s_0 \le t \; \mbox{ s.t. } 
				\tilde{W}^n_T(s_0) - \tilde{W}^n_S(s_0) \in \left(\frac{3 \varepsilon}{4},\epsilon\right) \mbox{ and }
				\tau^n(s_0) < \gamma^n(s_0)
		\right)  
		< 
		\frac{\eta}{4}.
	\end{equation}
	Hence, \eqref{eq:thm:asymptoticCoupling2.2} follows from  
		\eqref{eq:thm:asymptoticCoupling2.2.0}--\eqref{eq:thm:asymptoticCoupling2.2.2};
		 \eqref{eq:thm:asymptoticCoupling2} follows from  
		 \eqref{eq:thm:asymptoticCoupling2.1}--\eqref{eq:thm:asymptoticCoupling2.3}; and
		 \eqref{eq:thm:asymptoticCoupling1} follows from Proposition \ref{prop:queueWorkload},  		
		 	\eqref{eq:thm:asymptoticCoupling2}, and the Triangle Inequality. 
\end{proof}
\noindent Theorem \ref{thm:asymptoticCoupling} allows us to focus all of our efforts in proving that one of these sequences -- say  $\{\tilde{Q}^n_T, n \ge 1\}$ -- converges because the other sequence is brought along with it.

The service allocation process $T^n$ converges to the identity function.
\begin{proposition}{Convergence of the allocation process.} \label{prop:allocationLimit}
	Under the assumptions of Theorem \ref{thm:main}, for each $\alpha \in \{S,T\}$, and 
	any $\varepsilon, \eta, t > 0$,
	$$
		\limsup_{n \to \infty}
		 \Prob \left(
		 	\sup_{s \in [0,t]}
			\left| 
				T^n_\alpha(s)
				-
				s
			\right| 
			> \varepsilon
		\right)
		< \eta.
	$$
\end{proposition}
A further implication of this result is that the sequence of idle time processes converges to zero.

\subsection{Simplifying the reneging process}

Reneging can happen only if the associated job actually joins the queue.
This complication makes for involved expressions for the reneging processes in \eqref{eq:mod:renegingT^n} and \eqref{eq:mod:renegingS^n}.
Furthermore, notice that the expressions include both the queue length process as well as the workload process.
The process $\tilde{\epsilon}^n_\alpha$ allows one to replace the workload process with the queue length process, 
	to ignore whether or not the jobs have balked when considering whether they will abandon, 
	and to ignore the timing of when reneging is detected by the system.
The following proposition justifies this approximation.

For each $n \ge 1$ and $\alpha$, we define the processes $R^{0,n}_\alpha = \{R^{0,n}_\alpha(t), t \ge 0\}$, 
	$R^{1,n}_\alpha = \{R^{1,n}_\alpha(t), t \ge 0\}$, and
	$R^{2,n}_\alpha = \{R^{2,n}_\alpha(t), t \ge 0\}$, where
\begin{equation} \label{eq:R0}
	R^{0,n}_\alpha(t) 
	=
	\sum^{A^n(t)}_{i=1} 1(d_i \le Q^n_\alpha(t^n_i-)/ \mu^n),
\end{equation}
\begin{equation} \label{eq:R1}
	R^{1,n}_\alpha(t) 
	=
	\sum^{A^n(t)}_{i=1} 1(d_i \le W^n_\alpha(t^n_i-) ) 
\end{equation}
and
\begin{equation} \label{eq:R2}
	R^{2,n}_\alpha(t) 
	=
	\sum^{A^n(t)}_{i=1} 1(b_i > Q^n_\alpha(t^n_i-)/\mu^n ) \cdot 1(d_i \le W^n_\alpha(t^n_i-) ).
\end{equation}
Working backwards, notice that $R^{2,n}_\alpha$ is the fictitious abandonment process where the abandonment 
	takes place upon arrival.
Next, $R^{1,n}_\alpha$ is the fictitious abandonment process whereby abandonment happens upon arrival and
	jobs may abandon even if they have already balked.
Finally, $R^{0,n}_\alpha$ is the process by which abandonment happens upon arrival, is independent of the balking process,
	and is based on the weighted queue length upon arrival instead of the workload upon arrival.
The diffusion scaled analogs have the form
$\tilde R^{k,n}_\alpha = \{\tilde R^{k,n}_\alpha(t), t \ge 0\}$, where $\tilde R^{k,n}_\alpha(t) = (1/\sqrt{n}) R^{k,n}_\alpha(t)$ 
	for each $k=0,1,2$ and $\alpha \in \{S,T\}$.
Ultimately, we would like to replace $R^n_\alpha$ with $R^{0,n}_\alpha$; see Proposition \ref{prop:renegingIngnoresBalking}.
The following three propositions arrive at that conclusion progressively.

The first result verifies that abandonment may be treated as taking place upon arrival.
\begin{proposition} \label{prop:R2R}
	For any $\epsilon, \eta$ and $t >0$,
 \begin{eqnarray*}
\limsup_{n \to \infty} \mathbb{P}\left(
	\sup_{s \in [0,t]} \left(
		\tilde{R}^{2,n}_\alpha(s) - \tilde{R}^{n}_\alpha(s) 
	\right)
	> \epsilon \right) 
	< \eta.
\end{eqnarray*}
\end{proposition}

An immediate consequence of the next result is that effectively, no customer could have abandoned
	if their balking random variable was small enough to cause it to balk as well.
That is, the number of jobs that are candidates for both balking and abandonment is asymptotically negligible. 
\begin{proposition} \label{prop:R1R2}
	For any $\epsilon, \eta$ and $t >0$,
 \begin{eqnarray*}
\limsup_{n \to \infty} \mathbb{P}\left(
	\sup_{s \in [0,t]} \left(
		\tilde{R}^{1,n}_\alpha(s) - \tilde{R}^{2,n}_\alpha(s) 
	\right)
	> \epsilon \right) 
	< \eta.
\end{eqnarray*}
\end{proposition}

Now, we show that abandonment can be based on the weighted queue length as opposed to the workload.
\begin{proposition} \label{prop:R0R1}
	For any $\epsilon, \eta$ and $t >0$,
 \begin{eqnarray*}
\limsup_{n \to \infty} \Prob \left(
	\sup_{s \in [0,t]} \left| 
		\tilde{R}^{0,n}_\alpha(s) - \tilde{R}^{1,n}_\alpha(s) 
	\right| 
	> \epsilon \right) 
	< \eta.
\end{eqnarray*}
\end{proposition}

\begin{proposition}{Reneging among initial jobs is negligible.}
\label{prop:noInitialReneging}
	For every $\eta > 0$ and  each $\alpha \in \{S,T\}$, there exists an $L>0$ such that,
		 under the assumptions of Theorem \ref{thm:main},
	$$
		\limsup_{n \to \infty}
		\Prob
		\left(
			\hat{R}^n_\alpha(t) > L
		\right)
		<
		\eta.
	$$
\end{proposition}

\begin{proposition}{Reneging effectively ignores balking and takes place upon arrival.} \label{prop:renegingIngnoresBalking}
	For each $\alpha \in \{S,T\}$, and under the assumptions of Theorem \ref{thm:main},
	$$
		\tilde{\epsilon}^n_\alpha \rightarrow 0
	$$
in probability as $n \to \infty$.
\end{proposition}

\begin{proof}
	Notice that for any $t \ge 0$ and $\alpha \in \{S,T\}$ that 
	$$
		\tilde \epsilon^n_\alpha(t) 
			=  
		\frac{\hat{R}^n_\alpha(t)}{\sqrt{n}}
			+	
		\left( \tilde R^n_\alpha(t) - \tilde R^{2,n}_\alpha(t) \right)
			+
		\left( \tilde R^{2,n}_\alpha(t) - \tilde R^{1,n}_\alpha(t) \right)
			+
		\left( \tilde R^{1,n}_\alpha(t) - \tilde R^{0,n}_\alpha(t) \right).
	$$
	The result is an immediate consequence of Propositions \ref{prop:R2R}, \ref{prop:R1R2},   \ref{prop:R0R1},
		and \ref{prop:noInitialReneging}.
\end{proof}

\subsection{Tightness}
Next, we argue that the scaled processes are tight. 
This result is a key step in proving Proposition \ref{prop:deltaNegligible},
	which in turn is used to extract the restorative drift of the limiting diffusion process.
\begin{proposition}{Tightness of the scaled processes.} \label{prop:tightness}
	The processes $\{\tilde{Q}^n_\alpha, n \ge 1\}$ and $\{\tilde{W}^n_\alpha, n \ge 1\}$ are tight. 
\end{proposition}
\begin{proof}
By Theorem 13.2 of \cite{ref:B}, tightness follows from Lemma  \ref{lem:bounded} 
	and the fact that for any $\epsilon>0$, we have that for either $\alpha \in \{S,T\},$
\begin{equation} \label{eq:tightnessQ}
	\lim_{\delta \to 0}
	\limsup_{n \to \infty}
	\Prob \left(
		\sup_{u,v \in [0,t], v-u<\delta}
		\left|
			\tilde{Q}^n_\alpha(v) - \tilde{Q}^n_\alpha(u)
		\right|
		> \epsilon
	\right) = 0
\end{equation}
and
\begin{equation} \label{eq:tightnessW}
	\lim_{\delta \to 0}
	\limsup_{n \to \infty}
	\Prob \left(
		\sup_{u,v \in [0,t], v-u<\delta}
		\left|
			\tilde{W}^n_\alpha(v) - \tilde{W}^n_\alpha(u)
		\right|
		> \epsilon
	\right)
	= 0.
\end{equation}
We will show that \eqref{eq:tightnessW} holds.
Then \eqref{eq:tightnessQ} follows from Theorem \ref{prop:queueWorkload} and \eqref{eq:tightnessW}.
This will conclude the proof.

Fix $\alpha \in \{S,T\}$ and arbitrary constants $\epsilon, \eta>0$.
We will show that there exists a $\delta_0>0$ such that for any $\delta \le \delta_0$,
\begin{equation} \label{eq:tightnessW.1}
	\Prob \left(
		\sup_{u,v \in [0,t], v-u<\delta}
		\left|
			W^n_\alpha(v) - W^n_\alpha(u)
		\right|
		> \frac{\epsilon}{\sqrt{n}}
	\right)
	<
	\eta
\end{equation}
for sufficiently large $n$.
Changes in the workload are due to the arrival of work and the processing of work. 
That is, for any $v\ge u \ge 0$,
\begin{equation}  \label{eq:tightnessW.2.1}
	W^n_\alpha(v) - W^n_\alpha(u) 
	=
	\sum_{i=A^n(u)+1}^{A^n(v)}
		v^n_i 
		\cdot
		1(b_i > Q^n_\alpha(t^n_i-)/\mu^n)
		\cdot
		1(d_i > W^n_\alpha(t^n_i-))
	-
	\left( T^n_\alpha(v) - T^n_\alpha(u) \right).
\end{equation}
By \eqref{eq:mod:idletime^n}, the change in the service allocation process 
	can be written in terms of the increase in the idle process, 
	which in turn can be bounded by the shortfall in the arrival of workload relative to the potential server effort.
For any $v\ge u \ge 0$,
\begin{eqnarray}  \label{eq:tightnessW.2.2}
	I^n_\alpha(v) - I^n_\alpha(u)
	&=&
	(v-u) - \left( T^n_\alpha(v) - T^n_\alpha(u) \right) \\
	&\le& \nonumber
	\max \left(
		0, - \inf_{s\in [u,v]} \left(
			\sum_{i=A^n(u)+1}^{A^n(s)}
			v^n_i 
			\cdot
			1(b_i > Q^n_\alpha(t^n_i-)/\mu^n)
			\cdot
			1(d_i > W^n_\alpha(t^n_i-))
			- (s-u)
		\right)
	\right).
\end{eqnarray}
It follows from \eqref {eq:tightnessW.2.1} and \eqref {eq:tightnessW.2.2} that for any $0 < \delta \le \delta_0$,
\begin{eqnarray}  \label{eq:tightnessW.2}
	\sup_{u,v \in [0,t], v-u<\delta}
		\left|
			W^n_\alpha(v) - W^n_\alpha(u)
		\right|
	&\le&
	2 \sup_{u,v \in [0,t], v-u<\delta}
		\left|
		\sum_{i=A^n(u)+1}^{A^n(v)}
			v^n_i 
		-
		(v-u)
		\right|  \\ \nonumber
	&&+
	2 \sup_{u \in [0,t]}	
		\sum_{i=A^n(u)+1}^{A^n(u+\delta)}
			v^n_i  \left( 1(b_i \le Q^n_\alpha(t^n_i-)/\mu^n) + 1(d_i \le W^n_\alpha(t^n_i-)) \right) \\ \nonumber
	&\le&
	2 \sup_{u,v \in [0,t], v-u<\delta_0}
		\left|
		\sum_{i=A^n(u)+1}^{A^n(v)}
			v^n_i 
		-
		(v-u)
		\right|  \\ \nonumber
	&&+
	2 \sup_{u \in [0,t]}	
		\sum_{i=A^n(u)+1}^{A^n(u+\delta_0)}
			v^n_i  \left( 1(b_i \le Q^n_\alpha(t^n_i-)/\mu^n) + 1(d_i \le W^n_\alpha(t^n_i-)) \right).
\end{eqnarray}	
Finally,
\begin{eqnarray*}
\lefteqn{
	\Prob \left(
		\sup_{u,v \in [0,t], v-u<\delta}
		\left|
			W^n_\alpha(v) - W^n_\alpha(u)
		\right|
		> \frac{\epsilon}{\sqrt{n}}
	\right)
} && \\ 
	&\le&
	\Prob \left( 
		2 \sup_{u,v \in [0,t], v-u<\delta_0}
		\left|
		\sum_{i=A^n(u)+1}^{A^n(v)}
			v^n_i 
		-
		(v-u)
		\right|
		>
		\frac{\epsilon}{2\sqrt{n}}
	\right) \\
	&&+
	\Prob \left(
		2 \sup_{u \in [0,t]}	
		\sum_{i=A^n(u)+1}^{A^n(u+\delta_0)}
			v^n_i  \left( 1(b_i \le Q^n_\alpha(t^n_i-)/\mu^n) + 1(d_i \le W^n_\alpha(t^n_i-)) \right)
		>
		\frac{\epsilon}{2\sqrt{n}}
	\right) 	\\
	&\le&
	\Prob \left( 
		2 \sup_{u,v \in [0,t], v-u<\delta_0}
		\left|
		\sum_{i=A^n(u)+1}^{A^n(v)}
			v^n_i 
		-
		(v-u)
		\right|
		>
		\frac{\epsilon}{2\sqrt{n}}
	\right) 
	+ 
	\Prob \left( \sup_{s \le t + \delta_0} Q^n_\alpha(s)/\mu^n > K/\sqrt{n} \right) \\
	&&+
	\Prob \left(
		2 \sup_{u \in [0,t]}	
		\sum_{i=A^n(u)+1}^{A^n(u+\delta_0)}
			v^n_i  \left( 1(b_i \le  K/\sqrt{n})+ 1(d_i \le  K/\sqrt{n})  \right)
		>
		\frac{\epsilon}{2\sqrt{n}}
	\right) 	
	+
	\Prob \left( \sup_{s \le t + \delta_0} W^n_\alpha(s) > K/\sqrt{n} \right).	
\end{eqnarray*}
By Lemma \ref{lem:shortPotentialWork},
	Lemma \ref{lem:bounded}, Proposition  \ref{prop:boundedLostWorkload_shortInterval},
	and \eqref{eq:tightnessW.2}, we can choose a $K>0$ and a $\delta_0 >0$ such that \eqref{eq:tightnessW.1}  follows easily.

\end{proof}

\subsection{Convergence to Diffusion Processes}
In this subsection, we pull together the last ingredients for our diffusion process.  
We start with the driving Browning motions from our centered and scaled interarrival and service time processes.
We derive the restorative drift term of our limiting diffusion process from the 
	derivatives of the balking and abandonment distributions evaluated at zero.
Finally, we complete the proof of our main result.

Let $B_a$ and $B_s$ be independent, standard Brownian motions; i.e., $B_a(0) = B_s(0) = 0$ and the processes have zero drift and unitary infinitesimal variance.
The diffusion scaled arrival processes and service completion processes converge to scaled versions of these Brownian motions.
These are standard results -- see, for example, \cite{ref:JR}.
\begin{proposition} \label{prop:convergenceToBMs}
	Under the assumptions of Theorem \ref{thm:main},
		$$
			\tilde{A}^n \Rightarrow \mu \sigma_a B_a
			\quad
			\mbox{ and }
			\quad
			\tilde{S}^n_\alpha \Rightarrow \mu \sigma_s B_s
		$$
	as $n \to \infty$.
\end{proposition}

The process $\tilde{\delta}^n_\alpha$ swaps the sum of the abandonment and balking distributions, 
	evaluated at the scaled queue length at the times of arrivals, with a smooth function involving the derivatives of the
	distribution functions evaluated at zero and then multiplied by the scaled queue lengths.
The following proposition justifies this step and is proven in the Appendix.
\begin{proposition} \label{prop:deltaNegligible}
	For each $\alpha \in \{S,T\}$, and under the assumptions of Theorem \ref{thm:main},
	$$
		\tilde{\delta}^n_\alpha \rightarrow 0
	$$
in probability as $n \to \infty$.
\end{proposition}

We conclude with a proof of the main result.
\begin{proofof}{Theorem \ref{thm:main}}
Fix $\alpha \in \{S,T\}$.
Recall the expression for the scaled queue length and idleness process given in \eqref{eq:queueIdle}.
The elements of the process $\tilde{X}^n_\alpha$ (see \eqref{eq:mod:X^tilde})  
	converge to either zero, a drift, or a Brownian motion.
By Lemma \ref{lem:uniformArrivals} and Proposition \ref{prop:centeredBalkingReneging}, we have that
$$
	\tilde{M}^n_{b,\alpha} \circ \bar{A}^n \to 0
	\quad \mbox{and} \quad
	\tilde{M}^n_{d,\alpha} \circ \bar{A}^n \to 0	
$$
	as $n \to \infty$.
Furthermore, Propositions \ref{prop:renegingIngnoresBalking} and \ref{prop:deltaNegligible} inform us that
$$
	\tilde{\epsilon}^n_\alpha \rightarrow 0
	\quad \mbox{and} \quad	
	\tilde{\delta}^n_\alpha \rightarrow 0,
$$
	respectively, as $n \to \infty$.
Proposition \ref{prop:convergenceToBMs} provides the convergence of the scaled and centered arrival and potential departure processes.
Coupled with the service allocation process, we have
$$
	\tilde{A}^n + \tilde{S}^n_\alpha \circ T^n_\alpha 
	\rightarrow
	\mu \sigma_a B_a 
	+
	\mu \sigma_b B_b \
	{\buildrel D \over = } \
	\sigma B.
$$
Finally, \eqref{eq:limiting beta} provides the drift term
$$
	\frac{ \left(\lambda^n-\mu^n \right) }{\sqrt{n}}
	\to
	\beta
$$
as $n \to \infty$.
Hence,
$$	
	\tilde{X}^n_\alpha \Rightarrow \tilde{X}.
$$
From \eqref{eq:thm:main} we have that
$$
	(\tilde{Q}^n(0), \tilde{X}^n_\alpha) 
	\Rightarrow 
	(\tilde{Q}_0, \tilde{X}), \qquad \mbox{as } n \to \infty,
$$
and hence, by the Continuous Mapping Theorem,
$$
	(\tilde{Q}^n_\alpha, \tilde{Y}^n_\alpha) 
	=
	(\Phi_\theta, \Psi_\theta) (\tilde{Q}^n(0), \tilde{X}^n_\alpha)
	\Rightarrow
	(\Phi_\theta, \Psi_\theta) (\tilde{Q}(0), \tilde{X}) = (\tilde{Q},\tilde{Y}), \qquad \mbox{as } n \to \infty,
$$
	where recall $\tilde{Q}(0)$ is equal in distribution to $\tilde{Q}_0$.
	
	As $\alpha$ was chosen arbitrarily, this limit holds for both the ticket queue- and standard queue-related processes;
		both scaled queue length processes converge to the same limit in distribution.
	By \eqref{eq:thm:asymptoticCoupling1} of Theorem \ref{thm:asymptoticCoupling}, these processes are coupled:
$$
	( (\tilde{Q}^n_S, \tilde{Y}^n_S), (\tilde{Q}^n_T, \tilde{Y}^n_T)) 
	\Rightarrow
	((\tilde{Q},\tilde{Y}), (\tilde{Q},\tilde{Y})), \qquad \mbox{as } n \to \infty,
$$
	Finally, Proposition \ref{prop:queueWorkload} demonstrates that asymptotically the scaled workload and scaled queue lengths are  
		scalar multiples of each other and thus converge together to scalar multiples of the same process. Hence \eqref{eq:thm:main2} holds.
	This concludes the proof.
\end{proofof}

\section{Numerical results} \label{sec:numerics}
	We have tested the heuristics forwarded in this paper extensively.  Tables capturing these tests are given in the Appendix.  We describe some of the highlights and trends in the results here.  In Tables \ref{tab:ticket_1} - \ref{tab:ticket_4}, we see that our approximations generated from the heavy traffic limit theorems are quite good, especially when the arrival and service rates are large i.e. $\mu \approx$ 100.  In Table \ref{tab:ticket_1}, we simulate a scenario where all of the parameters are generated from exponential distributions.  We see that the approximations are good and accurate, except when the rates are small and the drift parameter $\beta$ is very negative.   

In Table \ref{tab:ticket_2}, we simulate a non-Markovian model where the service rate follows a lognormal distirbution and the balking and reneging distributions follow a uniform distribution.  Unlike the exponential distribution, the uniform distribution is bounded and we see that the simulated values in Table \ref{tab:ticket_1} are very similar to the values of Table \ref{tab:ticket_2} even though they differ in the types of distributions that generate the queueing dynamics.  

In Table \ref{tab:ticket_3} we simulate a scenario where all of the parameters are generated from exponential distributions.  This table is different than Table \ref{tab:ticket_1} since we use different balking and reneging parameter values.  Once again we see that the approximations are good and accurate, except when the rates are small and the drift parameter $\beta$ is very negative.  In Table \ref{tab:ticket_4} we simulate a non-Markovian model where the service rate follows a lognormal distirbution and the balking and reneging distributions follow a uniform distribution.  Once again the simulated results of Table \ref{tab:ticket_3} are very similar to Table \ref{tab:ticket_4} even though the distributions are not Markovian in Table \ref{tab:ticket_4}.  

In addition to our heavy traffic approximations accurately estimating the performance measures, we also notice that in every simulation the ticket queue is larger than the standard queue.  This fact is irrespective of the distributions that are used to generate the queueing dynamics.  Moreover, the fraction of balking in the ticket queue is also larger than the standard queue.  Even though these two processes converge as the rates tend to infinity, when the rates are finite, the ticket queue is perceived as being larger, which causes more people to balk from the system.  This difference in queue length is also a function of the parameter $\beta$.  When the $\beta$ parameter is larger and positive the difference is larger than when $\beta$ is negative.  Thus, when the rates are not infinite, the ticket queue exhibits interesting behavior that is to be expected.  
	
\section{Conclusions and extensions}
In this paper we studied the dynamics of a critically loaded queueing system subject to customers who might either balk 
	because the line is too long or abandon the queue from waiting too long in-queue.
We consider two types of abandonment protocols.
In the conventional approach to capturing abandonment, 
	customers immediately leave the system when their patience time has been 
	exceeded by their time waiting for service; we refer to the model in this setting as the standard queue.
In the standard queue, everyone is aware of an abandoning customer's departure at the time of abandonment. 
We compare the standard queue to the ticket queue.  In the ticket queue, customers whose patience has run out leave the queue in an unnoticed fashion. 
Their departure is only detected when their hypothetical service time would have begun.  The paper is complementary to the study of \cite{ref:XGO} who study a heavily loaded system with 
	impatient customers; in comparison, our customers are relatively patient. 
Our method of analysis is also fundamentally different from that of \cite{ref:XGO}.

We prove a heavy traffic limit theorem for the diffusion scaled queue length and workload processes.
A key result in the theorem is that the standard and ticket queues are asymptotically coupled under diffusion scaling.
The managerial interpretation is that regardless of how you implement your queue -- whether with a physical line, which
	is best modeled with a standard queue, or as a ticket queue -- the dynamics of the systems will not differ by much.
In addition to this insight based on the sensitivity of the diffusion scaling, we provide some heuristics for calculating
	certain performance metrics of operational importance. 
These heuristics are beyond the sensitivity of diffusion scaling. 
Nonetheless, we assess the accuracy of these heuristics through simulation. 
We find that in a broad range of parameter and distributional settings, the heuristics perform well.

\bibliographystyle{abbrvnat}
\bibliography{jp}

\section{Appendix}
The Appendix is split into two parts.  The first half provides proofs of results stated earlier in the paper.  The second half provides an extensive collection of numerical examples. 


\subsection{Proofs}
This Appendix contains many of the proofs from Section \ref{sec:proof}.
We start with a proof of Proposition \ref{prop:boundedBalkingAbandonment_shortInterval}.
\begin{proofof}{Proposition \ref{prop:boundedBalkingAbandonment_shortInterval}}
	Fix $\epsilon, \eta, t, K>0$ and set $\delta = \frac{\epsilon}{4 \mu  K \theta}$.  By Lemma \ref{lem:theta} and our choice of $\delta$, we have that
$$
	\mu n \delta 
	\left(
		F_b(K \sqrt{n}) + F_d(K/\sqrt{n})
	\right)
	<
	2 \mu \sqrt{n} K \theta \delta
	\le
	\frac{\epsilon \sqrt{n}}{2}
$$
	for sufficiently large $n$. 
It follows that, by Lemmas \ref{lem:boundedArrivals} and \ref{lem:uniformArrivals},
\begin{eqnarray}  \label{eq:proof:prop:boundedBalkingAbandonment_shortInterval.1}
\lefteqn{
		\Prob \left( 
			\sup_{s \le t}
			\sum_{i = A^n(s)+1}^{A^n(s+\delta)}
			\left(
				1(b_i \le K/\sqrt{n}) + 1(d_i \le K/\sqrt{n}) 
			\right)
			> \epsilon \sqrt{n}
		\right)
} &&\\
	&\le&\nonumber
		\Prob \left( 
			\sup_{0 \le j \le \left\lfloor 2\mu nt \right\rfloor}
			\sum_{i = j+1}^{j+ \left\lfloor 2\mu n \delta\right\rfloor}
			\left(
				1(b_i \le K/\sqrt{n}) + 1(d_i \le K/\sqrt{n}) 
			\right)
			> \epsilon \sqrt{n}
		\right) \\ \nonumber
	&&+ 
		\Prob \left(A^n(t) > 2 \mu n t\right) 
		+
		\Prob \left( \sup_{s\le t} \left( A^n(s+\delta) - A^n(s) \right) > 2\mu n \delta\right) \\
	&\le&\nonumber
		\frac{\eta}{2}
		+		
		\Prob \left( 2
			\sup_{0 \le j \le \left\lfloor 2 t/\delta \right\rfloor +1}
			\sum_{i = j\left\lfloor 2 \mu n \delta\right\rfloor + 1}^{(j+1)\left\lfloor 2\mu n \delta\right\rfloor}
			\left(
				1(b_i \le K/\sqrt{n}) - F_b(K/\sqrt{n}) 
			\right)
			> \frac{\epsilon \sqrt{n}}{6}
		\right) \\
	&&+\nonumber
		\Prob \left( 2
			\sup_{0 \le j \le \left\lfloor 2 t/\delta \right\rfloor +1}
			\sum_{i = j\left\lfloor 2 \mu n \delta\right\rfloor + 1}^{(j+1)\left\lfloor 2\mu n \delta\right\rfloor}
			\left(
				1(d_i \le K/\sqrt{n}) - F_d (K/\sqrt{n}) 
			\right)
			> \frac{\epsilon \sqrt{n}}{6}
		\right) \\
&&+\nonumber		
		\Prob \left( 2
			\sup_{0 \le j \le \left\lfloor 2 t/\delta \right\rfloor +1}
			\sum_{i = j\left\lfloor 2 \mu n \delta\right\rfloor + 1}^{(j+1)\left\lfloor 2\mu n \delta\right\rfloor}
			 F_b(K/\sqrt{n}) +  F_d (K/\sqrt{n}) 
			> \frac{\epsilon \sqrt{n}}{6}
		\right) 
\end{eqnarray}
for sufficiently large $n$.
\end{proofof}

We will show that, for sufficiently large $n$,
\begin{equation} \label{eq:proof:prop:boundedBalkingAbandonment_shortInterval.1.1}
		\Prob \left( 2
			\sup_{0 \le j \le \left\lfloor 2 t/\delta \right\rfloor +1}
			\sum_{i = j\left\lfloor \mu n \delta\right\rfloor + 1}^{(j+1)\left\lfloor \mu n \delta\right\rfloor}
			\left(
				1(b_i \le K/\sqrt{n}) - F_b(K/\sqrt{n}) 
			\right)
			> \frac{\epsilon \sqrt{n}}{6}
		\right)
	<
	\frac{\eta}{6},
\end{equation}
\begin{equation} \label{eq:proof:prop:boundedBalkingAbandonment_shortInterval.1.1b}
		\Prob \left( 2
			\sup_{0 \le j \le \left\lfloor 2 t/\delta \right\rfloor +1}
			\sum_{i = j\left\lfloor \mu n \delta\right\rfloor + 1}^{(j+1)\left\lfloor \mu n \delta\right\rfloor}
			 F_b(K/\sqrt{n}) + F_d(K/\sqrt{n}) 
			> \frac{\epsilon \sqrt{n}}{6}
		\right)
	<
	\frac{\eta}{6},
\end{equation}
and by symmetry it will also follow that
\begin{equation}  \label{eq:proof:prop:boundedBalkingAbandonment_shortInterval.1.2}
		\Prob \left( 2
			\sup_{0 \le j \le \left\lfloor 2 t/\delta \right\rfloor +1}
			\sum_{i = j\left\lfloor \mu n \delta\right\rfloor + 1}^{(j+1)\left\lfloor \mu n \delta\right\rfloor}
			\left(
				1(d_i \le K/\sqrt{n}) - F_d(K/\sqrt{n}) 
			\right)
			> \frac{\epsilon \sqrt{n}}{6}
		\right)
	<
	\frac{\eta}{6}.
\end{equation}
Consider \eqref{eq:proof:prop:boundedBalkingAbandonment_shortInterval.1.1}.
By Kolmogorov's inequality (see, e.g., \cite{ref:Bass}), 
\begin{eqnarray*} 
\lefteqn{
		\Prob \left( 2
			\sup_{0 \le j \le \left\lfloor 2 t/\delta \right\rfloor +1}
			\sum_{i = j\left\lfloor \mu n \delta\right\rfloor + 1}^{(j+1)\left\lfloor \mu n \delta\right\rfloor}
			\left(
				1(b_i \le K/\sqrt{n}) - F_b(K/\sqrt{n}) 
			\right)
			> \frac{\epsilon \sqrt{n}}{6}
		\right)
}
	&& \\ \nonumber
 	&\le& \nonumber
		\left( \frac{2  t}{\delta} +2\right)
		\left( \frac{144}{\epsilon^2 n} \right)
		\Exp \left[
			\left( 
			\sum_{i = 1}^{\left\lfloor \mu n \delta\right\rfloor}
			\left(
				1(b_i \le K/\sqrt{n}) - F_b(K/\sqrt{n}) 
			\right)
			\right)^2
		\right].
\end{eqnarray*}
The indicators above are independent, so the cross terms  all have expectations of zero. 
Hence
\begin{eqnarray*} 
\lefteqn{
		\Prob \left( 2
			\sup_{0 \le j \le \left\lfloor 2 \mu nt/\delta \right\rfloor +1}
			\sum_{i = j\left\lfloor \mu n \delta\right\rfloor + 1}^{(j+1)\left\lfloor \mu n \delta\right\rfloor}
			\left(
				1(b_i \le K/\sqrt{n}) - F_b(K/\sqrt{n}) 
			\right)
			> \frac{\epsilon \sqrt{n}}{6}
		\right)
}
	&& \\ \nonumber
	&\le& \nonumber
		\left( \frac{2 t}{\delta} +2\right)
		\left( \frac{144}{\epsilon^2 n} \right)
		\Exp \left[
			\sum_{i = 1}^{\left\lfloor \mu n \delta\right\rfloor}
			\left(
				1(b_i \le K/\sqrt{n}) - F_b(K/\sqrt{n}) 
			\right)^2
		\right] 
	\\ &<& \nonumber
		\left( \frac{2 \mu t}{\delta} +2\right)
		\left( \frac{144}{\epsilon^2 n} \right)
		\left(\mu n \delta\right)
		\left( \frac{2 K \theta}{\sqrt{n}} \right)
	\\ &\le& \nonumber
		\frac{\eta}{6}
\end{eqnarray*}
for sufficiently large $n$.
So \eqref{eq:proof:prop:boundedBalkingAbandonment_shortInterval.1.1} holds, as does
	\eqref{eq:proof:prop:boundedBalkingAbandonment_shortInterval.1.2} by symmetry.
The result  \eqref{eq:proof:prop:boundedBalkingAbandonment.0} 
	follows from \eqref{eq:proof:prop:boundedBalkingAbandonment_shortInterval.1}--\eqref{eq:proof:prop:boundedBalkingAbandonment_shortInterval.1.2}.

Next we prove Proposition \ref{prop:boundedLostWorkload_shortInterval}.
\begin{proofof}{Proposition \ref{prop:boundedLostWorkload_shortInterval}}
	We will alter \eqref{eq:proof:prop:boundedBalkingAbandonment_shortInterval} to 
		an expression that, without loss of generality, excludes the abandonment random variables.
	We will show that 
		for any $\eta, t >0$ and $K >0$, there exists an $\delta > 0$ such that
	\begin{equation} \label{eq:proof:prop:boundedLostWorkload_shortInterval.1}
		\Prob \left( 
			\sup_{s \le t}
			\sum_{i = A^n(s)+1}^{A^n(s+\delta)}
			v^n_i \cdot
				 1(b_i \le K / \sqrt{n}) 
			> \frac{\epsilon}{ \sqrt{n}}
		\right)
		< \eta,
	\end{equation}
	for sufficiently large $n$.
	Fix $\eta, t$ and $K$ and set $\delta = \frac{\epsilon}{12K \theta}$ and notice that by Lemma \ref{lem:theta},
	$$
		\left(\mu \delta n \right) \left( \frac{1}{\mu^n} \right) F_b(K/\sqrt{n})
		< 
		\frac{4 \delta K \theta}{\sqrt{n}}
		\le \frac{\epsilon}{3\sqrt{n}}
	$$
	for sufficiently large $n$.
	Proceeding in a manner similar to that of the proof of 
		Proposition \ref{prop:boundedBalkingAbandonment_shortInterval}, it follows
		by Lemmas \ref{lem:boundedArrivals} and \ref{lem:uniformArrivals} that
\begin{eqnarray} \label{eq:proof.1}
\lefteqn{
		\Prob \left( 
			\sup_{s \le t}
			\sum_{i = A^n(s)+1}^{A^n(s+\delta)}
			v^n_i \cdot
				 1(b_i \le K / \sqrt{n}) 
			> \frac{\epsilon}{ \sqrt{n}}
		\right)
}	&&\\ \nonumber
	&\le&
		\Prob \left( 
			\sup_{0 \le j \le \left\lfloor 2\mu nt \right\rfloor}
			\sum_{i = j+1}^{j+ \left\lfloor \mu n \delta\right\rfloor}
			v^n_i \cdot
				\left( 1(b_i \le K/\sqrt{n}) + 1(d_i \le K/\sqrt{n}) \right)
			>\frac{\epsilon}{ \sqrt{n}}
		\right) \\ \nonumber
	&&+ 
		\Prob \left(A^n(t) > 2 \mu n t \right) 
		+
		\Prob \left( \sup_{s\le t} \left( A^n(s+\delta) - A^n(s) \right) > 2\mu n \delta\right) \\
	&\le& \nonumber
	\frac{\eta}{4}
	+
		\Prob \left( 
			\sup_{0 \le j \le \left\lfloor 2\mu nt \right\rfloor}
			\sum_{i = j+1}^{j+ \left\lfloor \mu n \delta\right\rfloor}
			\left( v^n_i - \frac{1}{\mu^n} \right) \cdot
				1(b_i \le K/\sqrt{n}) 
			> \frac{\epsilon}{3 \sqrt{n}}
		\right)	\\
	&&+ \nonumber
		\Prob \left( 
			\sup_{0 \le j \le \left\lfloor 2\mu nt \right\rfloor}
			\sum_{i = j+1}^{j+ \left\lfloor \mu n \delta\right\rfloor}
			 \frac{1}{\mu^n}  \cdot
				\left( 1(b_i \le K/\sqrt{n}) - F_b(K/\sqrt{n}) \right)
			>\frac{\epsilon}{ 3 \sqrt{n}}
		\right)	
	\end{eqnarray}
	for sufficiently large $n$.
	Consider the third term on the far right hand side.
By Kolmogorov's inequality \cite{ref:Bass},
\begin{eqnarray} \label{eq:proof.2}
\lefteqn{
		\Prob \left( 
			\sup_{0 \le j \le \left\lfloor 2\mu nt \right\rfloor}
			\sum_{i = j+1}^{j+ \left\lfloor \mu n \delta\right\rfloor}
			 \frac{1}{\mu^n}  \cdot
				\left( 1(b_i \le K/\sqrt{n}) - F_b(K/\sqrt{n}) \right)
			>\frac{\epsilon}{ 3 \sqrt{n}}
		\right)	
} && \\ \nonumber
	&\le&
		\Prob \left( 2
			\sup_{0 \le j \le \left\lfloor 2 t/\delta \right\rfloor +1}
			\sum_{i = j\left\lfloor \mu n \delta\right\rfloor + 1}^{(j+1)\left\lfloor \mu n \delta\right\rfloor}
			\frac{1}{\mu^n} \cdot
			\left(
				1(b_i \le K/\sqrt{n}) - F_b(K/\sqrt{n}) 
			\right)
			> \frac{\epsilon}{3 \sqrt{n}}
		\right) \\
 	&\le& \nonumber
		\left( \frac{2  t}{\delta} +2\right)
		\left( \frac{36 n}{\epsilon^2} \right)
		\left(\frac{1}{\mu^n}\right)^2
		\Exp \left[
			\left( 
			\sum_{i = 1}^{\left\lfloor \mu n \delta\right\rfloor}
			\left(
				1(b_i \le K/\sqrt{n}) - F_b(K/\sqrt{n}) 
			\right)
			\right)^2
		\right] \\
 	&\le& \nonumber
		\left( \frac{2  t}{\delta} +2\right)
		\left( \frac{36 n}{\epsilon^2} \right)
		\left(\frac{2}{n \mu}\right)^2
		\left(n \mu \delta \right)
		F_b(K/\sqrt{n}) \\\nonumber
	&<&
		\left( 2t + 2 \delta \right) 
		\left( \frac{36 }{\epsilon^2} \right)
		\left( \frac{8K \theta}{\sqrt{n}} \right) \\
	&<& \nonumber
		\frac{\eta}{3}
\end{eqnarray}
for sufficiently large $n$.
As for the second term on the last right hand side of \eqref{eq:proof.1}, 
	applying Kolmogorov's inequality a second time yields
\begin{eqnarray} \label{eq:proof.3}
\lefteqn{
		\Prob \left( 
			\sup_{0 \le j \le \left\lfloor 2\mu nt \right\rfloor}
			\sum_{i = j+1}^{j+ \left\lfloor \mu n \delta\right\rfloor}
			\left( v^n_i - \frac{1}{\mu^n} \right) \cdot
				1(b_i \le K/\sqrt{n}) 
			> \frac{\epsilon}{3 \sqrt{n}}
		\right)	
} && \\ \nonumber
	&\le&
		\Prob \left( 2
			\sup_{0 \le j \le \left\lfloor 2 t/\delta \right\rfloor +1}
			\sum_{i = j\left\lfloor \mu n \delta\right\rfloor + 1}^{(j+1)\left\lfloor \mu n \delta\right\rfloor}
			\left( v^n_i - \frac{1}{\mu^n} \right) \cdot
				1(b_i \le K/\sqrt{n}) 
			> \frac{\epsilon}{3 \sqrt{n}}
		\right) \\
	&\le& \nonumber
		\left( \frac{2  t}{\delta} +2\right)
		\left( \frac{36 n}{\epsilon^2} \right)
		\Exp \left[
			\left(
			\left( v^n_i - \frac{1}{\mu^n} \right) \cdot
				1(b_i \le K/\sqrt{n}) 
			\right)^2
		\right] \\
	&\le& \nonumber
		\left( \frac{2  t}{\delta} +2\right)
		\left( \frac{36 n}{\epsilon^2} \right)
		\Exp \left[ 
			\left( v^n_i - \frac{1}{\mu^n} \right)^2 \cdot
				1(b_i \le K/\sqrt{n}) 
		\right],
\end{eqnarray}
where the last inequality follows because of the independence of the cross terms makes their expectations zero.
Further the service times are independent of the balking random variables. 
Hence,
\begin{eqnarray} \label{eq:proof.4}
\lefteqn{
		\left( \frac{2  t}{\delta} +2\right)
		\left( \frac{36 n}{\epsilon^2} \right)
		\Exp \left[ 
			\sum_{i = 1}^{\left\lfloor \mu n \delta\right\rfloor}
			\left( v^n_i - \frac{1}{\mu^n} \right)^2 \cdot
				1(b_i \le K/\sqrt{n}) 
		\right]
} && \\ \nonumber
	&\le& 
		\left( \frac{2  t}{\delta} +2\right)
		\left( \frac{36 n}{\epsilon^2} \right)
		\left(n  \mu \delta \right)
		\left( \frac{\sigma_b}{\mu^n} \right)^2
		F_b(K/\sqrt{n}) \\
	&<& \nonumber
		\left( 2  t + 2\delta \right)
		\left( \frac{36 }{\epsilon^2} \right)	
		\left( \frac{2 \sigma_b^2}{\mu} \right)
		\left(\frac{2 K \theta}{\sqrt{n}} \right) \\ \nonumber
	&\le& 
	\frac{\eta}{3}
\end{eqnarray}
for sufficiently large $n$.
The result \eqref{eq:proof:prop:boundedLostWorkload_shortInterval.1}
	 follows from  \eqref{eq:proof.1}--\eqref{eq:proof.4}. 
		
\end{proofof}

\begin{proofof}{Proposition \ref{prop:centeredBalkingReneging}}
It suffices to prove that the centered and scaled abandonment process is asymptotically negligible.
The analogous property for the balking process  can be proved in an identical fashion.
Fix $\varepsilon, \eta, t > 0$.
By Kolmogorov's inequality we have that
\begin{eqnarray*}
	\mathbb{P} 
		\left( \sup_{s \in [0,t]}  \tilde{M}^n_{d,\alpha}(s) > \epsilon \right) 
&=&
	\mathbb{P}  \left( \sup_{s \in [0,t]} M^n_{d,\alpha}(s)  > \sqrt{n} \cdot \epsilon \right) \\
& \leq &   
	\frac{1} {n \epsilon^2 }
	\Exp \left [ \left( M^n_{d,\alpha} (t) \right)^2 \right ] \\
& = & 
	\frac{1}{n \epsilon^2 } 
	\Exp \left [ 
		\left( \sum_{i=1}^{\lfloor nt \rfloor} 
			\left[1(d_i < Q^n_\alpha(t^n_i-)/\mu^n) - F_d(Q^n_\alpha(t^n_i-)/\mu^n) \right] 
		\right)^2 
	\right ]  
\end{eqnarray*}
Now by Burkholder's inequality and bounds for indicator functions, there exists a $c>0$ such that 
\begin{eqnarray*}
	\mathbb{P} 
		\left( \sup_{s \in [0,t]}  \tilde{M}^n_{d,\alpha}(s) > \epsilon \right) 
& \leq & 
	\frac{c}{n \epsilon^2 } 
	\Exp \left [ \sum_{i=1}^{\lfloor nt \rfloor} 
		\left[1(d_i < Q^n_\alpha(t^n_i-)/\mu^n) - F_d(Q^n_\alpha(t^n_i-)/\mu^n)\right]^2 
	\right ]  \\
& \leq & 
	\frac{c}{n \epsilon^2} 
 	\Exp \left [ \sum_{i=1}^{\lfloor nt \rfloor} 
		\left[1(d_i < Q^n_\alpha(t^n_i-)/\mu^n) + F_d(Q^n_\alpha(t^n_i-)/\mu^n)\right] 
	\right ]  \\
& \leq &
 	\frac{2 c t}{\epsilon^2}  
	\Prob \left(d_1 < \max_{s \le t}Q^n_\alpha(s)/\mu^n\right) \\
&<& \eta
\end{eqnarray*}
	for sufficiently large $n$.
The last inequality follows from Lemma \ref{lem:bounded}. 
This completes the proof.
\end{proofof}

\begin{proofof}{Proposition \ref{prop:queueWorkload}}
	Fix $\alpha \in \{S,T\}$ and $t >0$.
	For each $s\ge0,$ let $\breve{Q}^n_\alpha(s)$ denote the difference between the index of the last arriving job 
		 and the index of the job currently in service,
		plus any of the initial jobs present at time zero that remain in the system.
	The jobs present at time zero do not have indices.
	Note that, for each $s \ge 0$, $\breve{Q}^n_\alpha(s) \ge Q^n_\alpha(s)$, for both ticket and standard queues.
	The process $\breve{Q}^n_\alpha$ ignores the balking and abandonment that has taken place since the 
		arrival of the job currently in service. 
	One can think of the process as progressing in a manner similar to a ticket queue for which,
		in addition to the abandoned tickets,  balking is
		not accounted for until service would have begun for the departing job.
	We can bound the difference between the processes.
	Fix an arbitrary $\eta, \epsilon >0$ and choose an $L>0$ such that by 
		Lemma \ref{lem:bounded},
\begin{eqnarray}   \label{eq:lem:queueWorkload1}
\lefteqn{
	\Prob \left(
		\sup_{s \in [0,t]} \left(
			\breve{Q}^n_\alpha(s) - Q^n_\alpha(s) 
		\right)
		> 
		\frac{\epsilon \sqrt{n}}{3}
	\right)
} &&\\
	&\le& \nonumber
	\Prob \left(
		\sup_{s \in [0,t]} Q^n_\alpha(s) >  L \sqrt{n}
	\right)  
	+
	\Prob \left(
		\sup_{s \in [0,t]} W^n_\alpha(s) >  L / \sqrt{n}
	\right)	
	\\ \nonumber
	&&+ \; 
	\Prob \left(
		\sup_{j \le A^n(t)}
		\sum_{i=j+1}^{j+\lfloor (L + \epsilon/3) \sqrt{n} \rfloor} 
			\left(
			1\left(b_i \le L \sqrt{n} / \mu^n\right)
			+
			1\left(d_i \le L / \sqrt{n}  \right)
			\right)
		> 
		\frac{\epsilon \sqrt{n}}{3}
	\right) \\  \nonumber
	&<& 
	\frac{\eta}{5}
	+
	\Prob \left(
		\sup_{j \le A^n(t)}
		\sum_{i=j+1}^{j+\lfloor (L + \epsilon/3) \sqrt{n} \rfloor} 
			\left(
			1\left(b_i \le L  / \sqrt{n} \right)
			+
			1\left(d_i \le L / \sqrt{n} \right)
			\right)
		> 
		\frac{\epsilon \sqrt{n}}{3}
	\right)
\end{eqnarray}
	for sufficiently large $n$.
	To appreciate the above inequality, note that $\breve{Q}^n_\alpha$ is an inflated version of $Q^n_\alpha$.
	The jobs in the former not accounted for in the latter must have abandoned or balked, 
		or will have eventually abandoned.
	So if there exists an $s \in [0,t]$ such that $\breve{Q}^n_\alpha(s) - Q^n_\alpha(s)$ exceeds $(L+\epsilon/3) \sqrt{n}$ 
		and $Q^n_\alpha(u) \le L\sqrt{n}$
		for all $u \in [0,t]$,
		then there must be at least $\epsilon \sqrt{n}/3$ abandoned or balked jobs within some $(L + \epsilon/3) \sqrt{n}$ 
		consecutively arriving jobs.
	We place upper bounds on the queue length and workload and this makes our abandonment and balking
		indicators i.i.d. random variables. 
	
	For any $\delta >0$ it is true that $(L+\epsilon/3) \sqrt{n} < \mu \delta n$ for sufficiently large $n$.
	Hence, by Proposition  \ref{prop:boundedBalkingAbandonment_shortInterval},
	\begin{eqnarray}   \label{eq:lem:queueWorkload1.1}
	\lefteqn{
	\Prob \left(
		\sup_{j \le A^n(t)}
		\sum_{i=j+1}^{j+\lfloor (L + \epsilon/3) \sqrt{n} \rfloor} 
			\left(
			1\left(b_i \le L / \sqrt{n} \right)
			+
			1\left(d_i \le L / \sqrt{n} \right)
			\right)
		> 
		\frac{\epsilon \sqrt{n}}{3}
	\right)	
	} && \\ \nonumber
	&\le& \Prob \left(
		\sup_{j \le A^n(t)}
		\sum_{i=j+1}^{j+ \lfloor \mu \delta n \rfloor} 
			\left(
			1\left(b_i \le L / \sqrt{n} \right)
			+
			1\left(d_i \le L / \sqrt{n} \right)
			\right)
		> 
		\frac{\epsilon \sqrt{n}}{3}
	\right) \\ \nonumber
	&\le& \Prob \left(
		\sup_{s \le t}
		\sum_{i=A^n(s)+1}^{A^n(s+2\delta)}
			\left(
			1\left(b_i \le L / \sqrt{n} \right)
			+
			1\left(d_i \le L / \sqrt{n} \right)
			\right)
		> 
		\frac{\epsilon \sqrt{n}}{3}
	\right) \\ \nonumber
	&&+
	\Prob \left(
		\inf_{s \le t}
			\left(A^n(s+2\delta) - A^n(s) \right)
		<  \mu \delta n 
	\right) \\ \nonumber
	&<& 
	\frac{\eta}{5}.
	\end{eqnarray}

	Additionally, define for each $s \ge 0$ the quantity $\breve{W}^n_\alpha(s)$ that tracks the service times
		of the jobs associated with $\breve{Q}^n_\alpha(s)$.
	This process has the {\it entire} service time of the job currently in service and the service times of all jobs that arrive 
		after the arrival time of the job in service; 
		jobs that abandon or balk   contribute the process $\breve{W}^n_\alpha$ nonetheless.
	Just as with $Q^n_\alpha$ and its augmented version $\breve{Q}^n_\alpha$, 
		it is also the case that $\breve{W}^n_\alpha(s) \ge W^n_\alpha(s)$ for each $s \ge 0$.
	Unlike the process $W^n_\alpha$, the augmented process $\breve{W}^n_\alpha$ experiences both upward and
		 downward jumps.
	Upward jumps are the size of the would-be service time of each arriving job, even those that balk or abandon,
		and occur at the time of arrival of the corresponding job.
	The downward jumps occur at service completion times. 
	The downward jump size is equal to the service time of the job that was in service plus
		the would-be service times of jobs that arrived between the arrival time of the job just served and the
		arrival time of the next job to be served. 
	Just as we did in \eqref{eq:lem:queueWorkload1} for the queue lengths, we can bound the difference between
		these proceses.
	By Lemmas \ref{lem:boundedArrivals}, \ref{lem:boundedMaxService}, and \ref{lem:bounded},
		and Equation \eqref{eq:lem:queueWorkload1},
\begin{eqnarray} \label{eq:lem:queueWorkload2}
\lefteqn{
	\Prob \left(\mu^n
		\sup_{s \in [0,t]} \left(
			\breve{W}^n_\alpha(s) - W^n_\alpha(s) 
		\right)
		> 
		\frac{\epsilon \sqrt{n}}{3}
	\right)
} &&\\ \nonumber
	&\le&
	\Prob \left(
		\mu^n \cdot \sup_{i \le A^n(t)} v^n_i > \frac{\epsilon  \sqrt{n}}{6}
	\right)
	+
	\Prob \left(
		\sup_{s \in [0,t]} \hat Q^n_\alpha(s) >  (L+\epsilon/3) \sqrt{n}
	\right) \\ \nonumber
	&&+
	\Prob \left(
		\sup_{s \in [0,t]} Q^n_\alpha(s) / \mu^n >  L  / \sqrt{n}
	\right)	
	+
	\Prob \left(
		\sup_{s \in [0,t]} W^n_\alpha(s) >  L  / \sqrt{n}
	\right)		
	  \\ \nonumber
	&&+ \; 
	\Prob \left( \mu^n \cdot
		\sup_{j \le A^n(t)}
		\sum_{i=j+1}^{j+\lfloor (L + \epsilon/3) \sqrt{n} \rfloor} 
			v^n_i \cdot \left(
			1\left(b_i \le L /  \sqrt{n}  \right)
			+
			1\left(d_i \le L / \sqrt{n} \right)
			\right)
		> 
		\frac{\epsilon \sqrt{n}}{6}
	\right) 
	\\  \nonumber  
	&<& 
	\frac{\eta}{5}
	+
	\Prob \left( \mu^n \cdot
		\sup_{j \le A^n(t)}
		\sum_{i=j+1}^{j+\lfloor (L + \epsilon/3) \sqrt{n} \rfloor} 
			v^n_i \cdot \left(
			1\left(b_i \le L /  \sqrt{n}  \right)
			+
			1\left(d_i \le L / \sqrt{n} \right)
			\right)
		> 
		\frac{\epsilon \sqrt{n}}{6}
	\right) 	
\end{eqnarray}
	for sufficiently large $n$.
	As with \eqref{eq:lem:queueWorkload1}, we replace $(L + \epsilon)\sqrt{n}$ with a bigger quantity $\mu \delta n$ 
		(provided $n$ is sufficiently large), where, 
		by Proposition  \ref{prop:boundedLostWorkload_shortInterval},
		$\delta >0$ is chosen such that
	\begin{eqnarray}   \label{eq:lem:queueWorkload2.1}
	\lefteqn{
	\Prob \left(
		\mu^n
		\sup_{j \le A^n(t)}
		\sum_{i=j+1}^{j+\lfloor (L + \epsilon/3) \sqrt{n} \rfloor} 
			v^n_i \cdot
			\left(
			1\left(b_i \le L /  \sqrt{n}  \right)
			+
			1\left(d_i \le L / \sqrt{n} \right)
			\right)
		> 
		\frac{\epsilon \sqrt{n}}{6}
	\right)	
	} && \\ \nonumber
	&\le& \Prob \left(
		\mu^n
		\sup_{j \le A^n(t)}
		\sum_{i=j+1}^{j+ \lfloor \mu \delta n \rfloor} 
			v^n_i \cdot
			\left(
			1\left(b_i \le L /  \sqrt{n}  \right)
			+
			1\left(d_i \le L / \sqrt{n} \right)
			\right)
		> 
		\frac{\epsilon \sqrt{n}}{6}
	\right) \\ \nonumber
	&\le& \Prob \left(
		\mu^n
		\sup_{s \le t}
		\sum_{i=A^n(s)+1}^{A^n(s+2\delta)} 
			v^n_i \cdot
			\left(
			1\left(b_i \le L /  \sqrt{n}  \right)
			+
			1\left(d_i \le L / \sqrt{n} \right)
			\right)
		> 
		\frac{\epsilon \sqrt{n}}{6}
	\right) \\ \nonumber
	&&+
	\Prob \left(
		\inf_{s \le t}
			\left(A^n(s+2\delta) - A^n(s) \right)
		<  \mu \delta n 
	\right) \\ \nonumber
	&<& 
	\frac{\eta}{5}
	\end{eqnarray}		
	for sufficiently large $n$.

Lastly, note that by the functional weak law of large numbers,
\begin{equation} \label{eq:lem:queueWorkload3}
	\Prob \left( \sup_{s \in [0,t]} \left |\breve{Q}^n_\alpha(s) - \mu^n \breve{W}^n_\alpha(s) \right| > \frac{\varepsilon  \sqrt{n} }{3} \right)
	< 
	\frac{\eta}{5}
\end{equation}
for sufficiently large $n$.

 It now follows from the Triangle Inequality and \eqref{eq:lem:queueWorkload1}--\eqref{eq:lem:queueWorkload3} that
\begin{eqnarray*}
\lefteqn{ 
	\Prob \left( \sup_{s \in [0,t]} \left |\tilde{Q}^n_\alpha(s) - \mu \tilde{W}^n_\alpha(s) \right| > \varepsilon \right) } \\
& \leq &  
	\Prob \left( \sup_{s \in [0,t]} \left | Q^n_\alpha(s) - \breve{Q}^n_\alpha(s) \right| > \frac{\varepsilon  \sqrt{n} }{3} \right) 
	+ 
	\Prob \left(\mu^n  \sup_{s \in [0,t]} \left |W^n_\alpha(s) - \breve{W}^n_\alpha(s) \right| > \frac{\varepsilon \sqrt{n} }{3 } \right) \\
&&+   
	\Prob \left( \sup_{s \in [0,t]} \left |\breve{Q}^n_\alpha(s) - \mu^n \breve{W}^n_\alpha(s) \right| > \frac{\varepsilon  \sqrt{n} }{3} \right)  \\
&<& \eta
\end{eqnarray*}
for sufficiently large $n$.
This completes the proof.
\end{proofof}

\begin{proofof}{Proposition \ref{prop:allocationLimit}}
	Fix $\varepsilon, \eta, t > 0$ and $\alpha$.
	The server cannot work faster than rate one. It follows that for each $s \le t$,
\begin{equation} \label{eq:allocationLimit1}
	T^n_\alpha(s) \le s.
\end{equation}
Using \eqref{eq:mod:workload^n} we can provide a lower bound on the service allocation process for any $s \ge 0$,
\begin{equation} \label{eq:allocationLimit2}
 	T^n_\alpha(s) 
	\ge 
	-W^n_\alpha(s) 
	+ \sum_{i=1}^{A^n(s)}
			v^n_i
	- \sum_{i=1}^{A^n(s)}
			v^n_i
			\left(
				1(b_i < Q^n_\alpha(t_i-)/ \mu^n)
				+
				1(d_i < W^n_\alpha(t_i-))
			\right).
	\end{equation}
	It follows from \ref{eq:allocationLimit1}, and \eqref{eq:allocationLimit2} that
	\begin{eqnarray}   \label{eq:allocationLimit3}
	\Prob \left(
		 	\sup_{s \in [0,t]}
			\left| 
				T^n_\alpha(s)
				-
				s
			\right| 
			> \varepsilon
	\right)
	&=&
	\Prob \left(
		 \sup_{s \in [0,t]}
			T^n_\alpha(s)
			<
			s - \varepsilon
	\right) \\
	&\le& \nonumber
	\Prob \left(
		 \sup_{s \in [0,t]}
			W^n_\alpha(s) 
		> \frac{\varepsilon}{3}
	\right) 
	+
	\Prob \left(
		\sup_{s \in [0,t]}
		\left|
		\sum_{i=1}^{A^n(s)}
			v^n_i
		-
		s
		\right|
		> \frac{\varepsilon}{3}
	\right) 	\\
	&&+ \nonumber
	\Prob \left(
		\sum_{i=1}^{A^n(t)}
			v^n_i \cdot
			\left(
				1(b_i < Q^n_\alpha(t_i-)/ \mu^n)
				+
				1(d_i < W^n_\alpha(t_i-))
			\right)
		> \frac{\varepsilon}{3}
	\right)
\end{eqnarray}
From Lemma \ref{lem:uniformPotentialWork} and Lemma \ref{lem:bounded} we can bound the first two terms on the right hand side,
\begin{equation} 
	\Prob \left(
		 \sup_{s \in [0,t]}
			W^n_\alpha(s) 
		> \frac{\varepsilon}{3}
	\right) 
	+
	\Prob \left(
		\sup_{s \in [0,t]}
		\left|
		\sum_{i=1}^{A^n(s)}
			v^n_i
		-
		s
		\right|
		> \frac{\varepsilon}{3}
	\right) 
	<
	\frac{\eta}{2}.
\end{equation}
For the third term, by Lemma \ref{lem:boundedArrivals} and Proposition \ref{prop:boundedLostWorkload_shortInterval},
 \begin{eqnarray}   \label{eq:allocationLimit4}
\lefteqn{
	\Prob \left(
		\sum_{i=1}^{A^n(t)}
			v^n_i \cdot
			\left(
				1(b_i <Q^n_\alpha(t_i-)/ \mu^n)
				+
				1(d_i < W^n_\alpha(t_i-))
			\right)
		> \frac{\varepsilon}{3}
	\right)
} \\ 
&\le&
	\Prob \left(  A^n(t) > 2 \mu n  t \right)
	+
	\Prob \left(	
		\sum_{i=1}^{2 \mu n  t}
			v^n_i \cdot
			\left(
				1(b_i < Q^n_\alpha(t_i-)/ \mu^n)
				+
				1(d_i < W^n_\alpha(t_i-))
			\right) 
		> 
		\frac{ \epsilon}{3}
	\right) \\  \nonumber
&\le&
  \Prob \left(	
		\sum_{i=1}^{2 \mu n  t}
			\left( v^n_i - \frac{1}{\mu n} \right)  \cdot
			\left(
				1(b_i < Q^n_\alpha(t_i-)/ \mu^n)
				+
				1(d_i < W^n_\alpha(t_i-))
			\right) 
		> 
		\frac{ \epsilon}{9}
	\right) \\ \nonumber
&+&  \Prob \left(	
		\sum_{i=1}^{2 \mu n  t}
			 \frac{1}{\mu n} \cdot
			\left(
				1(b_i < Q^n_\alpha(t_i-)/ \mu^n) - F_b(Q^n_\alpha(t_i-)/ \mu^n)
				+
				1(d_i < W^n_\alpha(t_i-)) - F_d(W^n_\alpha(t_i-)) 
			\right) 
		> 
		\frac{ \epsilon}{9}
	\right) \\ \nonumber
&+&  \Prob \left(	
		\sum_{i=1}^{2 \mu n  t}
			 \frac{1}{\mu n} \cdot
			\left(
				 F_b(Q^n_\alpha(t_i-)/ \mu^n)
				+
				 F_d(W^n_\alpha(t_i-)) 
			\right) 
		> 
		\frac{ \epsilon}{9}
	\right) + \Prob \left(  A^n(t) > 2 \mu n  t \right) \\ \nonumber
&\leq&
	\frac{\eta}{8} + \frac{\eta}{8} + \frac{\eta}{8} + \frac{\eta}{8} \\ \nonumber
	&<&
	\frac{\eta}{2}.
\end{eqnarray} 
The result follows from \eqref{eq:allocationLimit2} - \eqref{eq:allocationLimit4} and a modification of the proofs of Propositions \ref{prop:boundedBalkingAbandonment_shortInterval} - \ref{prop:centeredBalkingReneging}.
\end{proofof}

\begin{proofof}{Proposition \ref{prop:R2R}}
	Fix $\epsilon, \eta$ and $t >0$. 
	First notice that for any $n \ge 0$, $s \ge 0$, and $\alpha \in \{S.T\}$, 
		we have that $R^{2,n}_\alpha(s) - R^{n}_\alpha(s) \ge 0.$
	It is instructive to expand this difference for each of the $\alpha$ values. 
	For the standard queue,
	\begin{eqnarray*}
		R^{2,n}_S(s) - R^{n}_S(s)
		&=&
		\sum_{i=1}^{A^n(s)}
			1(b_i > Q^n_S(t^n_i-)/\mu^n) \cdot 1(d_i \le W^n_S(t^n_i) )  \cdot 1( d_i > s - t^n_i ) \\
		&\le&
		\sum_{i=1}^{A^n(s)}
			1(d_i \le W^n_S(t^n_i) )  \cdot 1( W^n_S(t^n_i) > s - t^n_i).
	\end{eqnarray*}
	Note that we have eliminated the indicator associated with balking.
	Moreover, for an abandoning customer the workload upon arrival must exceed the patience quantity. 
	Hence we can replace the patience quantity in the last of the indicators with the workload upon arrival.
	Now we consider the ticket queue:
	\begin{eqnarray*}
		R^{2,n}_T(s) - R^{n}_T(s)
		&=&
		\sum_{i=1}^{A^n(s)}
			1(b_i > Q^n_T(t^n_i-)/\mu^n) \cdot 1(d_i \le W^n_T(t^n_i) )  \cdot 1( W^n_T(t^n_i) > s - t^n_i ) \\
		&\le&
		\sum_{i=1}^{A^n(s)}
			1(d_i \le W^n_T(t^n_i) )  \cdot 1( W^n_T(t^n_i) > s - t^n_i).
	\end{eqnarray*}	
	Both the standard and the ticket queue have the same bounds:
	$$
		\Prob \left(\sup_{s \in [0,t]} \left| 
			\tilde{R}^{2,n}_\alpha(s) - \tilde{R}^n_\alpha(s) 
		\right| > \epsilon \right)
		\le
		\Prob \left( 
			\sup_{s \in [0,t]}
			\sum^{A^n(s)}_{i=1} 
				1(d_i \le W^n_\alpha(t^n_i) )  \cdot 1( W^n_\alpha(t^n_i) > s - t^n_i )
			>
			 \sqrt{n}  \epsilon 
		\right)
	$$
		for each $\alpha \in \{S,T\}$.
	For the remainder of the proof, fix $\alpha$.
	Next we replace the workload quantities with an upper bound:
	\begin{eqnarray*} 
		\Prob \left( \sup_{s \in [0,t]} \left| 
			\tilde{R}^{2,n}_\alpha(s) - \tilde{R}^n_\alpha(s) 
		\right| > \epsilon \right)
		&\le&
		\Prob \left( 
			\sup_{s \in [0,t]}
			\sum^{A^n(s)}_{i=1} 
				1(d_i \le K/\sqrt{n} )  \cdot 1( K/\sqrt{n} > s - t^n_i )
			>
			 \sqrt{n}  \epsilon 
		\right) \\
		&&+ \Prob 
			\left(
				\sup_{s \in [0,t]} W^n_\alpha(s) > K/\sqrt{n} 
			\right) .
	\end{eqnarray*}
	Notice that in the first term on the righthand side above, the only jobs that contribute positively to the summation
		are those  jobs $i$ whose arrival time is after $s- K/\sqrt{n}$; that $t^n_i > s - K/\sqrt{n}$.
	By Lemma \ref{lem:bounded},
	$$
		\Prob \left( \sup_{s \in [0,t]} \left| 
			\tilde{R}^{2,n}_\alpha(s) - \tilde{R}^n_\alpha(s) 
		\right| > \epsilon \right)
	\le 	
		\frac{\eta}{2}
		+
		\Prob \left( 
			\sup_{s \in [0,t]}
			\sum^{A^n(s)}_{i=A^n((s-K/\sqrt{n})^+)} 
				1(d_i \le K/\sqrt{n} ) 
			>
			 \sqrt{n}  \epsilon 
		\right) 
	$$
	for sufficiently large $n$.
	For any arbitrarily chosen $\delta>0$ it is true that $\delta > K/\sqrt{n}$ for sufficiently large $n$.
	It follows then by Proposition \ref{prop:boundedBalkingAbandonment_shortInterval}
		that we can choose a $\delta >0$ so that
	\begin{eqnarray*} 
		\Prob \left( \sup_{s \in [0,t]} \left| 
			\tilde{R}^{2,n}_\alpha(s) - \tilde{R}^n_\alpha(s) 
		\right| > \epsilon \right)
	&\le& 	
		\frac{\eta}{2}
		+
		\Prob \left( 
			\sup_{s \in [0,t]}
			\sum^{A^n(s)}_{i=A^n((s-K/\sqrt{n})^+)} 
				1(d_i \le K/\sqrt{n} ) 
			>
			 \sqrt{n}  \epsilon 
		\right) \\
	&\le& 	
		\frac{\eta}{2}
		+
		\Prob \left( 
			\sup_{s \in [0,t]}
			\sum^{A^n(s)}_{i=A^n((s-\delta)^+)} 
				1(d_i \le K/\sqrt{n} ) 
			>
			 \sqrt{n}  \epsilon 
		\right) \\		
	&<& \eta
	\end{eqnarray*}
	for sufficiently large $n$. 	This completes the proof.
\end{proofof}

\begin{proofof}{Proposition \ref{prop:R1R2}}
	Fix $\epsilon, \eta, t >0$ and $\alpha \in \{S,T\}$.
	Notice that $R^{1,n}_\alpha - R^{2,n}_\alpha$ is nondecreasing.
	Hence, replacing the workload and queue length with an upper bound yields
	\begin{eqnarray*}
	\Prob\left(
		\sup_{s \in [0,t]} \left| 
			\tilde{R}^{1,n}_\alpha(s) - \tilde{R}^{2,n}_\alpha(s) 
		\right| 
		> \epsilon \right) 
	&\le&
	\Prob \left( 
		\sum^{A^n(t)}_{i=1} 
			 1(b_i \le Q^n_\alpha(t^n_i-)/\mu^n ) 
			 \cdot 
			 1(d_i \le W^n_T(t^n_i-) )
		>  \sqrt{n}  \epsilon 
	\right) \\
	&\le&
	\Prob \left( 
		\sum^{\left\lfloor 2\mu nt \right\rfloor}_{i=1} 
			 1(b_i \le K/\sqrt{n})
			 \cdot 
			 1(d_i \le K/\sqrt{n})
		>  \sqrt{n}  \epsilon 
	\right) 
	+
	\Prob \left( A^n(t) > 2 \mu nt\right)\\
	&&+
	\Prob \left( 
		\sup_{s \in [0,t]}
			Q^n_\alpha(s)/ \mu^n > K/\sqrt{n}
	\right)
	+
	\Prob \left( 
		\sup_{s \in [0,t]}
			W^n_\alpha(s) > K/\sqrt{n}
	\right).
	\end{eqnarray*}
	Choose a $K>0$ such that by Lemma \ref{lem:boundedArrivals}
		and \eqref{eq:lemBoundedQueue} and \eqref{eq:lemBoundedWorkload} of Lemma \ref{lem:bounded},
	\begin{equation} \label{eq:R1R2.1}
	\Prob\left(
		\sup_{s \in [0,t]} \left| 
			\tilde{R}^{1,n}_\alpha(s) - \tilde{R}^{2,n}_\alpha(s) 
		\right| 
		> \epsilon \right) 
	<
	\frac{\eta}{2}
	+	
	\Prob \left( 
		\sum^{\left\lceil 2\mu nt \right\rceil}_{i=1} 
			 1(b_i \le K/\sqrt{n})
			 \cdot 
			 1(d_i \le K/\sqrt{n})
		>  \sqrt{n}  \epsilon 
	\right)
	\end{equation}
	
	As for the second term in the right hand side above, we resort to adding and subtracting the mean of each summand,
	\begin{eqnarray} \nonumber
	\lefteqn{
	\Prob \left( 
		\sum^{\left\lfloor 2\mu nt \right\rfloor}_{i=1} 
			 1(b_i \le K/\sqrt{n})
			 \cdot 
			 1(d_i \le K/\sqrt{n})
		>  \sqrt{n}  \epsilon 
	\right)
	} \\ \nonumber
	&\le&
		\Prob \left( 
		\sum^{\left\lfloor 2\mu nt \right\rfloor}_{i=1}  
			 \Exp \left[
			 1(b_i \le K/\sqrt{n})
			 \cdot 
			 1(d_i \le K/\sqrt{n})
			 \right]		 
		>  \frac{\sqrt{n}  \epsilon}{2}
	\right) \\ \nonumber
	&&+
		\Prob \left( 
		\sum^{\left\lfloor 2\mu nt \right\rfloor}_{i=1}  \left(
			 1(b_i \le K/\sqrt{n})
			 \cdot 
			 1(d_i \le K/\sqrt{n})
			 -
			 \Exp \left[
			 1(b_i \le K/\sqrt{n})
			 \cdot 
			 1(d_i \le K/\sqrt{n})
			 \right]
			 \right)			 
		>  \frac{\sqrt{n}  \epsilon}{2}
	\right) \\ \nonumber
	&\le&
		\Prob \left( 
		 	(2\mu nt )
			 F_b( K/\sqrt{n})
			 F_d( K/\sqrt{n})	 
		>  \frac{\sqrt{n}  \epsilon}{2}
	\right) \\ \label{eq:R1R2.2}
	&&+
		\Prob \left( 
		\sum^{\left\lfloor 2\mu nt \right\rfloor}_{i=1}  \left(
			 1(b_i \le K/\sqrt{n})
			 \cdot 
			 1(d_i \le K/\sqrt{n})
			 -
			 F_b( K/\sqrt{n})
			 F_d( K/\sqrt{n})	 
			 \right)			 
		>  \frac{\sqrt{n}  \epsilon}{2}
		\right).
	\end{eqnarray}
	By Lemma \ref{lem:theta} we can bound the first term on the right hand side:
	\begin{equation} \label{eq:R1R2.3}
		(2\mu nt )
			 F_b( K/\sqrt{n})
			 F_d( K/\sqrt{n})	 
		<  \frac{\sqrt{n}  \epsilon}{2}
	\end{equation}
	for sufficiently large $n$.
	As for the second term, by Chebyshev's Inequality and \eqref{eq:R1R2.3}
	\begin{eqnarray} \nonumber
	\lefteqn{
		\Prob \left( 
		\sum^{\left\lfloor 2\mu nt \right\rfloor}_{i=1}  \left(
			 1(b_i \le K/\sqrt{n})
			 \cdot 
			 1(d_i \le K/\sqrt{n})
			 -
			 F_b( K/\sqrt{n})
			 F_d( K/\sqrt{n})	 
			 \right)			 
		>  \frac{\sqrt{n}  \epsilon}{2}
		\right)
	} \\ \nonumber
	&\le&
		\frac{4}{n \epsilon^2}
		\Exp\left[ \left(
			\sum^{\left\lfloor 2\mu nt \right\rfloor}_{i=1}  \left(
				 1(b_i \le K/\sqrt{n})
				 \cdot 
				 1(d_i \le K/\sqrt{n})
				 -
				 F_b( K/\sqrt{n})
				 F_d( K/\sqrt{n})	 
			 \right)		
		\right)^2 \right] \\ \nonumber
	&\le&
		\frac{4}{n \epsilon^2}
		\Exp\left[ \
			\sum^{\left\lfloor 2\mu nt \right\rfloor}_{i=1}  \left(
				 1(b_i \le K/\sqrt{n})
				 \cdot 
				 1(d_i \le K/\sqrt{n})
				 -
				 F_b( K/\sqrt{n})
				 F_d( K/\sqrt{n})	 
			 \right)^2		
		\right] \\ \nonumber
	&<&
		\frac{8}{n \epsilon^2}
		(2\mu nt )
				 F_b( K/\sqrt{n})
				 F_d( K/\sqrt{n})	 \\ \label{eq:R1R2.4}
	&<&
	\frac{\eta}{2}.
	\end{eqnarray}			
	The result follows from \eqref{eq:R1R2.1} - \eqref{eq:R1R2.4}.
\end{proofof}

\begin{proofof}{Proposition \ref{prop:R0R1}}
	Fix $\epsilon, \eta, t >0$, $\alpha \in \{S,T\}$, and set $\delta = \epsilon/(16 \mu t)$.
	By Lemma \ref{lem:bounded} and Proposition \ref{prop:queueWorkload}, respectively, we can choose a $K>0$
		such that for sufficiently large $n$,
	\begin{equation} \label{eq:R0R1.0.1}
		\Prob \left(
			\sup_{s \in [0,t]}
				W^n_\alpha(s)
			> K/\sqrt{n}
		\right)
		<
		\min\left(
			\frac{\eta}{4},
			\frac{\epsilon^2 \eta}{128 \mu t }
		\right)
	\end{equation}
	and
	\begin{equation} \label{eq:R0R1.0.2}
	\Prob \left(
		\sup_{s \in [0,t]}
		\left|
			Q^n_\alpha(s) / \mu^n 
			-
			W^n_\alpha(s)
		\right|
		> \delta / \sqrt{n}
	\right)
	<
	\frac{\eta}{8}.
	\end{equation}	
	Fix such a $K$.
	We start by replacing the weighted queue length with the workload process:
	\begin{eqnarray} \nonumber
	\lefteqn{
	\Prob \left(
		\sup_{s \in [0,t]} \left| 
			\tilde{R}^{0,n}_\alpha(s) - \tilde{R}^{1,n}_\alpha(s) 
		\right| 
		> \epsilon \right) 	
	} \\ \nonumber
	&=&
	\Prob \left(
		\sup_{s \in [0,t]} \left| 
			\sum_{i=1}^{A^n(s)}
				1(d_i \le Q^n_\alpha(t^n_i-)/\mu^n)
				-
				1(d_i \le W^n_\alpha(t^n_i-))
		\right| 
		> \sqrt{n} \epsilon \right) 	\\ \nonumber
	&\le&
	\Prob \left(
		\sum_{i=1}^{A^n(t)}
			1(d_i \in [W^n_\alpha(t^n_i-) - \delta / \sqrt{n}, W^n_\alpha(t^n_i-) + \delta / \sqrt{n}])
		>  \sqrt{n} \epsilon 
	\right) \\ \label{eq:R0R1.1}
	&&+
	\Prob \left(
		\sup_{s \in [0,t]}
		\left|
			Q^n_\alpha(s) / \mu^n 
			-
			W^n_\alpha(s)
		\right|
		> \delta / \sqrt{n}
	\right)
	\end{eqnarray}
	
	The second term of the far right hand side is handled by \eqref{eq:R0R1.0.2}.
	As for the first term, we construct a Martingale.
	By Lemma \ref{lem:boundedArrivals},
	\begin{eqnarray} \nonumber
	\lefteqn{
	\Prob \left(
		\sum_{i=1}^{A^n(t)}
			1(d_i \in [W^n_\alpha(t^n_i-) - \delta / \sqrt{n}, W^n_\alpha(t^n_i-) + \delta / \sqrt{n}])
		>  \sqrt{n} \epsilon 
	\right) 
	} \\ \nonumber
	&\le&
	\frac{\eta}{4}
	+
	\Prob \left(
		\sum_{i=1}^{\left\lfloor 2\mu nt \right\rfloor}
			\Exp\left[ \left. 1(d_i \in [W^n_\alpha(t^n_i-) - \delta / \sqrt{n}, W^n_\alpha(t^n_i-) + \delta / \sqrt{n}])
				\right|   \mathcal{F}^n_{i-1} \right]
		>  \frac{\sqrt{n} \epsilon}{2} 
	\right) \\ \nonumber
	&&+
	\Prob \left( 
		\sum_{i=1}^{\left\lfloor 2\mu nt \right\rfloor} \left(
			1(d_i \in [W^n_\alpha(t^n_i-) - \delta / \sqrt{n}, W^n_\alpha(t^n_i-) + \delta / \sqrt{n}])
			\right. \right. \\ \nonumber
	&& \qquad \left. \left.
			- \Exp\left[ \left. 1(d_i \in [W^n_\alpha(t^n_i-) - \delta / \sqrt{n}, W^n_\alpha(t^n_i-) + \delta / \sqrt{n}])
				\right|   \mathcal{F}^n_{i-1} \right] \right)
		>  \frac{\sqrt{n} \epsilon}{2} 
	\right) \\ \nonumber
	&\le&
	\frac{\eta}{4}
	+
	\Prob \left(
		\sum_{i=1}^{\left\lfloor 2\mu nt \right\rfloor} \left(
			F_d(W^n_\alpha(t^n_i-) + \delta / \sqrt{n})
			-
			F_d(W^n_\alpha(t^n_i-) - \delta / \sqrt{n})
				\right)  		
		>  \frac{\sqrt{n} \epsilon}{2} 
	\right) \\   \label{eq:R0R1.2}
	&&+
	\Prob \left(
		\sum_{i=1}^{\left\lfloor 2\mu nt \right\rfloor} \left(
			1(d_i \in [W^n_\alpha(t^n_i-) - \delta / \sqrt{n}, W^n_\alpha(t^n_i-) + \delta / \sqrt{n}])
			\right. \right. \\ \nonumber 
	&& \qquad \left. \left.
			- \left( F_d(W^n_\alpha(t^n_i-) + \delta / \sqrt{n})
			-
			F_d(W^n_\alpha(t^n_i-) - \delta / \sqrt{n})
			\right)
			\right)
		>  \frac{\sqrt{n} \epsilon}{2} 
	\right) \\ \nonumber
	\end{eqnarray} 
	
	As for the second term on the far right hand side of \eqref{eq:R0R1.2}, 
		notice that each summand contributes an amount
		equal to the increase of the abandonment distribution function over an interval  of length $2 \delta$. 
	By \eqref{eq:R0R1.0.1} and  Lemma \ref{lem:theta2},
	\begin{eqnarray} \nonumber
	\lefteqn{
	\Prob \left(
		\sum_{i=1}^{\left\lfloor 2\mu nt \right\rfloor} \left(
			F_d(W^n_\alpha(t^n_i-) + \delta / \sqrt{n})
			-
			F_d(W^n_\alpha(t^n_i-) - \delta / \sqrt{n}])
				\right)  		
		>  \frac{\sqrt{n} \epsilon}{2} 
	\right)
	} \\ \nonumber
	&\le&
	\Prob \left(
		\sup_{s \in [0,t]}
		\tilde W^n(s) 
		> K
	\right)
	+
	1 \left(
		2\mu nt
		\sup_{s \in [0,K]}
		\left(
			F_d( (s+2 \delta)/\sqrt{n})
			-
			F_d(s/\sqrt{n})
		\right)
	> \frac{\sqrt{n} \epsilon}{2}
	\right) \\   \label{eq:R0R1.3}
	&<&
	\frac{\eta}{4}
	\end{eqnarray}
	for sufficiently large $n$.
	Now consider the third term on the right hand side of \eqref {eq:R0R1.2}.
	The following steps are similar to those in the proof of Proposition \ref{prop:centeredBalkingReneging}.
	By \eqref{eq:R0R1.0.1} and Lemma \ref{lem:theta},
	\begin{eqnarray} \nonumber
	\lefteqn{
	\Prob \left(
		\sum_{i=1}^{\left\lfloor 2\mu nt \right\rfloor} \left(
			1(d_i \in [W^n_\alpha(t^n_i-) - \delta / \sqrt{n}, W^n_\alpha(t^n_i-) + \delta / \sqrt{n}])
			\right. \right.  }  \quad \\ \nonumber 
	&\quad& \left. \left.
			- \left( F_d(W^n_\alpha(t^n_i-) + \delta / \sqrt{n})
			-
			F_d(W^n_\alpha(t^n_i-) - \delta / \sqrt{n})
			\right)
			\right)
		>  \frac{\sqrt{n} \epsilon}{2} 
	\right)  \\ \nonumber	
	&\le& 
	\frac{4}{n \epsilon^2}
	\Exp\left[ \left(
		\sum_{i=1}^{\left\lfloor 2\mu nt \right\rfloor}
		\left( 
			1(d_i \in [W^n_\alpha(t^n_i-) - \delta / \sqrt{n}, W^n_\alpha(t^n_i-) + \delta / \sqrt{n}])
			\right. \right. \right. \\ \nonumber
	&& \qquad
			\left. \left. \left.
			- \left( F_d(W^n_\alpha(t^n_i-) + \delta / \sqrt{n})
			-
			F_d(W^n_\alpha(t^n_i-) - \delta / \sqrt{n})
			\right)
		\right)
		\right)^2
	\right] \\ \nonumber	
	&\le& 
	\frac{4}{n \epsilon^2}
	\Exp\left[ 
		\sum_{i=1}^{\left\lfloor 2\mu nt \right\rfloor} \left(
			1(d_i \in [W^n_\alpha(t^n_i-) - \delta / \sqrt{n}, W^n_\alpha(t^n_i-) + \delta / \sqrt{n}])
			\right. \right.  \\ \nonumber
	&& \qquad
			\left. \left. 
			- \left( F_d(W^n_\alpha(t^n_i-) + \delta / \sqrt{n})
			-
			F_d(W^n_\alpha(t^n_i-) - \delta / \sqrt{n})
			\right)
		\right)^2
	\right] \\ \nonumber	
	&\le& 
	\frac{8}{n \epsilon^2}
	\Exp\left[ 
		\sum_{i=1}^{\left\lfloor 2\mu nt \right\rfloor}
			F_d(W^n_\alpha(t^n_i-) + \delta / \sqrt{n})
		\right]
	\\  \nonumber	
	&\le& 
	\frac{16\mu t}{\epsilon^2}
		\left(
		\Prob \left(
			\sup_{s \in [0,t]}
			W^n_\alpha(s)
			> K/\sqrt{n}
		\right)
		+
			F_d((K + \delta) / \sqrt{n})
	\right) \\  
	\label{eq:R0R1.4}
	&<& \frac{\eta}{4},
	\end{eqnarray}
	for sufficiently large $n$ greater than $ \left( \frac{256 \mu t \theta (K + \delta)}{\epsilon^2\eta}
		\right)^2$.
	
	The result follows from \eqref{eq:R0R1.0.2} -- \eqref{eq:R0R1.4}.
\end{proofof}

\begin{proofof}{Proposition \ref{prop:noInitialReneging}}
Fix $\eta, t >0$ and $\alpha \in \{S,T\}$. 
By Lemma \ref{lem:bounded} there exists a $K>1$ such that, for sufficiently large $n$,
$$
	\Prob \left( \sup_{s \in [0,t]} Q^n_\alpha(s) > K \sqrt{n} \right) 
	+
	\Prob \left( \sup_{s \in [0,t]} W^n_\alpha(s) > \frac{K}{\sqrt{n}} \right) < \frac{\eta}{2}.
$$
We use these constants to replace the workload and queue length:
\begin{eqnarray}   \label{proof:prop:noInitialReneging.1}
	\Prob \left(
		\hat{R}^n_\alpha(t) > L 
	\right)
	&\le&
	\Prob \left(
		\sum_{i=1}^{Q^n(0)}
			1(\hat{d}_i \le \hat{w}^n_{i-1} )
		> 
		L
	\right)
	\\ &\le& \nonumber
	\Prob \left(
		\sum_{i=1}^{\lfloor K \sqrt{n} \rfloor}
			1(\hat{d}_i \le K / \sqrt{n} )
		> 
		L
	\right)		
	+
	\Prob \left(Q^n(0) \ge K \sqrt{n} \right)
	+
	\Prob \left(W^n(0) > \frac{K}{\sqrt{n}} \right)
	\\ &<& \nonumber
	\Prob \left(
		\sum_{i=1}^{\lfloor K \sqrt{n} \rfloor}
			1(\hat{d}_i \le K / \sqrt{n} )
		> 
		L
	\right)
	+ 
	\frac{\eta}{2}
\end{eqnarray}	
for sufficiently large $n$.

Recall that the residual deadlines of the initial jobs may have different distributions, $\hat{F}_{d,i}$,  but those distributions 
	have a common bound near the origin.
Namely, there exists an $\hat{f} \ge 1$ and an $h_0 \in (0, 1/\hat{f})$ 
	such that $\hat{F}_{d,i}(h) \le h \hat{f}$ for each $h \le h_0$.
Given such an $h_0$ and $\hat{f}$, set $L$ so that $L \ge \max(2 \hat{f} K^2,4 \hat{f} K/  \sqrt{\eta})$.
We replace the initial jobs' residual deadlines with random variables which are i.i.d.
In particular, on the same probability space, each $\hat{d}_i$ is replaced by $\breve{d}_i$.
Whenever $\hat{d}_i < h_0$ then $\breve{d}_i \ge \hat{d}_i$. 
Moreover, the common distribution of each $\breve{d}_i$ is
$$
	\breve{F}(x) = \left\{ \begin{array}{ll}
		x \hat{f},  & x < h_0 \\
		1, & x \ge h_0.
	\end{array}
	\right.
$$
When $K/\sqrt{n} \le h_0$ we have that
$$
	\left( K \sqrt{n} \right)
	\breve{F}(K/\sqrt{n})
	=
	\hat{f} K^2
	< L/2
$$
so that by Kolmogorov's inequality \cite{ref:Bass},
\begin{eqnarray}    \label{proof:prop:noInitialReneging.2}
	\Prob \left(
		\sum_{i=1}^{\lfloor K \sqrt{n} \rfloor}
			1(\hat{d}_i \le K / \sqrt{n} )
		> 
		L
	\right)
	&\le&
	\Prob \left(
		\sum_{i=1}^{\lfloor K \sqrt{n} \rfloor}
			1(\breve{d}_i \le K / \sqrt{n} )
		> 
		L
	\right)
	\\&\le& \nonumber
	\Prob \left(
		\sum_{i=1}^{\lfloor K \sqrt{n} \rfloor}
			\left( 1(\breve{d}_i \le K / \sqrt{n} ) - \breve{F}(K/\sqrt{n}) \right)
		> 
		\frac{L}{2}
	\right)	
	\\&\le& \nonumber
	\left( \frac{4}{L^2} \right)
	(K \sqrt{n})
	\Exp[ \left( 1(\breve{d}_1 \le K / \sqrt{n} ) - \breve{F}(K/\sqrt{n}) \right)^2]
	\\&\le& \nonumber
	\left( \frac{4}{L^2} \right)
	(K \sqrt{n})	
	\breve{F}(K/\sqrt{n})	
	\\ &<& \frac{\eta}{2} \nonumber
\end{eqnarray}	
for sufficiently large $n$. 
The result follows from \eqref{proof:prop:noInitialReneging.1} and \eqref{proof:prop:noInitialReneging.2}.

\end{proofof}

\begin{proofof}{Proposition \ref{prop:deltaNegligible}}
We can break $\delta^n_\alpha$ into four parts. For every $t \ge 0$ and $\alpha \in \{S,T\},$
\begin{eqnarray} \nonumber \label{eq:deltaNegligible}
	\delta^n_\alpha(t)
	&=&
	\frac{1}{\sqrt{n}}
	\sum_{i=1}^{A^n(t)}
		F_b(Q^n_\alpha(t^n_i-)/\mu^n)
		-
		F_b(Q^n_\alpha(t^n_i-)/(\mu n))
	+
		F_d(Q^n_\alpha(t^n_i-)/\mu^n)
		-
		F_d(Q^n_\alpha(t^n_i-)/(\mu n))
	\\ \nonumber
	&&+
	\frac{1}{\sqrt{n}}
	\sum_{i=1}^{A^n(t)}
	\left(
		F_b(Q^n_\alpha(t^n_i-)/(\mu n))
		+
		F_d(Q^n_\alpha(t^n_i-)/(\mu n))
		-
		\theta
		\frac{Q^n_\alpha(t^n_i-)}{\mu n}
	\right) \\ \nonumber
	&&+
	\theta \left(
		\int_0^t
			\tilde Q^n_\alpha(s-) d \left(\frac{\bar A^n(s)}{\mu} \right)
		-
		\int_0^t
			\tilde Q^n_\alpha(s) ds
	\right)
	\\&&+
	\frac{1}{\sqrt{n}} \left(
		\hat{S}^n(t) - \mu^n \min(t,W^n(0))
	\right).
\end{eqnarray}
We will show that each of these converges to zero in probability as $n \to \infty$.	
			
Fix $t>0$ and $\alpha \in \{S,T\}$ and select arbitrary constants $\epsilon, \eta>0$.
We first show that both
\begin{equation} \label{eq:deltaNegligible.1a}
	\limsup_{n \to \infty}
	\Prob \left(	
		\frac{1}{\sqrt{n}}
		\sum_{i=1}^{A^n(t)}
		\left|
			F_b(Q^n_\alpha(t^n_i-)/(\mu n))
			+
			F_b(Q^n_\alpha(t^n_i-)/(\mu n))
		\right|
		>
		\frac{\epsilon}{4}
	\right)
	<
	\frac{\eta}{4}
\end{equation}
and
\begin{equation} \label{eq:deltaNegligible.1b}
	\limsup_{n \to \infty}
	\Prob \left(	
		\frac{1}{\sqrt{n}}
		\sum_{i=1}^{A^n(t)}
		\left|
			F_d(Q^n_\alpha(t^n_i-)/(\mu n))
			+
			F_d(Q^n_\alpha(t^n_i-)/(\mu n))
		\right|
		>
		\frac{\epsilon}{4}
	\right)
	<
	\frac{\eta}{4}.
\end{equation}

We will prove \eqref{eq:deltaNegligible.1a} and the proof of \eqref{eq:deltaNegligible.1b} follows trivially.
The right hand side of  \eqref{eq:deltaNegligible.1a} can be expanded:
\begin{eqnarray}  \label{eq:deltaNegligible.1a.1}
\lefteqn{
	\Prob \left(	
		\frac{1}{\sqrt{n}}
		\sum_{i=1}^{A^n(t)}
		\left|
			F_b(Q^n_\alpha(t^n_i-)/(\mu n))
			+
			F_b(Q^n_\alpha(t^n_i-)/(\mu n))
		\right|
		>
		\frac{\epsilon}{4}
	\right)
} && \\ 	&\le& \nonumber
	\Prob \left(	
		A^n(t)
		\left(
		\sup_{s \in [0,t]}
			\left|
				F_b(Q^n_\alpha(s)/\mu^n)
				+
				F_b(Q^n_\alpha(s)/(\mu n))
			\right|
		\right)
		>
		\frac{\epsilon \sqrt{n}}{4}
	\right)
\\ \nonumber
	&\le&
	\Prob \left(	
		\sup_{s \in [0,t]}
				Q^n_\alpha(s)
		>
		K \sqrt{n}
	\right)
	+
	\Prob \left(	
		A^n(t)
		>
		2 \mu n t
	\right) 
	+
	\Prob \left(
		\sup_{s \le K}
			\left|
				F_b(s\sqrt{n}/(\mu n))
				+
				F_b(s\sqrt{n}/(\mu n) )
			\right|
		>
		\frac{\epsilon}{8 \mu t \sqrt{n}}
	\right).
\end{eqnarray}
Fix $K>0$ so that by \eqref{eq:lemBoundedQueue} of Lemma \ref{lem:bounded},
\begin{equation} \label{eq:deltaNegligible.1a.2}
	\Prob \left(	
		\sup_{s \in [0,t]}
				Q^n_\alpha(s)
		>
		K \sqrt{n}
	\right)
	<
	\frac{\eta}{8}.
\end{equation}

Notice that for any $\delta>0$ we have that
	$$
	\sup_{s \le K}
		\left| 
			\frac{s\sqrt{n}}{\mu n} 
			-
			\frac{s\sqrt{n}}{\mu n}
		\right|
		\le 
		\left| 
			\frac{K\sqrt{n}}{\mu^n} 
			-
			\frac{K\sqrt{n}}{\mu n}
		\right|
		<
		\frac{\delta}{\sqrt{n}}
	$$
for sufficiently large $n$.
Set $\delta = \epsilon/(16 \mu t \theta).$
The first result, \eqref{eq:deltaNegligible.1a}, follows from \eqref{eq:deltaNegligible.1a.1}, 
	 \eqref{eq:deltaNegligible.1a.2},  and  Lemmas \ref{lem:boundedArrivals}  and  \ref{lem:theta2}.
Likewise,  \eqref{eq:deltaNegligible.1b} follows from an identical argument.

Next we show that
\begin{equation} \label{eq:deltaNegligible.2}
	\limsup_{n \to \infty}
	\Prob \left(	
	\frac{1}{\sqrt{n}}
	\sum_{i=1}^{A^n(t)}
	\left|
		F_b(Q^n_\alpha(t^n_i-)/(\mu n))
		+
		F_d(Q^n_\alpha(t^n_i-)/(\mu n))
		-
		\theta
		\frac{Q^n_\alpha(t^n_i-)}{\mu n}
	\right|
		>
		\frac{\epsilon}{4}
	\right)
	<
	\frac{\eta}{4}.
\end{equation}
We explore the derivative of the abandonment and balking distributions:
\begin{eqnarray} \label{eq:deltaNegligible.2.1}
\lefteqn{
	\Prob \left(	
	\frac{1}{\sqrt{n}}
	\sum_{i=1}^{A^n(t)}
	\left|
		F_b(Q^n_\alpha(t^n_i-)/(\mu n))
		+
		F_d(Q^n_\alpha(t^n_i-)/(\mu n))
		-
		\theta
		\frac{Q^n_\alpha(t^n_i-)}{\mu n}
	\right|
		>
		\frac{\epsilon}{4}
	\right)
} &&\\ \nonumber
	&\le&
	\Prob \left(	
		A^n(t)
		\sup_{s \in [0,t]}
			\left|
				F_b(Q^n_\alpha(s)/(\mu n))
				+
				F_d(Q^n_\alpha(s)/(\mu n))
				-
				\theta \frac{ Q^n_\alpha(s)}{\mu n}
			\right|
		>
		\frac{\epsilon \sqrt{n}}{4}
	\right) \\ \nonumber
	&\le&
	\Prob \left(	
		\sup_{s \in [0,t]}
				Q^n_\alpha(s)
		>
		K \sqrt{n}
	\right)
	+
	\Prob \left(	
		A^n(t)
		>
		2 \mu n t
	\right) 
	 \\ \nonumber
	&&+
	1 \left(
	\sup_{s \in [0,K/\mu]}
	\left|
		F_b(s/ \sqrt{n})
		+
		F_d(s/ \sqrt{n})
		-
		\frac{\theta s}{\sqrt{n}}
	\right|
	>
	\frac{\epsilon}{8 \mu t \sqrt{n}}
	\right) 
\end{eqnarray}
Recall that the derivatives at zero of both $F_b$ and $F_d$ exist and sum to $\theta$. 
Hence, for a any given  $\delta>0$ there exists an $h_0$ such that 
$$
	\sup_{s \le h} 
	\left| 
		\frac{F_b(s)}{s} + \frac{F_d(s)}{s} - \theta
	\right|
	<
	\delta
$$
for all $h \le h_0$.
Let $\delta = \frac{\epsilon}{8Kt}$ 
	and let $h_0$ be a constant so that the above inequality holds.
Now choose an $n_0$ such that  $ K/(\mu \sqrt{n}) \le h_0$ for all $n \ge n_0$.
It follows that for each $n \ge n_0$,
\begin{equation} \label{eq:deltaNegligible.2.2}
	\sup_{s \in [0,K/(\mu \sqrt{n})]}
	\left| 
		F_b(s) + F_d(s) - \theta s
	\right|
	=
	\sup_{s \in [0,K/\mu]}
	\left| 
		F(s/\sqrt{n}) - 
		\frac{\theta s}{\sqrt{n}}
	\right|
	< \delta \frac{K}{\mu \sqrt{n}}
	=
 	\frac{\epsilon}{8 \mu t \sqrt{n}}.
\end{equation}
The result \eqref{eq:deltaNegligible.2} follows from  \eqref{eq:deltaNegligible.1a.2}, \eqref{eq:deltaNegligible.2.1}, 	
	\eqref{eq:deltaNegligible.2.2} and Lemma \ref{lem:boundedArrivals}.

Third, we show that 
	\begin{equation} \label{eq:deltaNegligible.3}
	\limsup_{n \to \infty}
	\Prob \left(
	 \left|
		\int_0^t
			\tilde Q^n_\alpha(s-) d \left(\frac{\bar A^n(s)}{\mu} \right)
		-
		\int_0^t
			\tilde Q^n_\alpha(s) ds
	\right|
	> \epsilon
	\right)
	< 
	\eta.
	\end{equation}
By Proposition \ref{prop:tightness}, the process $\{\tilde Q^n_\alpha(s), s \le t\}$ is tight. 
Consider a subsequence $\{n^\prime\}$ over which the process $\tilde Q^{n^\prime}_\alpha$ has a limit, 
	say $\tilde Q_\alpha$.
By the Skorohod Representation Theorem, there exists an alternative probability space on which 
	are defined a sequence $\{(\hat{\tilde Q}^{n^\prime}_\alpha, \hat{\bar A}^{n^\prime}_\alpha), n \ge 1\}$
	and, by Lemma \ref{lem:uniformArrivals}, a limit process $(\hat{\tilde Q}_\alpha, \hat{\bar A}_\alpha)$ such that
	$$
		(\hat{\tilde Q}^{n^\prime}_\alpha, \hat{\bar A}^{n^\prime}_\alpha) {\buildrel D \over = } (\tilde Q^{n^\prime}_\alpha, \bar A^{n^\prime}_\alpha)
	$$
	for each $n^\prime$ and such that
	$(\hat{\tilde Q}^{n^\prime}_\alpha, \hat{\bar A}^{n^\prime}_\alpha) \to (\hat{\tilde Q}_\alpha, \hat{\bar A}_\alpha)$
		almost surely as $n^\prime \to \infty$.
	It is also true that
	$$
		\int_0^u
			\hat{\tilde Q}^{n^\prime}_\alpha(s-) 
				d \left(\frac{\hat{\bar A}^{n^\prime} (s)}{\mu} \right)  
		{\buildrel D \over = }
		\int_0^u
			\tilde Q^{n^\prime}_\alpha(s-) 
				d \left(\frac{\bar A^{n^\prime} (s)}{\mu} \right) ,
	$$
	and
	$$		
		\int_0^t
			\hat{\tilde Q}^{n^\prime}_\alpha(s) ds
		{\buildrel D \over = }
		\int_0^t
			\tilde Q^{n^\prime}_\alpha(s) ds
	$$	
		for each $n^\prime$.
Applying Lemma 8.3 from \cite{ref:DD} twice we have
	$$
	\sup_{u \in [0,t]} 
	\left|
		\int_0^u
			\hat{\tilde Q}^{n^\prime}_\alpha(s-) 
				d \left(\frac{\hat{\bar A}^{n^\prime} (s)}{\mu} \right)  
		-
		\int_0^u \hat {\tilde{Q}} (s) ds
	\right|
	\to 
	0
	$$
	and
	$$
	\sup_{u \in [0,t]} 
	\left|
		\int_0^t
			\hat{\tilde Q}^{n^\prime}_\alpha(s) ds
		-
		\int_0^t \hat {\tilde{Q}} (s) ds
	\right|
	\to 0
	$$
almost surely as $n^\prime \to \infty$, so that
$$
	\sup_{u \in [0,t]} 
	\left|
		\int_0^u
			\hat{\tilde Q}^{n^\prime}_\alpha(s-) 
				d \left(\frac{\hat{\bar A}^{n^\prime} (s)}{\mu} \right)  
		-
		\int_0^t
			\hat{\tilde Q}^{n^\prime}_\alpha(s) ds
	\right|		
	\to 0
$$
almost surely as $n^\prime \to \infty$.
	It follows that in our original probability space
$$
	\sup_{u \in [0,t]} 
	\left|
		\int_0^u
			\tilde Q^{n^\prime}_\alpha(s-) 
				d \left(\frac{\hat{\bar A}^{n^\prime} (s)}{\mu} \right)  
		-
		\int_0^t
			\tilde Q^{n^\prime}_\alpha(s) ds
	\right|		
	\to 0
$$	
as $n^\prime \to \infty$.
This limit holds on the arbitrarily chosen subsequence $\{n^\prime\}$.
Hence \eqref{eq:deltaNegligible.3} holds. 

Consider the fourth term on the right hand side of \eqref{eq:deltaNegligible}. The service times of initial jobs that are actually served constitute an i.i.d sequence.  This sequence obeys a weak law of large numbers, as does the renewal process, $\hat{S}^n$,  constructed by these service times over intervals of time that are of the order $1/\sqrt{n}$.  Hence, by Proposition \ref{prop:tightness} 

\begin{equation} \label{eq:deltaNegligible.4}
	\Prob \left(
		\sup_{t \le W^n(0)}
		\left|
			\hat{S}^n(t) 
			-
			\mu_n t
		\right|
		>
		\epsilon \sqrt{n}
	\right)
	< \Prob \left( \sup_{t \le K / \sqrt{n}} \left| \hat{S}^n(t) - \mu_n t \right| > \epsilon \sqrt{n} \right) + 
	\Prob\left( W^n(0) > K / \sqrt{n} \right) 
	<
	\eta,
\end{equation}
 holds for sufficiently large $n$.

Finally, our result follows from \eqref{eq:deltaNegligible}, \eqref{eq:deltaNegligible.1a}, \eqref{eq:deltaNegligible.1b}, 
	\eqref{eq:deltaNegligible.2},  \eqref{eq:deltaNegligible.3}, and
	Lemma \ref{lem:boundedArrivals}.
\end{proofof}

\subsection{Numerical Tables}

\input{tab_ticket1.tex}
\input{tab_ticket2.tex}

\input{tab_ticket3.tex}
\input{tab_ticket4.tex}

\end{document}

%% file: tab_ticket1.tex
\begin{table}[ht]
\captionsetup{justification=centering}
 \caption{Simulated Results vs. Heavy Traffic Approximations 
Arrivals = Exp($\mu + \sqrt{\mu} \cdot \beta $), Service = Exp($\mu$), Balking = Exp($\theta_b$), 
Reneging = Exp($\theta_r$) }
\centering
\scalebox{0.8}{
\begin{tabular}{c c c c c c c c c c c c c c c }
\hline\hline
$Q_{\alpha}$ & $\rho$ & $\mu$ & $\beta$ & $\theta_b$ & $\theta_r$ & Q & $Q_{ROU}$ & W & $W_{ROU}$ & R  & $R_{ROU}$   & B & $B_{ROU}$  \\ [0.5ex] 
\hline
$Q_S$    	&1.1    	&100        	&1            	&.1 		&.1                    	     & 49.48 $\pm$ .199 & 50.74 		& .481 $\pm$ .002 & .5074 		& .044 $\pm$ .0002 & .0461 		& .048 $\pm$ .0002 & .0461 	\\
$Q_T$    	&1.1   	&100       	&1            	&.1 		&.1                   	     & 50.18 $\pm$ .203 & 50.74 		& .475 $\pm$ .002 & .5074 		& .043 $\pm$ .0002 & .0461 		& .049 $\pm$ .0002 & .0461 	\\[1ex]

$Q_S$    	&1.1    	&25          	&.5           	&.1 		&.1                   	     & 15.25 $\pm$ .071 & 15.24 		& .587 $\pm$ .003 & .6099 		& .0519 $\pm$ .0002 & .0554 		& .0585 $\pm$ .0003 & .0554	\\
$Q_T$    	&1.1   	&25    	&.5           	&.1 		&.1  			     & 15.59 $\pm$ .073 & 15.24 		& .578 $\pm$ .003 & .6099 		& .0510 $\pm$ .0002 & .0554 		& .0597 $\pm$ .0003 & .0554	\\[1ex]

$Q_S$    	&1.1   	&4      	&.2          	&.1 		&.1 			     & 4.47 $\pm$ .021 & 4.40 		& 1.04 $\pm$ .0006 & 1.100 	& .0813 $\pm$ .0004 & .1000 		& .102 $\pm$ .0006 & .1000 	\\ 
$Q_T$    	&1.1   	&4        	&.2          	&.1 		&.1 			     & 4.67 $\pm$ .023 & 4.40 		& 1.02 $\pm$ .0006 & 1.100 	& .0791 $\pm$ .0004 & .1000  		& .106 $\pm$ .0006 & .1000      \\[1ex] \hline 

$Q_S$    	&1.01  	&100        	&.1         	&.1 		&.1 			     & 19.79 $\pm$ .094 & 19.78 		& .195 $\pm$ .001 & .1978 		& .0186 $\pm$ .0001 & .0195 		& .0195 $\pm$ .0001 & .0195     \\
$Q_T$    	&1.01 	&100        	&.1         	&.1 		&.1 			     & 19.98 $\pm$ .095 & 19.78 		& .194 $\pm$ .001 & .1978 		& .0186 $\pm$ .0001 & .0195 		& .0196 $\pm$ .0001 & .0195	\\[1ex]

$Q_S$    	&1.01  	&25          	&.05       	&.1 		&.1 			     & 9.34 $\pm$ .045 & 9.39 		& .363 $\pm$ .0018 & .3956	& .0333 $\pm$ .0002 & .03718 		& .0362 $\pm$ .0002 & .03718	\\
$Q_T$    	&1.01 	&25          	&.05       	&.1 		&.1 			     & 9.51 $\pm$ .046 & 9.39 		& .360 $\pm$ .0018 & .3956 	& .0330 $\pm$ .0002 & .03718		& .0368 $\pm$ .0002 & .03718	\\[1ex]

$Q_S$    	&1.01  	&4            	&.02       	&.1 		&.1 			     & 3.61 $\pm$ .019 & 3.64 		& .855 $\pm$ .005 & .9104 		& .0681 $\pm$ .0005 & .0901 		& .0843 $\pm$ .0006 & .0901	\\
$Q_T$    	&1.01 	&4            	&.02       	&.1 		&.1 			     & 3.77 $\pm$ .020 & 3.64 		& .842 $\pm$ .005 & .9104 		& .0667 $\pm$ .0004 & .0901		& .0873 $\pm$ .0006 & .0901	\\[1ex]\hline

$Q_S$    	&1       	&100       	&0           	&.1 		&.1 			     & 17.82 $\pm$ .083 & 17.84 		& .176 $\pm$ .0009 & .1784	& .0168 $\pm$ .0001 & .0178 		& .0176 $\pm$ .0001 & .0178	\\
$Q_T$    	&1       	&100      	&0           	&.1 		&.1 			     & 17.98 $\pm$ .085 & 17.84 		& .175 $\pm$ .0009 & .1784 	& .0168 $\pm$ .0001 & .0178 		& .0177 $\pm$ .0001 & .0178	\\[1ex]

$Q_S$    	&1       	&25         	&0           	&.1 		&.1 			     & 8.86 $\pm$ .046 & 8.92 		& .345 $\pm$ .002 & .3568 		& .0317 $\pm$ .0002 & .0356 		& .0344 $\pm$ .0002 & .0356	\\
$Q_T$    	&1       	&25        	&0           	&.1 		&.1 			     & 9.03 $\pm$ .047 & 8.92 		& .343 $\pm$ .002 & .3568 		& .0314 $\pm$ .0002 & .0356 		& .0350 $\pm$ .0002 & .0356	\\[1ex]

$Q_S$    	&1       	&4           	&0           	&.1 		&.1 			     & 3.52 $\pm$ .019 & 3.56 		& .832 $\pm$ .005 & .8920 		& .0670 $\pm$ .0005 & .0892 		& .0817 $\pm$ .0005 & .0892	\\
$Q_T$    	&1       	&4          	&0           	&.1 		&.1 			     & 3.67 $\pm$ .021 & 3.56 		& .820 $\pm$ .005 & .8920 		& .0657 $\pm$ .0005 & .0892 		& .0848 $\pm$ .0005 & .0892	\\[1ex]
\hline

$Q_S$    	&.99    	&100        	&-.1           &.1 		& .1                    	     & 15.95 $\pm$ .067 & 16.14 		& .157 $\pm$ .0007 & .1614 	& .0152 $\pm$ .0001 & .0163 		& .0157 $\pm$ .0001 & .0163 	\\
$Q_T$    	&.99   	&100       	&-.1           &.1 		& .1                   	     & 16.10 $\pm$ .068 & 16.14 		& .157 $\pm$ .0007 & .1614 	& .0151 $\pm$ .0001 & .0163		& .0159 $\pm$ .0001 & .0163 	\\[1ex]

$Q_S$    	&.99    	&25          	&-.05         	&.1 		& .1                   	     & 8.40 $\pm$ .046 & 8.48 		& .328 $\pm$ .002 & .3392 		& .0302 $\pm$ .0002 & .0342 		& .0326 $\pm$ .0002 & .0342	\\
$Q_T$    	&.99   	&25    	&-.05         &.1 		& .1  			     & 8.55 $\pm$ .046 & 8.48 		& .325 $\pm$ .002 & .3392 		& .0299 $\pm$ .0002 & .0342 		& .0332 $\pm$ .0002 & .0342	\\[1ex]

$Q_S$    	&.99   	&4      	&-.02         &.1 		& .1 			     & 3.44 $\pm$ .020 & 3.49 		& .814 $\pm$ .0058 & .8741 	& .0657 $\pm$ .0005 & .0882 		& .0801 $\pm$ .0005 & .0882	\\ 
$Q_T$    	&.99   	&4        	&-.02         &.1 		& .1 			     & 3.58 $\pm$ .022 & 3.49 		& .803 $\pm$ .0055 & .8741 	& .0644 $\pm$ .0005 & .0882 		& .0830 $\pm$ .0005 & .0882     \\[1ex]\hline

$Q_S$    	&.9  		&100        	&-1         	&.1 		& .1 			     & 7.12 $\pm$ .0264 & 7.77 		& .0708 $\pm$ .0003 & .0777 	& .00692 $\pm$ .0001 & .0086 		& .0071 $\pm$ .0001 & .0086      \\
$Q_T$    	&.9 		&100        	&-1         	&.1 		& .1 			     & 7.16 $\pm$ .0267 & 7.77 		& .0707 $\pm$ .0003 & .0777	& .00691 $\pm$ .0001 & .0086  		& .0071 $\pm$ .0001 & .0086 	\\[1ex]

$Q_S$    	&.9  		&25          	&-.5      	&.1 		& .1 			     & 5.19 $\pm$ .024 & 5.61 		& .204 $\pm$ .001 & .2246		& .0192 $\pm$ .0001 & .0249 		& .0204 $\pm$ .0001 & .0249	\\
$Q_T$    	&.9 		&25          	&-.5       	&.1 		& .1 			     & 5.27 $\pm$ .024 & 5.61 		& .203 $\pm$ .001 & .2246 		& .0192 $\pm$ .0001 & .0249		& .0207 $\pm$ .0001 & .0249	\\[1ex]

$Q_S$    	&.9  		&4            	&-.2       	&.1 		& .1 			     & 2.74 $\pm$ .013 & 2.93 		& .653 $\pm$ .0036 & .7328 	& .0542 $\pm$ .0004 & .0814 		& .0645 $\pm$ .0005 & .0814	\\
$Q_T$    	&.9 		&4            	&-.2       	&.1 		& .1 			     & 2.84 $\pm$ .014 & 2.93 		& .647 $\pm$ .0035 & .7328 	& .0534 $\pm$ .0004 & .0814 		& .0667 $\pm$ .0005 & .0814	\\[1ex]\hline

$Q_S$    	&.8      	&100       	&-2           	&.1 		& .1 			     & 3.73 $\pm$ .0098 & 4.59 		& .0372 $\pm$ .0001 & .0459	& .0037 $\pm$ .0001 & .0057 		& .0037 $\pm$ .0001 & .0057 	\\
$Q_T$    	&.8       	&100      	&-2           	&.1 		& .1 			     & 3.75 $\pm$ .0099 & 4.59 		& .0372 $\pm$ .0001 & .0459 	& .0037 $\pm$ .0001 & .0057 		& .0037 $\pm$ .0001 & .0057 	\\[1ex]

$Q_S$    	&.8       	&25         	&-1           	&.1 		& .1 			     & 3.22 $\pm$ .011 & 3.88 		& .127 $\pm$ .0005 & .1555 	& .0122 $\pm$ .0001 & .0194 		& .0127 $\pm$ .0001 & .0194	\\
$Q_T$    	&.8       	&25        	&-1           	&.1 		& .1 			     & 3.25 $\pm$ .012 & 3.88 		& .127 $\pm$ .0005 & .1555 	& .0122 $\pm$ .0001 & .0194 		& .0129 $\pm$ .0001 & .0194	\\[1ex]

$Q_S$    	&.8      	&4           	&-.4           	&.1 		& .1 			     & 2.11 $\pm$ .012 & 2.44 		& .508 $\pm$ .0036 & .6113 	& .0432 $\pm$ .0003 & .0764 		& .0496 $\pm$ .0003 & .0764	\\
$Q_T$    	&.8       	&4          	&-.4           	&.1 		& .1 			     & 2.18 $\pm$ .013 & 2.44 		& .505 $\pm$ .0035 & .6113 	& .0428 $\pm$ .0003 & .0764 		& .0511 $\pm$ .0003 & .0764	\\[1ex]
\hline
\end{tabular}}
\label{tab:ticket_1}
\end{table}

%% file: tab_ticket2.tex
\begin{table}[ht]
\captionsetup{justification=centering}
 \caption{Simulated Results vs. Heavy Traffic Approximations 
Arrivals = Exp($\mu + \sqrt{\mu} \cdot \beta $), Service = LogNormal($1/\mu$,1/$\mu^2$), \\
Balking = Uniform(0,$1/\theta_b$), Reneging = Uniform(0,$1/\theta_r$) }
\centering
\scalebox{0.8}{
\begin{tabular}{c c c c c c c c c c c c c c c }
\hline\hline
$Q_{\alpha}$ & $\rho$ & $\mu$ & $\beta$ & $\theta_b$ & $\theta_r$ & Q & $Q_{ROU}$ & W & $W_{ROU}$ & R  & $R_{ROU}$   & B & $B_{ROU}$  \\ [0.5ex] 
\hline
$Q_S$    	&1.1    	&100        	&1            	&.1 		&.1                    	     & 48.17 $\pm$ .174 & 50.74 		& .468 $\pm$ .002 & .5074 		& .044 $\pm$ .0002 & .0461 		& .048 $\pm$ .0002 & .0461 	\\
$Q_T$    	&1.1   	&100       	&1            	&.1 		&.1                   	     & 48.85 $\pm$ .177 & 50.74 		& .462 $\pm$ .002 & .5074 		& .043 $\pm$ .0002 & .0461 		& .049 $\pm$ .0002 & .0461 	\\[1ex]

$Q_S$    	&1.1    	&25          	&.5           	&.1 		&.1                   	     & 14.65 $\pm$ .073 & 15.24 		& .563 $\pm$ .003 & .6099 		& .0517 $\pm$ .0002 & .0554 		& .0587 $\pm$ .0003 & .0554	\\
$Q_T$    	&1.1   	&25    	&.5           	&.1 		&.1  			     & 14.97 $\pm$ .076 & 15.24 		& .555 $\pm$ .003 & .6099 		& .0508 $\pm$ .0002 & .0554 		& .0599 $\pm$ .0003 & .0554	\\[1ex]

$Q_S$    	&1.1   	&4      	&.2          	&.1 		&.1 			     & 4.17 $\pm$ .022 & 4.40 		& .968 $\pm$ .0006 & 1.100 	& .0813 $\pm$ .0004 & .1000 		& .104 $\pm$ .0006 & .1000 	\\ 
$Q_T$    	&1.1   	&4        	&.2          	&.1 		&.1 			     & 4.35 $\pm$ .024 & 4.40 		& .946 $\pm$ .0006 & 1.100 	& .0786 $\pm$ .0004 & .1000  		& .108 $\pm$ .0006 & .1000      \\[1ex] \hline 

$Q_S$    	&1.01  	&100        	&.1         	&.1 		&.1 			     & 19.35 $\pm$ .086 & 19.78 		& .191 $\pm$ .0009 & .1978 	& .0185 $\pm$ .0001 & .0195 		& .0194 $\pm$ .0001 & .0195     \\
$Q_T$    	&1.01 	&100        	&.1         	&.1 		&.1 			     & 19.55 $\pm$ .087 & 19.78 		& .189 $\pm$ .0009 & .1978 	& .0184 $\pm$ .0001 & .0195 		& .0196 $\pm$ .0001 & .0195	\\[1ex]

$Q_S$    	&1.01  	&25          	&.05       	&.1 		&.1 			     & 9.04 $\pm$ .047 & 9.39 		& .352 $\pm$ .0020 & .3956	& .0331 $\pm$ .0002 & .03718 		& .0362 $\pm$ .0002 & .03718	\\
$Q_T$    	&1.01 	&25          	&.05       	&.1 		&.1 			     & 9.22 $\pm$ .048 & 9.39 		& .349 $\pm$ .00189 & .3956 	& .0328 $\pm$ .0002 & .03718		& .0369 $\pm$ .0002 & .03718	\\[1ex]

$Q_S$    	&1.01  	&4            	&.02       	&.1 		&.1 			     & 3.40 $\pm$ .019 & 3.64 		& .795 $\pm$ .006 & .9104 		& .0684 $\pm$ .0005 & .0901 		& .0849 $\pm$ .0005 & .0901	\\
$Q_T$    	&1.01 	&4            	&.02       	&.1 		&.1 			     & 3.54 $\pm$ .020 & 3.64 		& .792 $\pm$ .005 & .9104 		& .0667 $\pm$ .0005 & .0901		& .0884 $\pm$ .0006 & .0901	\\[1ex]\hline

$Q_S$    	&1       	&100       	&0           	&.1 		&.1 			     & 17.45 $\pm$ .082 & 17.84 		& .172 $\pm$ .0009 & .1784	& .0167 $\pm$ .0001 & .0178 		& .0174 $\pm$ .0001 & .0178	\\
$Q_T$    	&1       	&100      	&0           	&.1 		&.1 			     & 17.62 $\pm$ .084 & 17.84 		& .171 $\pm$ .0009 & .1784 	& .0167 $\pm$ .0001 & .0178 		& .0176 $\pm$ .0001 & .0178	\\[1ex]

$Q_S$    	&1       	&25         	&0           	&.1 		&.1 			     & 8.62 $\pm$ .037 & 8.92 		& .335 $\pm$ .002 & .3568 		& .0317 $\pm$ .0002 & .0356 		& .0344 $\pm$ .0002 & .0356	\\
$Q_T$    	&1       	&25        	&0           	&.1 		&.1 			     & 8.78 $\pm$ .038 & 8.92 		& .333 $\pm$ .002 & .3568 		& .0314 $\pm$ .0002 & .0356 		& .0351 $\pm$ .0002 & .0356	\\[1ex]

$Q_S$    	&1       	&4           	&0           	&.1 		&.1 			     & 3.30 $\pm$ .016 & 3.56 		& .773 $\pm$ .006 & .8920 		& .0663 $\pm$ .0005 & .0892 		& .0825 $\pm$ .0005 & .0892	\\
$Q_T$    	&1       	&4          	&0           	&.1 		&.1 			     & 3.44 $\pm$ .017 & 3.56 		& .761 $\pm$ .005 & .8920 		& .0647 $\pm$ .0005 & .0892 		& .0860 $\pm$ .0005 & .0892	\\[1ex]
\hline

$Q_S$    	&.99    	&100        	&-.1           &.1 		& .1                    	     & 15.78 $\pm$ .077 & 16.14 		& .156 $\pm$ .0008 & .1614 	& .0152 $\pm$ .0001 & .0163 		& .0157 $\pm$ .0001 & .0163 	\\
$Q_T$    	&.99   	&100       	&-.1           &.1 		& .1                   	     & 15.92 $\pm$ .078 & 16.14 		& .155 $\pm$ .0008 & .1614 	& .0151 $\pm$ .0001 & .0163		& .0159 $\pm$ .0001 & .0163 	\\[1ex]

$Q_S$    	&.99    	&25          	&-.05         	&.1 		& .1                   	     & 8.18 $\pm$ .040 & 8.48 		& .319 $\pm$ .002 & .3392 		& .0301 $\pm$ .0002 & .0342 		& .0327 $\pm$ .0002 & .0342	\\
$Q_T$    	&.99   	&25    	&-.05         &.1 		& .1  			     & 8.32 $\pm$ .042 & 8.48 		& .317 $\pm$ .002 & .3392 		& .0298 $\pm$ .0002 & .0342 		& .0333 $\pm$ .0002 & .0342	\\[1ex]

$Q_S$    	&.99   	&4      	&-.02         &.1 		& .1 			     & 3.24 $\pm$ .018 & 3.49 		& .761 $\pm$ .006 & .8741 		& .0655 $\pm$ .0005 & .0882 		& .0812 $\pm$ .0005 & .0882	\\ 
$Q_T$    	&.99   	&4        	&-.02         &.1 		& .1 			     & 3.38 $\pm$ .019 & 3.49 		& .750 $\pm$ .006 & .8741 		& .0640 $\pm$ .0004 & .0882 		& .0846 $\pm$ .0005 & .0882     \\[1ex]\hline

$Q_S$    	&.9  		&100        	&-1         	&.1 		& .1 			     & 7.09 $\pm$ .0252 & 7.77 		& .0705 $\pm$ .0003 & .0777 	& .00694 $\pm$ .0001 & .0086 		& .0071 $\pm$ .0001 & .0086      \\
$Q_T$    	&.9 		&100        	&-1         	&.1 		& .1 			     & 7.13 $\pm$ .0256 & 7.77 		& .0704 $\pm$ .0003 & .0777	& .00693 $\pm$ .0001 & .0086  		& .0071 $\pm$ .0001 & .0086 	\\[1ex]

$Q_S$    	&.9  		&25          	&-.5      	&.1 		& .1 			     & 5.10 $\pm$ .027 & 5.61 		& .200 $\pm$ .001 & .2246		& .0193 $\pm$ .0001 & .0249 		& .0205 $\pm$ .0001 & .0249	\\
$Q_T$    	&.9 		&25          	&-.5       	&.1 		& .1 			     & 5.18 $\pm$ .028 & 5.61 		& .200 $\pm$ .001 & .2246 		& .0192 $\pm$ .0001 & .0249		& .0207 $\pm$ .0001 & .0249	\\[1ex]

$Q_S$    	&.9  		&4            	&-.2       	&.1 		& .1 			     & 2.61 $\pm$ .015 & 2.93 		& .619 $\pm$ .004 & .7328 		& .0541 $\pm$ .0004 & .0814 		& .0655 $\pm$ .0005 & .0814	\\
$Q_T$    	&.9 		&4            	&-.2       	&.1 		& .1 			     & 2.71 $\pm$ .016 & 2.93 		& .612 $\pm$ .004 & .7328 		& .0532 $\pm$ .0004 & .0814 		& .0680 $\pm$ .0005 & .0814	\\[1ex]\hline

$Q_S$    	&.8      	&100       	&-2           	&.1 		& .1 			     & 3.72 $\pm$ .0098 & 4.59 		& .0371 $\pm$ .0001 & .0459	& .0037 $\pm$ .0001 & .0057 		& .0037 $\pm$ .0001 & .0057 	\\
$Q_T$    	&.8       	&100      	&-2           	&.1 		& .1 			     & 3.74 $\pm$ .0099 & 4.59 		& .0371 $\pm$ .0001 & .0459 	& .0037 $\pm$ .0001 & .0057 		& .0037 $\pm$ .0001 & .0057 	\\[1ex]

$Q_S$    	&.8       	&25         	&-1           	&.1 		& .1 			     & 3.19 $\pm$ .013 & 3.88 		& .126 $\pm$ .0006 & .1555 	& .0123 $\pm$ .0001 & .0194 		& .0127 $\pm$ .0001 & .0194	\\
$Q_T$    	&.8       	&25        	&-1           	&.1 		& .1 			     & 3.22 $\pm$ .014 & 3.88 		& .126 $\pm$ .0006 & .1555 	& .0122 $\pm$ .0001 & .0194 		& .0129 $\pm$ .0001 & .0194	\\[1ex]

$Q_S$    	&.8      	&4           	&-.4           	&.1 		& .1 			     & 2.02 $\pm$ .011 & 2.44 		& .481 $\pm$ .0034 & .6113 	& .0429 $\pm$ .0004 & .0764 		& .0505 $\pm$ .0004 & .0764	\\
$Q_T$    	&.8       	&4          	&-.4           	&.1 		& .1 			     & 2.08 $\pm$ .011 & 2.44 		& .478 $\pm$ .0033 & .6113 	& .0424 $\pm$ .0004 & .0764 		& .0522 $\pm$ .0004 & .0764	\\[1ex]
\hline
\end{tabular}}
\label{tab:ticket_2}
\end{table}

%% file: tab_ticket3.tex
\begin{table}[ht]
\captionsetup{justification=centering}
 \caption{Simulated Results vs. Heavy Traffic Approximations 
Arrivals = Exp($\mu + \sqrt{\mu} \cdot \beta $), Service = Exp($\mu$), Balking = Exp($\theta_b$), 
Reneging = Exp($\theta_r$) }
\centering
\scalebox{0.8}{
\begin{tabular}{c c c c c c c c c c c c c c c }
\hline\hline
$Q_{\alpha}$ & $\rho$ & $\mu$ & $\beta$ & $\theta_b$ & $\theta_r$ & Q & $Q_{ROU}$ & W & $W_{ROU}$ & R  & $R_{ROU}$   & B & $B_{ROU}$  \\ [0.5ex] 
\hline
$Q_S$    	&1.1    	&100        	&1            	&1 		&1                    	     & 7.96 $\pm$ .012 & 7.88 		& .075 $\pm$ 0001 & .0788 		& .0637 $\pm$ .0001 & .0717 		& .075 $\pm$ .0001 & .0717  	\\
$Q_T$    	&1.1   	&100       	&1            	&1 		&1                   	     & 8.21 $\pm$ .013 & 7.88 		& .074 $\pm$ .0001 & .0788 		& .0624 $\pm$ .0001 & .0717  		& .077 $\pm$ .0001 & .0717  	\\[1ex]

$Q_S$    	&1.1    	&25          	&.5           	&1 		&1                   	     & 3.38  $\pm$ .006 & 3.32 		& .125  $\pm$ .0002  & .133 		& .122 $\pm$ .0001  & .1209 		& .091  $\pm$ .0002  & .1209	\\
$Q_T$    	&1.1   	&25    	&.5           	&1 		&1  			     & 3.56 $\pm$ .006 & 3.32 		& .122  $\pm$ .0003 & .133 		& .0088  $\pm$ .0002 & .1209 		& .127  $\pm$ .0002 & .1209	\\[1ex]

$Q_S$    	&1.1   	&4      	&.2          	&1 		&1 			     & 1.24 $\pm$ .003 & 1.20		& .277  $\pm$ .0009  & .301 		& .124  $\pm$ .0004  & .2736 		& .240  $\pm$ .0005  & .2736 	\\ 
$Q_T$    	&1.1   	&4        	&.2          	&1 		&1 			     & 1.37 $\pm$ .003 & 1.20 		& .268  $\pm$ .0009  & .301 		& .117  $\pm$ .0004  & .2736  		& .255  $\pm$ .0005  & .2736      \\[1ex] \hline 

$Q_S$    	&1.01  	&100        	&.1         	&1 		&1 			     & 5.78 $\pm$ .009 & 5.82 		& .055  $\pm$ .0001  & .0582 		& .0483  $\pm$ .0001  & .0576 		& .0552  $\pm$ .0001  & .0576     \\
$Q_T$    	&1.01 	&100        	&.1         	&1 		&1 			     & 5.94 $\pm$ .009 & 5.82 		& .055  $\pm$ .0001  & .0582		& .0476  $\pm$ .0001 & .0576 		& .0567  $\pm$ .0001 & .0576	\\[1ex]

$Q_S$    	&1.01  	&25          	&.05       	&1 		&1 			     & 2.83 $\pm$ .005 & 2.86 		& .105 $\pm$ .0002 & .1146		& .079 $\pm$ .0002 & .1135 		& .103 $\pm$ .0002 & .1135	\\
$Q_T$    	&1.01 	&25          	&.05       	&1 		&1 			     & 2.97 $\pm$ .005 & 2.86 		& .103 $\pm$ .0002 & .1146 		& .077 $\pm$ .0002 & .1135		& .107 $\pm$ .0002 & .1135	\\[1ex]

$Q_S$    	&1.01  	&4            	&.02       	&1 		&1 			     & 1.14 $\pm$ .003 & 1.13 		& .255 $\pm$ .0009 & .2839 		& .118 $\pm$ .0005 & .2811 		& .222 $\pm$ .0006 & .2811	\\
$Q_T$    	&1.01 	&4            	&.02       	&1 		&1 			     & 1.25 $\pm$ .003 & 1.13 		& .249 $\pm$ .0009 & .2839 		& .112 $\pm$ .0004 & .2811		& .235 $\pm$ .0006 & .2811	\\[1ex]\hline

$Q_S$    	&1       	&100       	&0           	&1 		&1 			     & 5.58 $\pm$ .009 & 5.64 		& .0537 $\pm$ .0001 & .0564		& .047 $\pm$ .0001 & .0564 		& .053 $\pm$ .0001 & .0564	\\
$Q_T$    	&1       	&100      	&0           	&1 		&1 			     & 5.74 $\pm$ .009 & 5.64 		& .0531 $\pm$ .0001 & .0564 		& .046 $\pm$ .0001 & .0564		& .055 $\pm$ .0001 & .0564	\\[1ex]

$Q_S$    	&1       	&25         	&0           	&1 		&1 			     & 2.78 $\pm$ .005 & 2.82 		& .104 $\pm$ .0002 & .1128 		& .078 $\pm$ .0002 & .1128 		& .102 $\pm$ .0002 & .1128 	\\
$Q_T$    	&1       	&25        	&0           	&1 		&1 			     & 2.92 $\pm$ .005 & 2.82 		& .102 $\pm$ .0002 & .1128 		& .076 $\pm$ .0002 & .1128  		& .106 $\pm$ .0002 & .1128 	\\[1ex]

$Q_S$    	&1       	&4           	&0           	&1 		&1 			     & 1.12 $\pm$ .003 & 1.12 		& .253 $\pm$ .0010 & .2820 		& .117 $\pm$ .0005 & .2820 		& .220 $\pm$ .0006 & .2820	\\
$Q_T$    	&1       	&4          	&0           	&1 		&1 			     & 1.23 $\pm$ .003 & 1.12 		& .247 $\pm$ .0009 & .2820 		& .110 $\pm$ .0004 & .2820 		& .233 $\pm$ .0006 & .2820	\\[1ex]
\hline

$Q_S$    	&.99    	&100        	&-.1           &1 		& 1                    	     & 5.39 $\pm$ .009 & 5.46 		& .052 $\pm$ .0001 & .0551 		& .045 $\pm$ .0001 & .0546 		& .052 $\pm$ .0001 & .0546 	\\
$Q_T$    	&.99   	&100       	&-.1           &1 		& 1                   	     & 5.54 $\pm$ .009 & 5.46 		& .051 $\pm$ .0001 & .0551 		& .045 $\pm$ .0001 & .0546		& .053 $\pm$ .0001 & .0546	\\[1ex]

$Q_S$    	&.99    	&25          	&-.05         	&1 		& 1                   	     & 2.73 $\pm$ .005 & 2.77 		& .102 $\pm$ .0002 & .1110 		& .077 $\pm$ .0002 & .1121 		& .100 $\pm$ .0002 & .1121	\\
$Q_T$    	&.99   	&25    	&-.05         &1 		& 1  			     & 2.86 $\pm$ .006 & 2.77 		& .100 $\pm$ .0002 & .1110 		& .075 $\pm$ .0002 & .1121 		& .104 $\pm$ .0002 & .1121	\\[1ex]

$Q_S$    	&.99   	&4      	&-.02         &1 		& 1 			     & 1.11 $\pm$ .003 & 1.12 		& .251 $\pm$ .0009 & .2802 		& .117 $\pm$ .0005 & .2831 		& .219 $\pm$ .0006 & .2831	\\ 
$Q_T$    	&.99   	&4        	&-.02         &1 		& 1 			     & 1.22 $\pm$ .003 & 1.12 		& .245 $\pm$ .0009 & .2802 		& .110 $\pm$ .0005 & .2831 		& .232 $\pm$ .0006 & .2831     \\[1ex]\hline

$Q_S$    	&.9  		&100        	&-1         	&1 		& 1 			     & 3.88 $\pm$ .005 & 4.16 		& .0377 $\pm$ .0001 & .0416 		& .034 $\pm$ .0001 & .0462 		& .037 $\pm$ .0001 & .0462      \\
$Q_T$    	&.9 		&100        	&-1         	&1 		& 1 			     & 3.97 $\pm$ .006 & 4.16 		& .0375 $\pm$ .0001 & .0416		& .034 $\pm$ .0001 & .0462 		& .038 $\pm$ .0001 & .0462 	\\[1ex]

$Q_S$    	&.9  		&25          	&-.5      	&1 		& 1 			     & 2.25 $\pm$ .005 & 2.41 		& .085 $\pm$ .0002 & .0964		& .066 $\pm$ .0002 & .1071 		& .083 $\pm$ .0002 & .1071	\\
$Q_T$    	&.9 		&25          	&-.5       	&1 		& 1 			     & 2.36 $\pm$ .005 & 2.41 		& 084 $\pm$ .0002 & .0964 		& .065 $\pm$ .0002 & .1071		& .087 $\pm$ .0002 & .1071	\\[1ex]

$Q_S$    	&.9  		&4            	&-.2       	&1 		& 1 			     & 1.01 $\pm$ .003 & 1.05 		& .229 $\pm$ .0009 & .2646 		& .109 $\pm$ .0005 & .2940 		& .200 $\pm$ .0006 & .2940	\\
$Q_T$    	&.9 		&4            	&-.2       	&1 		& 1 			     & 1.10 $\pm$ .003 & 1.05 		& .224 $\pm$ .0009 & .2646 		& .103 $\pm$ .0004 & .2940 		& .211 $\pm$ .0007 & .2940	\\[1ex]\hline

$Q_S$    	&.8      	&100       	&-2           	&1 		& 1 			     & 2.70 $\pm$ .005 & 3.19 		& .026 $\pm$ .0001 & .0319		& .024 $\pm$ .0001 & .0399 		& .026 $\pm$ .0001 & .0399 	\\
$Q_T$    	&.8       	&100      	&-2           	&1 		& 1 			     & 2.75 $\pm$ .005 & 3.19 		& .026 $\pm$ .0001 & .0319 		& .024 $\pm$ .0001 & .0399 		& .027 $\pm$ .0001 & .0399 	\\[1ex]

$Q_S$    	&.8       	&25         	&-1           	&1 		& 1 			     & 1.80 $\pm$ .003 & 2.08 		& .069 $\pm$ .0002 & .0832 		& .055 $\pm$ .0002 & .1040 		& .067 $\pm$ .0002 & .1040	\\
$Q_T$    	&.8       	&25        	&-1           	&1 		& 1 			     & 1.87 $\pm$ .003 & 2.08 		& .068 $\pm$ .0002 & .0832 		& .054 $\pm$ .0002 & .1040 		& .069 $\pm$ .0002 & .1040	\\[1ex]

$Q_S$    	&.8      	&4           	&-.4           	&1 		& 1 			     & .89 $\pm$ .002 & .9947 		& .204 $\pm$ .0008 & .2486 		& .100 $\pm$ .0005 & .3108 		& .178 $\pm$ .0007 & .3108	\\
$Q_T$    	&.8       	&4          	&-.4           	&1 		& 1 			     & .97 $\pm$ .002 & .9947 		& .200 $\pm$ .0008 & .2486 		& .095 $\pm$ .0004 & .3108 		& .188 $\pm$ .0007 & .3108	\\[1ex]
\hline
\end{tabular}}
\label{tab:ticket_3}
\end{table}

%% file: tab_ticket4.tex
\begin{table}[ht]
\captionsetup{justification=centering}
 \caption{Simulated Results vs. Heavy Traffic Approximations 
Arrivals = Exp($\mu + \sqrt{\mu} \cdot \beta $), Service = LogNormal($1/\mu$,1/$\mu^2$), \\
Balking = Uniform(0,$1/\theta_b$), Reneging = Uniform(0,$1/\theta_r$) }
\centering
\scalebox{0.8}{
\begin{tabular}{c c c c c c c c c c c c c c c }
\hline\hline
$Q_{\alpha}$ & $\rho$ & $\mu$ & $\beta$ & $\theta_b$ & $\theta_r$ & Q & $Q_{ROU}$ & W & $W_{ROU}$ & R  & $R_{ROU}$   & B & $B_{ROU}$  \\ [0.5ex] 
\hline
$Q_S$    	&1.1    	&100        	&1            	&1 		&1                    	     & 7.56 $\pm$ .011 & 7.88 		& .072 $\pm$ 0001 & .0788 		& .0636 $\pm$ .0001 & .0717 		& .076 $\pm$ .0001 & .0717  	\\
$Q_T$    	&1.1   	&100       	&1            	&1 		&1                   	     & 7.79 $\pm$ .012 & 7.88 		& .070 $\pm$ .0001 & .0788 		& .0622 $\pm$ .0001 & .0717  		& .078 $\pm$ .0001 & .0717  	\\[1ex]

$Q_S$    	&1.1    	&25          	&.5           	&1 		&1                   	     & 3.11  $\pm$ .005 & 3.32 		& .114  $\pm$ .0003  & .133 		& .091 $\pm$ .0001  & .1209 		& .124  $\pm$ .0002  & .1209	\\
$Q_T$    	&1.1   	&25    	&.5           	&1 		&1  			     & 3.27 $\pm$ .006 & 3.32 		& .111  $\pm$ .0002 & .133 		& .088  $\pm$ .0002 & .1209 		& .131  $\pm$ .0002 & .1209	\\[1ex]

$Q_S$    	&1.1   	&4      	&.2          	&1 		&1 			     & 1.06 $\pm$ .002 & 1.20		& .232  $\pm$ .001  & .301 			& .118  $\pm$ .0005  & .2736 		& .266  $\pm$ .0006  & .2736 	\\ 
$Q_T$    	&1.1   	&4        	&.2          	&1 		&1 			     & 1.15 $\pm$ .002 & 1.20 		& .225  $\pm$ .001  & .301 			& .105  $\pm$ .0004  & .2736  		& .289  $\pm$ .0007  & .2736      \\[1ex] \hline 

$Q_S$    	&1.01  	&100        	&.1         	&1 		&1 			     & 5.57 $\pm$ .009 & 5.82 		& .052  $\pm$ .0001  & .0582 		& .0470  $\pm$ .0001  & .0576 		& .0557  $\pm$ .0001  & .0576     \\
$Q_T$    	&1.01 	&100        	&.1         	&1 		&1 			     & 5.73 $\pm$ .009 & 5.82 		& .051  $\pm$ .0001  & .0582		& .0463  $\pm$ .0001 & .0576 		& .0573  $\pm$ .0001 & .0576	\\[1ex]

$Q_S$    	&1.01  	&25          	&.05       	&1 		&1 			     & 2.65 $\pm$ .004 & 2.86 		& .098 $\pm$ .0002 & .1146		& .080 $\pm$ .0002 & .1135 		& .106 $\pm$ .0002 & .1135	\\
$Q_T$    	&1.01 	&25          	&.05       	&1 		&1 			     & 2.78 $\pm$ .005 & 2.86 		& .096 $\pm$ .0002 & .1146 		& .078 $\pm$ .0002 & .1135		& .111 $\pm$ .0002 & .1135	\\[1ex]

$Q_S$    	&1.01  	&4            	&.02       	&1 		&1 			     & .99 $\pm$ .002 & 1.13 		& .218 $\pm$ .001 & .2839 			& .113 $\pm$ .0005 & .2811 		& .247 $\pm$ .0006 & .2811	\\
$Q_T$    	&1.01 	&4            	&.02       	&1 		&1 			     & 1.07 $\pm$ .002 & 1.13 		& .211 $\pm$ .001 & .2839 			& .101 $\pm$ .0004 & .2811		& .268 $\pm$ .0007 & .2811	\\[1ex]\hline

$Q_S$    	&1       	&100       	&0           	&1 		&1 			     & 5.38 $\pm$ .008 & 5.64 		& .0516 $\pm$ .0001 & .0564		& .047 $\pm$ .0001 & .0564 		& .054 $\pm$ .0001 & .0564	\\
$Q_T$    	&1       	&100      	&0           	&1 		&1 			     & 5.53 $\pm$ .009 & 5.64 		& .0510 $\pm$ .0001 & .0564 		& .046 $\pm$ .0001 & .0564		& .055 $\pm$ .0001 & .0564	\\[1ex]

$Q_S$    	&1       	&25         	&0           	&1 		&1 			     & 2.59 $\pm$ .005 & 2.82 		& .095 $\pm$ .0003 & .1128 		& .079 $\pm$ .0002 & .1128 		& .104 $\pm$ .0002 & .1128 	\\
$Q_T$    	&1       	&25        	&0           	&1 		&1 			     & 2.72 $\pm$ .005 & 2.82 		& .094 $\pm$ .0002 & .1128 		& .077 $\pm$ .0002 & .1128  		& .109 $\pm$ .0002 & .1128 	\\[1ex]

$Q_S$    	&1       	&4           	&0           	&1 		&1 			     & .98 $\pm$ .002 & 1.12 		& .217 $\pm$ .001 & .2820 			& .113 $\pm$ .0005 & .2820 		& .245 $\pm$ .0006 & .2820	\\
$Q_T$    	&1       	&4          	&0           	&1 		&1 			     & 1.06 $\pm$ .002 & 1.12 		& .210 $\pm$ .001 & .2820 			& .100 $\pm$ .0004 & .2820 		& .265 $\pm$ .0006 & .2820	\\[1ex]
\hline

$Q_S$    	&.99    	&100        	&-.1           &1 		& 1                    	     & 5.18 $\pm$ .008 & 5.46 		& .049 $\pm$ .0001 & .0551 		& .045 $\pm$ .0001 & .0546 		& .052 $\pm$ .0001 & .0546 	\\
$Q_T$    	&.99   	&100       	&-.1           &1 		& 1                   	     & 5.33 $\pm$ .008 & 5.46 		& .049 $\pm$ .0001 & .0551 		& .045 $\pm$ .0001 & .0546		& .053 $\pm$ .0001 & .0546	\\[1ex]

$Q_S$    	&.99    	&25          	&-.05         	&1 		& 1                   	     & 2.54 $\pm$ .005 & 2.77 		& .094 $\pm$ .0002 & .1110 		& .078 $\pm$ .0002 & .1121 		& .102 $\pm$ .0002 & .1121	\\
$Q_T$    	&.99   	&25    	&-.05         &1 		& 1  			     & 2.67 $\pm$ .005 & 2.77 		& .093 $\pm$ .0002 & .1110 		& .075 $\pm$ .0002 & .1121 		& .108 $\pm$ .0002 & .1121	\\[1ex]

$Q_S$    	&.99   	&4      	&-.02         &1 		& 1 			     & .97 $\pm$ .002 & 1.12 		& .214 $\pm$ .0009 & .2802 		& .112 $\pm$ .0005 & .2831 		& .242 $\pm$ .0006 & .2831	\\ 
$Q_T$    	&.99   	&4        	&-.02         &1 		& 1 			     & 1.05 $\pm$ .002 & 1.12 		& .208 $\pm$ .0009 & .2802 		& .101 $\pm$ .0005 & .2831 		& .263 $\pm$ .0006 & .2831     \\[1ex]\hline

$Q_S$    	&.9  		&100        	&-1         	&1 		& 1 			     & 3.77 $\pm$ .007 & 4.16 		& .0364 $\pm$ .0001 & .0416 		& .034 $\pm$ .0001 & .0462 		& .038 $\pm$ .0001 & .0462      \\
$Q_T$    	&.9 		&100        	&-1         	&1 		& 1 			     & 3.86 $\pm$ .007 & 4.16 		& .0346 $\pm$ .0001 & .0416		& .034 $\pm$ .0001 & .0462 		& .039 $\pm$ .0001 & .0462 	\\[1ex]

$Q_S$    	&.9  		&25          	&-.5      	&1 		& 1 			     & 2.13 $\pm$ .004 & 2.41 		& .080 $\pm$ .0002 & .0964		& .068 $\pm$ .0002 & .1071 		& .085 $\pm$ .0002 & .1071	\\
$Q_T$    	&.9 		&25          	&-.5       	&1 		& 1 			     & 2.23 $\pm$ .004 & 2.41 		& .078 $\pm$ .0002 & .0964 		& .065 $\pm$ .0002 & .1071		& .089 $\pm$ .0002 & .1071	\\[1ex]

$Q_S$    	&.9  		&4            	&-.2       	&1 		& 1 			     & .89 $\pm$ .002 & 1.05 		& .198 $\pm$ .0009 & .2646 		& .106 $\pm$ .0005 & .2940 		& .223 $\pm$ .0006 & .2940	\\
$Q_T$    	&.9 		&4            	&-.2       	&1 		& 1 			     & .96 $\pm$ .002 & 1.05 		& .192 $\pm$ .0009 & .2646 		& .096 $\pm$ .0004 & .2940 		& .241 $\pm$ .0007 & .2940	\\[1ex]\hline

$Q_S$    	&.8      	&100       	&-2           	&1 		& 1 			     & 2.65 $\pm$ .004 & 3.19 		& .026 $\pm$ .0001 & .0319		& .024 $\pm$ .0001 & .0399 		& .026 $\pm$ .0001 & .0399 	\\
$Q_T$    	&.8       	&100      	&-2           	&1 		& 1 			     & 2.70 $\pm$ .004 & 3.19 		& .026 $\pm$ .0001 & .0319 		& .024 $\pm$ .0001 & .0399 		& .027 $\pm$ .0001 & .0399 	\\[1ex]

$Q_S$    	&.8       	&25         	&-1           	&1 		& 1 			     & 1.72 $\pm$ .003 & 2.08 		& .065 $\pm$ .0002 & .0832 		& .056 $\pm$ .0002 & .1040 		& .069 $\pm$ .0002 & .1040	\\
$Q_T$    	&.8       	&25        	&-1           	&1 		& 1 			     & 1.79 $\pm$ .003 & 2.08 		& .064 $\pm$ .0002 & .0832 		& .055 $\pm$ .0002 & .1040 		& .072 $\pm$ .0002 & .1040	\\[1ex]

$Q_S$    	&.8      	&4           	&-.4           	&1 		& 1 			     & .79 $\pm$ .002 & .9947 		& .180 $\pm$ .0008 & .2486 		& .090 $\pm$ .0005 & .3108 		& .199 $\pm$ .0007 & .3108	\\
$Q_T$    	&.8       	&4          	&-.4           	&1 		& 1 			     & .86 $\pm$ .002 & .9947 		& .176 $\pm$ .0008 & .2486 		& .099 $\pm$ .0004 & .3108 		& .215 $\pm$ .0007 & .3108	\\[1ex]
\hline
\end{tabular}}
\label{tab:ticket_4}
\end{table}